\documentclass[aap]{imsart}
\usepackage[utf8]{inputenc}
\usepackage{graphicx}
\usepackage{xcolor}
\usepackage{amsmath,amssymb,amsthm}
\usepackage{booktabs,stmaryrd,makecell}
\usepackage{hyperref}
\newtheorem{theorem}{Theorem}[section]
\newtheorem{corollary}[theorem]{Corollary}
\newtheorem{Problem}[theorem]{Problem}
\newtheorem{lemma}[theorem]{Lemma}
\newtheorem{proposition}[theorem]{Proposition}
\newtheorem{remark}[theorem]{Remark}
\newtheorem{definition}[theorem]{Definition}
\newtheorem{example}[theorem]{Example}

\newcommand{\R}{\mathbb{R}}
\newcommand{\D}{\mathbb{D}}
\newcommand{\E}{\mathbb{E}}
\newcommand{\N}{\mathbb{N}}
\newcommand{\C}{\mathbb{C}}
\newcommand{\cX}{\mathcal{X}}
\newcommand{\bbP}{\mathbb{P}}
\newcommand{\bbX}{\mathbb{X}}
\newcommand{\bX}{\mathbf{X}}
\newcommand{\bY}{\mathbf{Y}}
\newcommand{\bV}{\mathbf{V}}
\newcommand{\bA}{\mathbf{A}}
\newcommand{\bN}{\mathbf{N}}
\newcommand{\bbF}{\mathbb{F}}

\newcommand{\lin}{\mathcal{L}}
\newcommand{\cP}{\mathcal{P}}
\newcommand{\cD}{\mathcal{D}}
\newcommand{\cF}{\mathcal{F}}
\newcommand{\cH}{\mathcal{H}}
\newcommand{\cI}{\mathcal{I}}
\newcommand{\so}{\mathfrak{so}}
\newcommand{\fu}{\mathfrak{u}}
\newcommand{\fg}{\mathfrak{g}}
\newcommand{\dd}{\mathrm{d}}

\newcommand{\bK}{\mathbf{K}}

\newcommand{\Ad}{\mathrm{Ad}}
\newcommand{\var}{\text{-}\mathrm{var}}
\newcommand{\Lip}{\mathrm{Lip}}
\newcommand{\GL}{\mathrm{GL}}
\newcommand{\ESig}{\mathrm{ESig}}
\newcommand{\op}{\mathrm{op}}
\newcommand{\HS}{\mathrm{HS}}
\newcommand{\Lie}{\mathrm{Lie}}
\newcommand{\Tr}{\mathrm{Tr}}

\begin{document}

\begin{frontmatter}
\title{The $L^q$-bounds for Derivatives of Unitary Developments of Random Continuous Geometric Rough Paths}
\runtitle{Bounds for Derivatives of Unitary Developments}

\begin{aug}
\author[A]{\fnms{Chong}~\snm{Liu}\ead[label=e1]{liuchong@shanghaitech.edu.cn}}
\author[B]{\fnms{Zijiu}~\snm{Lyu}\ead[label=e2]{zijiu.lyu@berkeley.edu}}
\author[C]{\fnms{Hao}~\snm{Ni}\ead[label=e3]{h.ni@ucl.ac.uk}}
\address[A]{Institute of Mathematical Science, ShanghaiTech University\printead[presep={,\ }]{e1}}

\address[B]{Department of Industrial Engineering and Operations Research, University of California, Berkeley\printead[presep={,\ }]{e2}}

\address[C]{Department of Mathematics, University College London\printead[presep={,\ }]{e3}}
\end{aug}



\begin{abstract}
In this paper we derive an explicit formula for radial derivatives of unitary developments of random  continuous geometric rough paths of all orders and establish proper (locally) uniform $L^q$-bounds for them, which allow us to study the analyticity of their characteristic functions in a quantitative manner. Then we use these bounds to prove within a class satisfying suitable analyticity conditions, the distributions of their signatures are determined by their expected signatures, provided the latter have positive radius of convergence, which gives an alternative derivative-based proof of the same problem studied in \cite{ChevyrevLyons2016}. These $L^q$-estimates also yield explicit error bounds for Monte Carlo estimation of PCF derivatives and Taylor approximation of the PCF. Numerical experiments with Ornstein--Uhlenbeck processes illustrate increased two-sample rejection power from radial derivatives and moving-frame matrix-signature features.
\end{abstract}

\end{frontmatter}
\section{Introduction}

Let $X$ be an $\R$-valued random variable. In probability theory the following standard results are well known, see e.g. \cite{Billingsley1995}:
\begin{enumerate}
    \item The characteristic function $\phi_X(\lambda) : = \E[e^{\mathrm{i}\lambda X}], \lambda \in \R$ determines the distribution $\mu_X$ of $X$.
    \item If $X$ has finite $n$-th moment for some $n \ge 1$, then one has
    \begin{equation}\label{eq: dev bound for static rv}
        \frac{\dd^n}{\dd \lambda^n}\phi_X(\lambda) = \E\bigg[\frac{\dd^n}{\dd \lambda^n} e^{\mathrm{i}\lambda X}\bigg] = \E[(\mathrm{i}X)^ne^{\mathrm{i}\lambda X}], \quad \sup_{\lambda \in \R} \bigg\| \frac{\dd^n}{\dd \lambda^n} e^{\mathrm{i}\lambda X}\bigg\|_{L^1} \le \E[|X|^n].
    \end{equation}
    \item  If the power series of its moments generating function $\varphi_X(\lambda) : = \sum_{n = 0}^\infty \frac{\E[X^n]}{n!}\lambda^n,\lambda \in \R$ has a positive radius of convergence $r > 0$, then $\phi_X(\cdot)$ is (real) analytic such that
    \begin{equation}\label{eq: roc for static case}
         \forall \lambda \in \R, \forall |h| < r, \phi_X(\lambda + h) = \sum_{n=0}^\infty \frac{\phi_X^{(n)}(\lambda)}{n!}h^n, \quad \phi_X^{(n)}(\lambda) = \frac{\dd^n}{\dd \lambda^n}\phi_X(\lambda).
    \end{equation}

    \item  If the power series of its moments generating function $\varphi_X(\lambda) : = \sum_{n = 0}^\infty \frac{\E[X^n]}{n!}\lambda^n,\lambda \in \R$ has a positive radius of convergence $r > 0$,
    \begin{equation}\label{eq: moments and cf for static case}
        \phi_X(\cdot) \text{ and } \mu_X \text{ are uniquely determined by } (1, \E[X], \ldots, \frac{\E[X^n]}{n!}, \ldots)
    \end{equation}
 where $\mu_X$ is the distribution of $X$.
\end{enumerate}

On the other hand, it is also well known that in  signature and rough path theory one also has the notions ``moments generating function'' and ''characteristic function'' for random continuous ($\R^d$-valued) geometric rough paths\footnote{See Section \ref{sect; preliminaries} and the literature therein for these conceptions.}. More precisely, let $\bX_{[0,T]}$ be a random continuous geometric rough path defined on some time interval $[0,T]$ (where $T$ can be random), by the Lyons extension theorem (\cite[Theorem 3.1.2]{Lyons2002}) it admits a unique \textit{signature} $S(\bX)_{0,T}$, which can be expressed as an infinite sequence $(1, S(\bX)^1_{0,T}, \ldots, S(\bX)^n_{0,T}, \ldots) \in \prod_{n = 0}^\infty (\R^d)^{ \otimes n}$. If $S(\bX)^n_{0,T}$ is integrable for all $n \ge 0$, then the sequence $\ESig(\bX_{[0,T]}) : = (1, \E[S(\bX)^1_{0,T}], \ldots, \E[S(\bX)^n_{0,T}],\ldots)$ is called the \textit{expected signature} of $\bX_{[0,T]}$. Further, for any linear operator $M \in \bigsqcup_{k = 1}^\infty \lin(\R^d, \fu(k))$ with $\fu(k) \subset \C^{k \times k}$ denoting the unitary Lie algebra of order $k$, the unitary matrix-valued series $\Phi_{\bX_{[0,T]}}(M): = \sum_{n = 0}^\infty M^{ \otimes n}(S(\bX)^n_{0,T}) \in U(k)$, which is actually the unique solution $Y^M_T$ (evaluated at the terminal time $T$) to the following linear differential equation
\begin{equation}\label{eq: RDE in introduction}
   \dd Y^M_t = Y^M_t M(\dd \bX_t),\quad Y^M_0 = I_k,
\end{equation}
is called the unitary development of $\bX_{[0,T]}$ under $M$, and the expectation of unitary development $\phi_{\bX_{[0,T]}}(M) = \E[\Phi_{\bX_{[0,T]}}(M)]$ or $\E[Y^M_T]$ is called the \textit{path characteristic function (PCF)} of $\bX_{[0,T]}$ (evaluated at $M$). It is easy to see that if we consider an $\R$-valued random variable $X$ as a linear path $X_{[0,1]} (t) : = tX, t \in [0,1]$, then $\ESig(X_{[0,1]}) = (1, \E[X], \ldots, \frac{\E[X^n]}{n!}, \ldots)$. By choosing $M \in \lin(\R, \fu(1))$ (which means that $M(x) = (\lambda x)\mathrm{i}$ for some $\lambda \in \R$) one has $\phi_{X_{[0,1]}}(M) = \phi_X(\lambda) = \E[e^{\mathrm{i}\lambda X}]$.

Since the expected signature $\ESig(\bX_{[0,T]})$ and path characteristic function $\phi_{\bX_{[0,T]}}(M)$ are natural extensions of moments generating function and characteristic functions of $\R$-valued random variables $X$ to geometric rough path-valued random variables $\bX_{[0,T]}$, one may ask whether the above claims 1--4 also holds true for random geometric rough path $\bX_{[0,T]}$ if we replace moments generating function $\varphi_X$ by expected signature $\ESig(\bX_{[0,T]})$ and characteristic function $\phi_X(\lambda)$ by PCF $\phi_{\bX_{[0,T]}}(M)$. In the seminal paper \cite{ChevyrevLyons2016} the authors first introduced the notion of PCF and proved that the distribution $\mu_{S(\bX)_{0,T}}$ of signature is uniquely determined by its PCF $\phi_{\bX_{[0,T]}}(\cdot)$. Such characteristicness of expected signatures and PCF was further studied in \cite{PCFGAN2023} and \cite{RestrictedPCF2024},
which allow people to apply them as computable and powerful discriminators to distinguish the distributions of (signatures of) random geometric rough paths, see e.g. \cite{Oberhauser2019}, \cite{ni2021sig} \cite{liao2024sig}, \cite{ChevyrevOberhauser2022}, \cite{Bonnier2023},  and \cite{TaoLiuNi2024} for concrete applications in classifications of time series and generative modelling of stochastic processes which are now very popular topics in machine learning and mathematical finance.

This paper mainly aims to study the following problem, which can be viewed as the counterpart of \eqref{eq: dev bound for static rv} and \eqref{eq: roc for static case}  for random continuous geometric rough paths:

\begin{Problem}\label{thm: new main question}
Let $\bX_{[0,T]}$ be a random continuous geometric rough path defined on some time interval $[0,T]$ (which can also be random).   For a given $\phi_{\bX_{[0,T]}}(M),  M \in \bigsqcup_{k = 1}^\infty \lin(\R^d, \fu(k))$ and $q \ge 1$ (mainly for $q=1$ or $q = 2$), can we find a sequence of positive numbers $(B_q(n))_{n \ge 0}$ and a positive number $r>0$ which can depend on $q, \bX_{[0,T]}$ and $M$, but not on $\lambda$, such that
\begin{equation}\label{eq: upper dev bound in intro}
   \sup_{\lambda \in \R}  \frac{\bigg\|\frac{\dd^n}{\dd \lambda^n}Y^{\lambda M}_T\bigg\|_{L^q}}{n!} \le B_q(n), \quad \forall |h| < r, \quad \sum_{n=0}^\infty B_q(n)|h|^n < \infty.
\end{equation}
\end{Problem}

Clearly, if we could find such a sequence $(B_q(n))_{n \ge 0}$ and a positive radius $r>0$ which satisfy \eqref{eq: upper dev bound in intro}, then the $n$-th radial derivative $\phi^{(n)}_{\bX_{[0,T]}}(\lambda M) := \frac{\dd^n}{\dd \lambda^n}\phi_{\bX_{[0,T]}}(\lambda M)$ coincides with the expectation of the $n$-th radial derivative of the unitary development $\E[\frac{\dd^n}{\dd \lambda^n}Y^{\lambda M}_T]$ for all $n \ge 0$; and the PCF $\phi_{\bX_{[0,T]}}(\lambda M)$ is (real) analytic in $\lambda$ such that for any base point $\lambda_0$, the Taylor series $\sum_{n=0}^\infty \frac{\phi^{(n)}_{\bX_{[0,T]}}(\lambda_0 M)}{n!}(\lambda - \lambda_0)^n$ converges for all $\lambda$ such that $|\lambda - \lambda_0|<r$.

Unlike the classical case for $\R$-valued random variables which only amounts to use the simple fact $\frac{\dd^n}{\dd \lambda^n}e^{\mathrm{i}\lambda x} = (\mathrm{i}x)^n e^{\mathrm{i}\lambda x}$, finding such a proper upper bounds $(B_q(n))_{n \ge 0}$ and radius $r$ in \eqref{eq: upper dev bound in intro} becomes highly non-trivial when $M \in \lin(\R^d, \fu(k))$ for $k \ge 2$ because of the non-commutativity of the matrix group $U(k)$. Actually, for a general continuous geometric rough path the globally uniform bound $B_q(n)$ in \eqref{eq: upper dev bound in intro} may not exist, see a counterexample in Remark \ref{rem: a counterexample}.
One of the main contributions of this paper is to find an explicit formula for all high-order radial derivatives $\frac{\dd^n}{\dd \lambda^n} Y^{\lambda M}_T$ and then establish explicit bounds $(B_q(n))_{n \ge 0}$ for the $L^q$-bounds of $\frac{\dd^n}{\dd \lambda^n}Y^{\lambda M}_T$ and a radius $r>0$ such that $\lambda \mapsto \phi_{\bX_{[0,T]}}(\lambda M)$ is analytic in $(\lambda - r, \lambda + r)$ globally or locally uniformly in $\lambda$
as in \eqref{eq: upper dev bound in intro}, for many random rough paths with good probability structures,  see the table \ref{table: intro} for a brief summary of possible choices of $(B_q(n))_{n \ge 0}$ and $r$ related to various examples. Note that all constants $A,B$ and $C$ in this table are independent of $\lambda$, and for general random geometric $p$-rough paths with $p>2$ in the second row, $(B_q(n))_{n \ge 0}$ and $r$ are taken locally uniformly over bounded subset $\bK \subset \R$ and in other cases they are taken uniformly over $\R$.

\begin{table}\label{table: intro}
\centering
\caption{Choices for $(B_q(n))_{n \ge 0}$ and $r$}
\begin{tabular}{|c|c|c|c|}
   \hline
   Process $\bX_{[0,T]}$ & Conditions & $B_q(n)$ & $r$ \\
   \hline
   bounded variation & $\exists R>0, \E[e^{R \|X\|_{1\var}}]<\infty$ & $B_1(n) = \frac{C_n}{n!}$,\eqref{eq: bound for dev of smooth paths}& $\infty$ \\
   \hline
   $p > 2$& $\exists \eta>0, \E[e^{\eta \bN}]<\infty$ & $B_q(n) = Cr^{-n},\eqref{eq: upper bounds for greedy seuqence case}$ & $r(\frac{\eta}{q},M,\bK)$\\
   \hline
   It\^o semimartingale &  bounded drift and volatility & $B_2(n) = \frac{A^n}{\sqrt{n!}}$, \eqref{eq: bound for matrix signature of Ito martingales}& $\infty$ \\
   \hline
   OU process & $X_0$ indep. Gaussian & $B_2(n) = \frac{A B^n}{\sqrt{n!}}$, \eqref{eq: bounds for OU} & $\infty$ \\
   \hline
   stopped Brownian motion & $T = \tau_\D$, $\D$ a regular domain & $B_1(n) = Ar^{-n}$,\eqref{eq: bound on the high order derivative for stopped BM} & $r(M,\D)$ \\
   \hline
   \makecell{stopped OU process \\
   $\dd X_t = K(m - X_t)\dd t + \Sigma \dd W_t$} & \makecell{$T = \tau, \exists \eta>0, \E[e^{\eta \tau}]<\infty$ \\
   $X_0$ indep. Gaussian \\
   $-K$ is a Hurwitz matrix} & $B_2(n) = A r^{-n}$, \eqref{eq: bound for stopped OU} & $r(M, X, \eta)$\\
   \hline
   \makecell{stopped semimartingale\\
   $X_t = X_0 + N_t + B_t$} & \makecell{$T = \tau, \exists \alpha, \gamma>0$, \\
   $\E[e^{\alpha \|B\|_{1\var,[0,\tau]}+ \gamma \llbracket N \rrbracket_\tau}]<\infty$}  & $B_2(n) = A r^{-n}$, \eqref{eq: bound for stopped semimart} & $r(M,\alpha,\gamma)$\\
   \hline
\end{tabular}
\end{table}

Moreover, solving the above problem \ref{thm: new main question} can help us to answer the following question, which can be viewed as the counterpart of the result \eqref{eq: moments and cf for static case} in the random geometric rough paths setup.

\begin{Problem}\label{thm: main question}
Let $\bX_{[0,T]}$ be a random continuous geometric rough path defined on some time interval $[0,T]$ (which can also be random). Suppose that the power series of its expected signature $\ESig(\bX_{[0,T]})$:
  $$
  \sum_{n = 0}^\infty \|\E[S(\bX)^n_{0,T}]\|\lambda^n,\quad \lambda \in \R
  $$
has a positive (but only finite) radius of convergence $r_2(\bX_{[0,T]}) > 0$, then under which additional conditions on $\bX_{[0,T]}$ its PCF $\phi_{\bX_{[0,T]}}(M),  M \in \bigsqcup_{k = 1}^\infty \lin(\R^d, \fu(k))$ and the distribution of the random signature $S(\bX)_{0,T}$ can be determined by $\ESig(\bX_{[0,T]})$?
\end{Problem}

Indeed, if a random geometric rough path $\bX_{[0,T]}$  possesses bounds $(B_q(n))_{n \ge 0}$ and radius $r$, and meanwhile its expected signature $\ESig(\bX_{[0,T]})$ has a positive radius of convergence $r_2(\bX_{[0,T]})>0$, then its PCF $\phi_{\bX_{[0,T]}}(M)$ and the distribution of the random signature $S(\bX)_{0,T}$ can be determined by $\ESig(\bX_{[0,T]})$: since $r_2(\bX_{[0,T]}) > 0$, one has $\phi_{\bX_{[0,T]}}(\lambda M) = \sum_{n = 0}^\infty M^{ \otimes n}(\E[S(\bX)^n_{0,T}])\lambda^n$ whenever $|\lambda| < \frac{r_2(\bX_{[0,T]})}{2\sqrt{d}\|M\|_{\op}}$, which implies that $\phi_{\bX_{[0,T]}}(\lambda M)$ is determined by $\ESig(\bX_{[0,T]})$ in a neighborhood of $0$. On the other hand, we have noticed that if \eqref{eq: upper dev bound in intro} holds, then the PCF $\phi_{\bX_{[0,T]}}(\lambda M)$ is (real) analytic in  $(\lambda_0 - r, \lambda_0 + r)$ for any base point $\lambda_0$ with $r$ independent of $\lambda_0$. As a consequence, a value propagation argument will provide that the value of $\phi_{\bX_{[0,T]}}(\lambda M)$ for any $\lambda \in \R$ is determined by the values of $\phi_{\bX_{[0,T]}}(\lambda M)$ in a neighborhood of $0$, which is in turn determined by $\ESig(\bX_{[0,T]})$.

In fact, in \cite[Section 6.2]{ChevyrevLyons2016} the authors has studied Problem \ref{thm: main question} via a very functional-analytical approach: they managed to find a sufficient condition (see \cite[Theorem 6.13]{ChevyrevLyons2016}) on $\bX_{[0,T]}$, which can be usually verified by studying the so-called greedy sequences introduced in \cite{CassLittererLyons2013}, such that the function-valued Bochner integral $\E[Y^{zM}_T] = \phi_{\bX_{[0,T]}}(zM)$ of the entire function $z \in \C \mapsto Y^{zM}_T$ is well-defined and also remains holomorphic on a complex neighbourhood of each $\lambda \in \R$, and then the identity theorem for analytic functions applies to deduce the determinacy of $\phi_{\bX_{[0,T]}}(\lambda M)$ and $\mu_{S(\bX)_{0,T}}$ by $\ESig(\bX_{[0,T]})$. In their framework no explicit forms of derivatives of unitary developments/PCF were directly involved. Therefore another contribution of this paper is to give an alternative derivative-based perspective (in terms of  \eqref{eq: upper dev bound in intro}) for studying Problem \ref{thm: main question}. Our approach can handle many important examples (see table \ref{table: intro}) under directly verifiable stochastic assumptions.

The estimates in \eqref{eq: upper dev bound in intro} also yield explicit Monte Carlo and Taylor approximation error bounds for the PCF, which we record in Lemma~\ref{lemma: error bounds} in Appendix~\ref{sect: numerical error bounds}. These results motivate augmenting the PCF $\phi_{\bX_{[0,T]}}(\cdot)$ with its derivatives as complementary features for characterising probability laws on path space. Whereas the PCF captures distributional information through function value, its derivatives encode local sensitivity to perturbations of the development parameters and may reveal discrepancies that are less visible to the PCF alone. This yields a richer and theoretically controlled representation for stochastic processes, with potential applications to two-sample testing, distribution-shift detection, and the training and evaluation of generative models for synthetic time series, naturally extending the PCF-GAN framework \cite{PCFGAN2023}. As a concrete application, we test these representations on two Ornstein--Uhlenbeck laws with identical one-time marginals and different temporal dependence. In this experiment, second- and third-order matrix-signature features which are closed related to the second and third-order radial derivatives of PCF, attain higher mean test power than the PCF baselines.

The paper is organized as follows: We collect some preliminaries on Lie algebra/group theory and signature and rough path theory, in particular, definitions on expected signature and path characteristic function, in Section \ref{sect; preliminaries}, which will be used later. In Section \ref{sect: high-order derivatives} we derive an explicit formula for high-order radial derivatives of unitary developments and establish feasible $L^q$-bounds (mainly for $q = 1$ or $q = 2$) of these derivatives for stochastic processes with smooth sample paths in Subsection \ref{subsect: smooth}, for random geometric $p$-rough paths with $p > 2$ whose random number $\bN(\bX_{[0,T]})$ associated with greedy sequences has an exponential tail in Subsection \ref{subsect: greedy sequence}, for continuous semimartingales on deterministic time horizon (including It\^o semimartingales, Ornstein-Uhlenbeck processes) in Subsection \ref{subsect: semimartingale on deterministic time} and for stopped continuous semimartingales (including Brownian motion stopped by the first exit time of some regular domain and stopped Ornstein-Uhlenbeck processes with exponentially integrable stooping time) in Subsection \ref{subsect: stopped BM}. In Section~\ref{sect: numerical experiments}, we first introduce feature distances and a two-sample testing framework for general stochastic processes, and then compare PCF, radial-derivative and moving-frame matrix-signature features for stationary Ornstein--Uhlenbeck processes, using linear maps into $\fu(k)$ and its real subspaces.Appendices~\ref{sect: matrix stochastic calculus} and \ref{sect: proof of auxiliary results} collect auxiliary results and proofs. Appendix~\ref{sect: numerical error bounds} contains the Monte Carlo and Taylor approximation error bounds for the numerical computation of radial derivatives of PCF.

\textbf{Notations}: We equip $V = \R^d$ with the usual Euclidean $\ell^2$-norm $|\cdot|$, equip $\C^{k \times k}$ (the space of $k \times k$-complex matrices) with the Hilbert-Schmidt norm $\|\cdot\|_{\HS}$. That is, for $A = (A^{ij})_{i,j = 1}^k \in \C^{k \times k}$, $\|A\|_{\HS} = (\sum_{i,j = 1}^k |A^{ij}|^2)^{\frac{1}{2}} = \sqrt{\text{Tr}(A A^\dagger)}$, where $A^\dagger$ is the conjugate transpose of $A$ and $\text{Tr}$ denotes the usual trace operator. Note that $(\C^{k \times k}, \|\cdot\|_{\HS})$ is a Hilbert space and also a Banach algebra in the sense that $\|AB\|_{\HS} \le \|A\|_{\HS}\|B\|_{\HS}$ for all $A,B \in \C^{k \times k}$. For a matrix $A \in \C^{k \times k}$, we use both $A^{ij}$ and $A_{ij}$ for its $(i,j)$-entry, depending on the notation context.

For a random variable $X$ defined on a probability space $(\Omega,\cF,\bbP)$ taking values in some topological space $\cX$ equipped with its Borel $\sigma$-algebra, we use $\mu_X$ to denote the distribution of $X$ on $\cX$, that is, $\mu_X = X_\sharp \bbP$ is the push-forward probability measure of $\bbP$ under $X:\Omega \to \cX$.

Unless otherwise stated, every path $t \mapsto X_t$ is defined on $[0,T]$ for some finite number $T > 0$, and we will sometimes write $X_{[0,T]}$ or simply $X$ for such a path. Note that if $X_{[0,T]}(\omega)$ is a random path, the time $T = T(\omega)$ can also be random. For $p \ge 1$, we use $C^{p\var}([0,T],W)$ to denote the space of all continuous $p$-variation paths in a Banach space $W$, that is, $Y \in C^{p\var}([0,T],W)$ iff $\|Y\|_{p\var;[0,T]}^p = \sup_{\cD \subset [0,T]}\sum_{t_j \in \cD} \|Y_{t_{j + 1}} - Y_{t_j}\|^p < \infty$, where the supremum is taken over all possible dissections $\cD$ of $[0,T]$. The notation $\|Y\|_{p\var;[s,t]}$ for $[s,t] \subset [0,T]$ is then self-explanatory.

For $n \ge 1$ and $0 \le s \le t \le T$, let $\Delta_{(s,t)}^n = \{(t_1,\ldots,t_n): s \le t_1 < \ldots < t_n \le t\}$ be the $n$-dimensional simplex over $[s,t]$ and write $\Delta_{t}^n = \Delta_{(0,t)}^n$ for $t \in [0,T]$. For rough-path increments we use the closed simplex $\overline{\Delta_T^2} : = \{(s,t): 0 \le s \le t \le T\}$.

For two Banach spaces $V$ and $W$, we use $\lin(V,W)$ to denote the space of all continuous linear mappings from $V$ to $W$. If $M \in \lin(V,W)$, then $\|M\|_{\op}$ denotes its operator norm, i.e., $\|M\|_{\op} = \sup_{v \in V,\|v\| = 1} \|M(v)\|$. We will always equip $\lin(V,W)$ with this operator norm. Moreover, we use $W^\prime$ to denote the dual of $W$, i.e., the space of all continuous $\R$-valued linear functionals on $W$.

For two reals $a,b \in \R$, we write $a \wedge b = \min\{a,b\}$ and $a \vee b = \max\{a,b\}$.

\section{Preliminaries on the rough path theory}\label{sect; preliminaries}

\subsection{Basic notions in rough path theory}\label{subsect: basic notions in rough path theory}
For notational simplicity, in the present paper we also use $V$ for $\R^d$. We use $T(V) = \oplus_{n = 0}^\infty V^{ \otimes n}$ to denote the tensor algebra over $V$ and use $T((V)) = \prod_{n = 0}^\infty V^{ \otimes n}$ for the completed tensor algebra. For $x \in T((V))$ we write $x = (x^0, x^1, \ldots, x^n, \ldots)$ with $x^n \in V^{ \otimes n}$.
For each $N \in \N$, $T^N(V) = \oplus_{n = 0}^N V^{ \otimes n}$ denotes the truncated tensor algebra at level $N$. The spaces $T(V), T((V)),$ and $T^N(V)$ are algebras with respect to the tensor product $ \otimes $ as the multiplication operation, see \cite[Chapter 7]{Friz2010}.

For each $n \ge 1$, we endow $V^{ \otimes n}$ the projective tensor norm $\|\cdot\|_{ \otimes n}$ induced by the norm $|\cdot|$ on $V$. Given a vector $x \in T((V))$, we define $R(x)$ to be the radius of convergence for the series $\sum_{n = 0}^\infty\|x^n\|_{ \otimes n}\lambda^n$, and simply call it the radius of convergence of $x$. We set $R(x) = 0$ if the above series does not converge for any $\lambda > 0$.

For $p \ge 1$ and $T > 0$, the space of $p$-rough paths $\Omega_p(V)$ is the collection of all continuous maps $\bX = \bX_{[0,T]}: \overline{\Delta_T^2} \to T^{\lfloor p \rfloor}(V)$ such that there exists a control function $w(\cdot,\cdot)$ in the sense of \cite[Definition 1.6]{Friz2010} and
\begin{itemize}
    \item $\bX^0_{s,t} = 1$ and $\bX_{s,u} \otimes \bX_{u,t} = \bX_{s,t}$ for $0 \le s \le u \le t \le T$.
    \item for all $(s,t) \in \overline{\Delta_T^2}$,
    \begin{equation}\label{eq: rough path decay}
        \sup_{1 \le n \le \lfloor p \rfloor} \bigg( (n/p)!\gamma_p\|\bX^n_{s,t}\|_{ \otimes n}
     \bigg)^{\frac{p}{n}} \le w(s,t)
    \end{equation}
for some constant $\gamma_p$ only depending on $p$.
\end{itemize}
We can also view $\bX \in \Omega_p(V)$ as $t \mapsto \bX_{0,t}$, and write $\bX_t = \bX_{0,t}$ because $\bX_{s,t} = \bX_{0,s}^{ - 1} \otimes \bX_{0,t}$. Given $\bX \in \Omega_p(V)$, its $p$-variation norm is defined as
$$
\|\bX\|^p_{p\var;[0,T]} : = \sum_{n = 1}^{\lfloor p \rfloor} \sup_{\cD \subset [0,T]}\sum_{t_j \in \cD}
 \|\bX^n_{t_j, t_{j + 1}}\|_{ \otimes n}^{\frac{p}{n}}.
$$
This $p$-variation norm induces the $p$-variation metric/topology on $\Omega_p(V)$:
$$
\|\bX;\bY\|^p_{p\var;[0,T]} : = \sum_{n = 1}^{\lfloor p \rfloor} \sup_{\cD \subset [0,T]}\sum_{t_j \in \cD}
 \|\bX^n_{t_j, t_{j + 1}} - \bY^n_{t_j, t_{j + 1}}\|_{ \otimes n}^{\frac{p}{n}}
$$
for $\bX,\bY \in \Omega_p(V)$.

By the Lyons extension theorem, see \cite[Theorem 3.1.2]{Lyons2002}, for each $\bX \in \Omega_p(V)$, there exists a unique lift $S(\bX): \overline{\Delta_T^2} \to T((V))$ such that $S(\bX)^0_{s,t} = 1$ and $S(\bX)_{s,u} \otimes S(\bX)_{u,t} = S(\bX)_{s,t}$ for $0 \le s \le u \le t \le T$ and \eqref{eq: rough path decay} is satisfied with the same control function $w$ and with $\sup_{1 \le n \le \lfloor p \rfloor}$ replaced by $\sup_{n \ge 1}$. We will call this lift $S(\bX)$ the \textit{signature path} of $\bX$ and call $S(\bX)_{0,T}$ the \textit{signature} of $\bX$. A remarkable property of signature is the so-called Chen's relation (see \cite[Theorem 2.9]{Lyons2007}). Below we will need a special case of Chen's relation, namely for any partition $0 = \tau_0 < \tau_1 < \ldots < \tau_N < \tau_{N + 1} = T$ of $[0,T]$ and any $\bX_{[0,T]} \in \Omega_p(V)$, one has $ S(\bX)_{0,T} = \prod_{j = 0}^N S(\bX_{[\tau_j, \tau_{j + 1}]}) = S(\bX_{[0,\tau_1]}) \otimes \ldots \otimes S(\bX_{[\tau_N,T]}) $.

In fact, the signature path $S(\bX)$ takes values in a proper subspace of $T((V))$. Following \cite{ChevyrevLyons2016} we use $E$ to denote the following linear subspace of $T((V))$:
\begin{equation}\label{eq: def of E}
    E = \bigg\{x \in T((V)):\forall \lambda \ge 0,  \sum_{n = 0}^\infty \lambda^n \|x^n\|_{ \otimes n} < \infty \bigg\}.
\end{equation}
In view of \cite{ChevyrevLyons2016}, the space $E$ equipped with the seminorms $\{\exp(m|\cdot|): m \in \N\}$ is a complete locally convex algebra, where $\exp(m|\cdot|)(x) : = \sum_{n = 0}^\infty m^n \|x^n\|_{ \otimes n}$. Clearly, for all $x \in E$ we have $R(x) = \infty$. Thanks to the factorial decay condition \eqref{eq: rough path decay}, we have $S(\bX)_{s,t} \in E$ for all $(s,t) \in \overline{\Delta_T^2}$. Moreover, it is well known that the signature evaluation mapping $\cI^p_t: \Omega_p(V) \to E, \bX \mapsto S(\bX)_{0,t}$ is (locally Lipschitz) continuous when $\Omega_p(V)$ is endowed with the $p$-variation topology and $E$ equipped with the seminorms $\{\exp(m|\cdot|): m \in \N\}$ as above, see \cite[Corollary 5.5]{ChevyrevLyons2016}, \cite[Theorem 3.1.3]{Lyons2002} or \cite[Theorem 9.10]{Friz2010}.

If $X \in \Omega_1(V)$, then it is clear that the condition \eqref{eq: rough path decay} ensures that $X \in C^{1\var}([0,T],V)$ is of bounded variation. In this case it is well known that for each $n \ge 2$,
\begin{align}\label{eq: signature of BV}
    S(X)_{s,t}^n & = \int_{\Delta_{(s,t)}^n} \dd X_{s_1} \otimes \cdots \otimes \dd X_{s_n}
     = \sum_{i_1,\ldots,i_n = 1}^d \int_{\Delta_{(s,t)}^n} \dd X^{i_1}_{s_1} \cdots \dd X^{i_n}_{s_n} e_{i_1} \otimes \cdots \otimes e_{i_n} \in V^{ \otimes n},
\end{align}
where $e_1,\ldots,e_d$ are canonical basis of $V = \R^d$.  We now define the space of geometric $p$-rough paths $G\Omega_p(V)$ as the closure of $\Omega_1(V)$ in $\Omega_p(V)$ in the $p$-variation topology.

\begin{definition}\label{def: geometric rough paths}
For $p \ge 1$,  the space of geometric $p$-rough paths $G\Omega_p(V)$ is the closure of $\Omega_1(V)$ in $\Omega_p(V)$ in the $p$-variation topology through the mapping $S_{\lfloor p \rfloor} = (1, S^1, \ldots, S^{\lfloor p \rfloor})$. That is, $\bX \in G\Omega_p(V)$ iff $\bX \in \Omega_p(V)$ and there exists a sequence of bounded variation paths $X(n) \in \Omega_1(V), n \ge 1$ such that $\lim_{n \to \infty} \|S_{\lfloor p \rfloor}(X(n)); \bX \|_{p\var;[0,T]} = 0$.
\end{definition}
We endow $G\Omega_p(V)$ with the subspace topology from $\Omega_p(V)$. Thanks to the continuity of signature lift $S(\cdot)$ on $\Omega_p(V)$, if a sequence of bounded variation paths $X(n) \in \Omega_1(V), n \ge 1$ satisfies that $\lim_{n \to \infty} \|S_{\lfloor p \rfloor}(X(n)); \bX \|_{p\var;[0,T]} = 0$, then $\lim_{n \to \infty}S(X(n))_{s,t} = S(\bX)_{s,t}$ in $E$ for all $(s,t) \in \overline{\Delta_T^2}$.

\subsection{Preliminaries on Rough Differential Equation}\label{subsect: RDE}
In this subsection we will collect some basic notions and results on rough differential equations (RDE) driven by geometric $p$-rough paths which will be used later. Here we will adopt the approach from \cite{Friz2010}, other equivalent methodologies on solving RDE can be found in e.g. \cite{Lyons1998}, \cite{Davie2008}, \cite{Gubinelli2004}, or \cite{Friz2014}.

As before, let $V = \R^d$, and let $W$ be another Banach space (in the present paper, we usually choose $W = \R^{k \times k}$). Further, let $\bV: W \to \lin(V,W)$ be a given vector field, and $\bX \in G\Omega_p(V)$ be a geometric $p$-rough path. Following \cite[Definition 10.17]{Friz2010}, we give the definition of solution to the rough differential equation
\begin{equation}\label{eq: RDE}
    dY_t = \bV(Y_t)\dd \bX_t, \quad Y_0 = y_0
\end{equation}
driven by $\bX$ along $\bV$ as follows:
\begin{definition}\label{def: RDE solution}
We say that a path $Y \in C^{p\var}([0,T],W)$ is a solution to the RDE $dY_t = \bV(Y_t)d\bX_t$ with initial value $Y_0 = y_0 \in W$, if the following condition hold: there exists a sequence of bounded variation paths $X(n) \in \Omega_1(V), n \ge 1$ such that $\lim_{n \to \infty} \|S_{\lfloor p \rfloor}(X(n)); \bX \|_{p\var;[0,T]} = 0$ and
     for each $n$, the usual ordinary differential equation $$dY(n)_t = \bV(Y(n)_t)\dd X(n)_t, \quad Y(n)_0 = y_0$$ driven by $X(n)$ admits a solution $Y(n)$, so that $Y(n) \to Y$ uniformly on $[0,T]$.
  \end{definition}

If the vector field $\bV$ is regular enough, then the RDE $dY_t = \bV(Y_t)d\bX_t, Y_0 = y_0$ is well-posed, and it admits a continuous flow in $\bV$, $y_0$ and $\bX$. We will use $\Lip^\gamma$ (with $\gamma > 0$) to denote the $\gamma$-Lipschitz function space, see \cite[Definition 10.2]{Friz2010}.

\begin{theorem}\label{thm: well posedness of RDE}
 If $\bX \in G\Omega_p(V)$ and $\bV \in \Lip^\gamma(W, \lin(V,W))$ with $\gamma > p$, then for any $y_0 \in W$, the RDE \eqref{eq: RDE} admits a unique solution $Y = Y^{(\bX,\bV,y_0)}$ in the sense of Definition \ref{def: RDE solution}, whose $p$-variation satisfies
$\|Y\|_{p\var} \le C(\|\bV\|_{\Lip^\gamma}\|\bX\|_{p\var} \vee \|\bV\|_{\Lip^\gamma}^p\|\bX\|_{p\var}^p)$
for a constant $C = C(p,\gamma)$ only depending on $p$ and $\gamma$. Moreover, for $\bX_i \in G\Omega_p(V)$, $\bV_i \in \Lip^\gamma(W, \lin(V,W))$ with $\gamma > p$ and $y^i_{0} \in W$, let $Y^i$ denote the solution to RDE \eqref{eq: RDE} according to $\bX_i,\bV_i,$ and $y^i_0$ for $i = 1,2$ respectively, then one has
     \begin{equation*}\label{eq: continuity of RDE solution 1}
         \|Y^{1} - Y^{2}\|_{p\var} \le C_1C_2 (\|\bX_1;\bX_2\|_{p\var} + \|\bV_1 - \bV_2\|_{\Lip^\gamma} + \|y^1_0 - y^2_0\|),
     \end{equation*}
where $C_1 = C_1(\gamma,p)$ depends on $\gamma$ and $p$, and $C_2$ can be chosen uniformly when $\max_{i = 1,2}\{\|\bX_i\|_{p\var},\|\bV_i\|_{\Lip^\gamma}\}$ is bounded.

                                    \end{theorem}
For a concrete proof of the above theorem, see \cite[Theorems 10.14, 10.26]{Friz2010}. Existence, uniqueness and local continuity also hold when $\bV$ is linear. See \cite[Theorem 10.53]{Friz2010} or \cite[Theorem 1]{linearRDE2010} for these estimates.

\subsection{Basic notions on matrix Lie groups and their Lie algebras}\label{subsect: Lie group}

Given a $k \ge 1$, let $\GL(k,\C) \subset \C^{k \times k}$ be the general linear group of all $k \times k$ invertible complex matrices, whose group operation is the matrix multiplication. A matrix Lie group $G$ is a closed subgroup of $\GL(k,\C)$ (for some $k \ge 1$), where $\GL(k,\C)$ is equipped with the topology induced by the Hilbert-Schmidt norm on $\C^{k \times k}$.

Given a matrix Lie group $G \subset \GL(k,\C)$, its Lie algebra $\fg$ is the set of all matrices $A \in \C^{k \times k}$ such that $e^{tA} = \sum_{m = 0}^\infty t^m\frac{A^m}{m!}$ belongs to $G$ for all $t \in \R$. It can be shown that $\fg$ is an $\R$-linear subspace of $\C^{k \times k}$, and equipped with the Lie bracket $[A,B]_\Lie = AB - BA$ it becomes a Lie algebra, see e.g. \cite[Sect. 3.3]{Hall2015}.

The matrix Lie groups and their Lie algebras used in this paper are the following:  the Lie group of $k$-order unitary matrices $U(k) : = \{A \in \C^{k \times k}: AA^\dagger = I_k \}$ (where $I_k$ is the identity matrix in $\C^{k \times k}$), its Lie algebra is $\fu(k) : = \{A \in \C^{k \times k}: A + A^\dagger = 0 \}$ consisting of anti-Hermitian matrices. The matrix Lie group of $k$-order orthogonal matrices $O(k) : = \{A \in \R^{k \times k}: AA^\dagger = I_k \}$ with its Lie algebra $\so(k) : = \{A \in \R^{k \times k}: A + A^\dagger = 0\}$ consisting of anti-symmetric matrices.

If $\fg$ is the Lie algebra of some matrix Lie group $G$, then
for every $A \in G$ we can define a linear map
$$
\Ad_A:\fg \to \fg, \quad B \mapsto ABA^{ - 1}.
$$
Clearly this $\Ad_A$ is invertible and therefore belongs to $\GL(\fg)$, the space of all automorphisms on $\fg$. The map $\Ad: G \to \GL(\fg), A \mapsto \Ad_A$ is called an Adjoint map.
For more details on matrix Lie groups, Lie algebras and Adjoint maps, we refer readers to \cite[Sect. 3]{Hall2015}.

\subsection{Expected Signatures and Characteristic Functions of random geometric rough paths}\label{subsect: PCF}

Most stuff of this subsection is taken from \cite{ChevyrevLyons2016}. Let $(\bA,\|\cdot\|_{\bA})$ be a unital Banach algebra. Due to the universal property of the tensor algebra, every linear operator $M \in \lin(V,\bA)$ admits a unique unital algebra homomorphism $\widetilde M:T(V) \to \bA$ extending $M$. For $x = (x^0,x^1,\ldots,x^n,\ldots) \in T((V))$ for which the following series converges absolutely, we use the same notation and set
\begin{equation*}
    \widetilde M(x) = \sum_{n = 0}^\infty M^{ \otimes n}(x^n)
\end{equation*}
where $M^{ \otimes 0}: \R \to \bA$ maps $1$ to the unit element $e_{\bA} \in \bA$, and for $n \ge 1$ and $z_1 \otimes \ldots \otimes z_n \in V^{ \otimes n}$, $M^{ \otimes n}(z_1 \otimes \ldots \otimes z_n) = M(z_1)\ldots M(z_n)$. We define $R(\widetilde M(x))$ to be the radius of convergence for the series $\sum_{n = 0}^\infty \|M^{ \otimes n}(x^n)\|_{\bA}\lambda^n$ and call it the radius of convergence of $\widetilde M(x)$. Since for all $n \ge 1$ it obviously holds that $\|M^{ \otimes n}(x^n)\|_{\bA} \le \|M\|_{\op}^n\|x^n\|_{ \otimes n}$, for $M \ne 0$ we have the following relation between $R(\widetilde M(x))$ and $R(x)$:
\begin{equation}\label{eq: ROC of M and ROC of x}
    R(\widetilde M(x)) \geq \frac{R(x)}{\|M\|_{\op}}.
\end{equation}
Recall the space $E \subset T((V))$ defined as in \eqref{eq: def of E}. From \cite[Sect. 2]{ChevyrevLyons2016} we know for every linear operator $M \in \lin(V,\bA)$, its extension $\widetilde M: E \to \bA$ restricted to $E$ is continuous. Also note that since $R(x) = \infty$ for all $x \in E$, the relation \eqref{eq: ROC of M and ROC of x} guarantees $R(\widetilde{M}(x)) = \infty$ for all $x \in E$. If $M = 0$, the series is constant and its radius of convergence is infinite.

Also, as noticed in \cite[Sect. 4]{ChevyrevLyons2016}, for any $k$ if we take $\bA = \C^{k \times k}$ and consider $M \in \lin(V, \fu(k))$, then for any $\R^d$-valued geometric $p$-rough path $\bX \in G\Omega_p(V)$, it holds that $\widetilde M(S(\bX)_t) \in U(k)$ for all $t \in [0,T]$. Similarly, if $M \in \lin(V,\so(k))$, then $\widetilde M(S(\bX)_t) \in O(k)$. Also following \cite{PCFGAN2023}, given an $\bX \in G\Omega_p(V)$  we define
\begin{equation}\label{eq: unitary dev}
   \Phi_{\bX_{[0,T]}}: \bigsqcup_{k = 1}^\infty \lin(\R^d, \fu(k)) \to \bigsqcup_{k = 1}^\infty \C^{k \times k}, \quad \Phi_{\bX_{[0,T]}}(M) : = \widetilde M(S(\bX)_{0,T}),
\end{equation}
and call $\Phi_{\bX_{[0,T]}}(M)$ the unitary development of $\bX_{[0,T]}$ under $M$. Moreover, if $\bX_{[0,T]}(\omega)$ is a random geometric $p$-rough path (and we also allow $T(\omega)$ to be  random), then we call the following mapping
\begin{equation*}
    \phi_{\bX_{[0,T]}}: \bigsqcup_{k = 1}^\infty \lin(\R^d, \fu(k)) \to \bigsqcup_{k = 1}^\infty \C^{k \times k}, \quad \phi_{\bX_{[0,T]}}(M) : = \E[\Phi_{\bX_{[0,T]}}(M)],
\end{equation*}
namely the expectation of $\omega \mapsto \Phi_{\bX_{[0,T]}(\omega)}(M)$, the \textit{Path Characteristic Function (PCF)} of $\bX_{[0,T]}$. Note that it is the characteristic function of the distribution $\mu_{S(\bX)_{0,T}}$ on $E$ defined in \cite{ChevyrevLyons2016}.

As its name indicates, the most remarkable property of PCF is that it can determine the law of the signature of random geometric $p$-rough paths, which has been proved in \cite[Sect. 4]{ChevyrevLyons2016}: for two random geometric $p$-rough paths $\bX_{[0,T]}$ and $\bY_{[0,T]}$, one has $\mu_{S(\bX)_{0,T}} = \mu_{S(\bY)_{0,T}}$ iff $\phi_\bX(M) = \phi_\bY(M)$ for all $M \in \bigsqcup_{k = 1}^\infty \lin(\R^d, \fu(k))$. In a recent work \cite{RestrictedPCF2024}, the authors showed that the above characteristicness remains valid if one only applies $M \in \bigsqcup_{k = 1}^\infty \lin(\R^d, \so(k))$. For simplicity, from now on we will focus on $\so(k)$-valued linear operator $M$, and summarize the above facts into the following theorem, whose proof can be found in \cite[Theorem 5.3]{RestrictedPCF2024} (see also \cite[Corollary 4.12]{ChevyrevLyons2016} for the cases of using complex matrices).

\begin{theorem}\label{thm: PCF determines the law}
For two random geometric $p$-rough paths $\bX_{[0,T]}$ and $\bY_{[0,T]}$, one has $\mu_{S(\bX)_{0,T}} = \mu_{S(\bY)_{0,T}}$ if and only if $\phi_{\bX_{[0,T]}}(M) = \phi_{\bY_{[0,T]}}(M)$ for all $M \in \bigsqcup_{k = 1}^\infty \lin(\R^d, \so(k))$.
\end{theorem}

\begin{remark}\label{rem: time-augmented paths}
For two random geometric $p$-rough paths with the same fixed initial point $\bX_{[0,T]}(\omega)$ and $\bY_{[0,T]}(\omega)$ such that their first projection $X_{[0,T]}(\omega) = \bX^1_{[0,T]}(\omega)$ as well as  $Y_{[0,T]}(\omega) = \bY^1_{[0,T]}(\omega)$ can be expressed as $X_{t}(\omega) = (t, X^1_t(\omega), \ldots,X^d_t(\omega))$ and $Y_{t}(\omega) = (t, Y^1_t(\omega), \ldots,Y^d_t(\omega))$ for all $t \in [0,T]$, then they have the same law ($\mu_{\bX_{[0,T]}} = \mu_{\bY_{[0,T]}}$) if and only if $\phi_{\bX_{[0,T]}}(M) = \phi_{\bY_{[0,T]}}(M)$ for all $M \in \bigsqcup_{k = 1}^\infty \lin(\R^{d + 1}, \so(k))$, because the existence of the time component $t$ ensures the injectivity of the signature mapping, see \cite[Lemma 4.6]{BLGY2016}.
\end{remark}

Besides the (path) characteristic functions for random geometric rough paths defined as above, we also have a notion like ``moments generating functions'' for them, namely the \textit{expected signatures}: given a random geometric $p$-rough path $\bX_{[0,T]}(\omega) \in G\Omega_p(V)$, its expected signature is defined as
$$
\ESig(\bX_{[0,T]}) = (1, \E[S(\bX)_{0,T}^1], \ldots, \E[S(\bX)_{0,T}^n], \ldots) \in T((V)),
$$
provided for all $n \ge 1$, $S(\bX)_{0,T}^n \in V^{ \otimes n}$ is Gelfand-Pettis integrable in the sense that there is a vector $x^n \in V^{ \otimes n}$ such that for all $f \in (V^{ \otimes n})^\prime$ it holds $\E[f(S(\bX)_{0,T}^n)] = f(x^n)$, and in this case one sets $\E[S(\bX)_{0,T}^n] : = x^n$. Note that in the contrast to PCF which is well-defined for any random geometric $p$-rough paths, not every random geometric $p$-rough path has an expected signature.

If $\ESig(\bX_{[0,T]})$ exists, we write $r_2(\bX_{[0,T]}) = R(\ESig(\bX_{[0,T]}))$ for the radius of convergence of $\ESig(\bX_{[0,T]})$, i.e., the radius of convergence of the series
$
\sum_{n = 0}^\infty \|\E[S(\bX)_{0,T}^n]\|_{ \otimes n} \lambda^n
$
and write $r_1(\bX)$  for the radius of convergence of the series
$
\sum_{n = 0}^\infty \E[\|S(\bX)_{0,T}^n\|_{ \otimes n}] \lambda^n,
$
provided $\E[\|S(\bX)_{0,T}^n\|_{ \otimes n}] < \infty$ for all $n \ge 1$. Based on the fact that for geometric $p$-rough path $\bX_{[0,T]}$, its signature $S(\bX)_{0,T}$ takes values in the so-called character group in $E$ related to the Hopf algebra structure induced by the shuffle product (see \cite{ChevyrevLyons2016} for more details on these algebraic properties of signature) the following relation between $r_1$ and $r_2$ holds:

\begin{theorem}[Proposition 3.4 in \cite{ChevyrevLyons2016}]\label{thm: ROC of expected signatures}
Assume that $\bX_{[0,T]}(\omega) \in G\Omega_p(V)$ is a random geometric $p$-rough path. Then $\ESig(\bX_{[0,T]})$ exists iff $\E[\|S(\bX)_{0,T}^n\|_{ \otimes n}] < \infty$ for all $n \ge 1$, and one has $r_1(\bX_{[0,T]}) \le r_2(\bX_{[0,T]}) \le 2\sqrt{d}r_1(\bX_{[0,T]})$.
\end{theorem}

Here we must emphasize again that in general we do \textbf{not} have $\ESig(\bX_{[0,T]})$ also takes values in $E$ or $r_2(\bX_{[0,T]}) = \infty$. A counterexample is given in \cite{finiteROC2022}: the expected signature of stopped $d$-dimensional Brownian motion up to the first exit time of some $C^{2,\alpha}$-domain with $2 \le d \le 8$ only has a finite radius of convergence.

\subsection{Analyticity of the expectation of random functions} \label{subsec: random_func}
Theorem \ref{thm: random_func} below plays an essential role in proving our main results in the present paper.

Let $(\Omega,\cF,\bbP)$ be a probability space and $f(\lambda,\omega): \C \times \Omega \rightarrow \C$ be a random function with parameter $\lambda \in \C$. Define
\[
F(\lambda) : = \E[f(\lambda, \cdot)] = \int_{\omega \in \Omega} f(\lambda, \omega) \bbP(\dd \omega)
\]
to be the expectation of $f(\lambda,\omega)$ with respect to $\omega$, as long as $\omega \mapsto f(\lambda, \omega)$ is $\bbP$-integrable for such $\lambda$. We use $D_\C(\lambda_0, r) : = \{z \in \C: |z - \lambda_0| < r\}$ to denote the disk which is centred at $\lambda_0 \in \C$ and has a radius of $r > 0$. And we use $f^{(n)} (\lambda, \omega)$ to denote the $n$-th order partial derivative of $f$ with respect to $\lambda$ for $n \ge 0$ with the convention $f^{(0)} = f$.

\begin{theorem} \label{thm: random_func}
Let $\lambda_0 \in \C$ and $r > 0$. Suppose for every $\omega \in \Omega$, the function $\lambda \mapsto f(\lambda, \omega)$ is analytic on $D_\C(\lambda_0, r)$, and
    \[
    \sum_{n = 0}^\infty \frac{\E\left[|f^{(n)}(\lambda_0, \cdot)|\right]}{n!} r^n < \infty.
    \]
Then, the function $F$ is analytic on $D_\C(\lambda_0, r)$ and the following expansion holds:
    \begin{equation*}
        \forall \lambda \in D_\C(\lambda_0, r), \quad F(\lambda) = \sum_{n = 0}^\infty \frac{\E\left[ f^{(n)}(\lambda_0, \cdot)\right]}{n!} (\lambda - \lambda_0)^n.
    \end{equation*}
\end{theorem}

\begin{proof}
Without loss of generality, we assume $\lambda_0 = 0$. For any $|\lambda| < r$ and $N \geq 1$, define the function
    \begin{equation*}
        G_N (\lambda) : = \sum_{n = 0}^N \frac{\E\left[ f^{(n)}(0, \cdot)\right]}{n!} \lambda^n.
    \end{equation*}
By the hypothesis and Tonelli's theorem, the random Taylor series converges absolutely in $L^1$ for every $|\lambda| < r$. Its sum is $f(\lambda,\cdot)$ by samplewise analyticity, so $F(\lambda)$ is well-defined. Then, for any $|\lambda| < r$,
    \begin{align*}
        |F(\lambda) - G_N(\lambda)| = \left| \E \left[\sum_{n = N + 1}^\infty \frac{f^{(n)}(0, \cdot)}{n!} \lambda^n \right]\right| \leq \sum_{n = N + 1}^\infty \frac{\E\left[|f^{(n)}(0, \cdot)|\right]}{n!} r^n.
    \end{align*}
The right-hand side tends to zero independently of $\lambda \in D_\C(0,r)$. Thus $G_N$ converges locally uniformly to $F$, and $F$ is analytic as a locally uniform limit of polynomials. This also proves the stated expansion.
\end{proof}

\begin{proposition}[\textbf{Conditions (ROC)}]\label{prop: main criterion}
Let $\bX_{[0,T]}(\omega) \in G\Omega_p(V)$ be a random continuous geometric $p$-rough path defined on some probability space $(\Omega,\cF,\bbP)$ and (deterministic or random) time interval $[0,T]$. Assume that $\bX_{[0,T]}(\omega)$ satisfies the following two conditions, which we also call them \textbf{Conditions (ROC)}\footnote{ROC means Radius Of Convergence.}:
 \begin{enumerate}
     \item the expected signature $\ESig(\bX_{[0,T]}) = (1, \E[S(\bX)^1_{0,T}], \ldots, \E[S(\bX)^n_{0,T}], \ldots)$ has a positive radius of convergence $r_2(\bX_{[0,T]}) > 0$ (which is equivalent to $r_1(\bX_{[0,T]}) > 0$, see Theorem \ref{thm: ROC of expected signatures}).
     \item for every bounded subset $\bK \subset \R$ and every $M \in \bigsqcup_{k = 1}^\infty \lin(\R^d, \so(k))$, there exists a positive number $r = r(\bK,M) > 0$ such that the function $\lambda ( \in \R) \mapsto \phi_{\bX_{[0,T]}}(\lambda M) = \E[\widetilde{\lambda M}(S(\bX)_{0,T})]$ is analytic in $(\lambda_0 - r, \lambda_0 + r)$ for all $\lambda_0 \in \bK$.
 \end{enumerate}
Then the distribution $\mu_{S(\bX)_{0,T}}$ is determined by $\ESig(\bX_{[0,T]})$ within the class of random geometric rough paths satisfying the Conditions (ROC) as in Remark \ref{rem: VIP}.
\end{proposition}

\begin{proof}
Let us fix a non-zero $M \in \bigsqcup_{k = 1}^\infty \lin(\R^d, \so(k))$. Using the condition (1) we see that for all $\lambda \in \R$ small enough such that $|\lambda| < \frac{r_1(\bX_{[0,T]})}{\|M\|_{\op}}$, one has by the dominated convergence that
$$
\phi_{\bX_{[0,T]}}(\lambda M) = \E[\widetilde{\lambda M}(S(\bX)_{0,T})] = \sum_{n = 0}^\infty \lambda^n M^{ \otimes n}(\E[S(\bX)^n_{0,T}]) = \widetilde{\lambda M}(\ESig(\bX_{[0,T]})),
$$
which implies that the function $\lambda \mapsto \phi_{\bX_{[0,T]}}(\lambda M)$ is analytic in $( - \frac{r_1(\bX_{[0,T]})}{\|M\|_{\op}}, \frac{r_1(\bX_{[0,T]})}{\|M\|_{\op}})$, and its value in this interval is completely determined by its expected signature $\ESig(\bX_{[0,T]})$.

Now let $\bY_{[0,T]}$ be another random geometric rough path satisfying Conditions (ROC), with $\ESig(\bY_{[0,T]}) = \ESig(\bX_{[0,T]})$. The preceding argument shows that $\phi_{\bX_{[0,T]}}(\lambda M)$ and $\phi_{\bY_{[0,T]}}(\lambda M)$ agree for $|\lambda| < \min\{r_1(\bX_{[0,T]}),r_1(\bY_{[0,T]})\}/\|M\|_{\op}$. By condition (2), both functions are real analytic on $\R$. The identity theorem, applied entrywise on the connected interval $\R$, therefore gives $\phi_{\bX_{[0,T]}}(\lambda M) = \phi_{\bY_{[0,T]}}(\lambda M)$ for all $\lambda \in \R$. In particular, they agree at $\lambda = 1$. Since $M$ was arbitrary and equality at $M = 0$ is automatic, Theorem \ref{thm: PCF determines the law} implies $\mu_{S(\bX)_{0,T}} = \mu_{S(\bY)_{0,T}}$.
\end{proof}
Note that Condition (1) gives (P1) in \cite[Definition 6.12]{ChevyrevLyons2016}; and Condition (2) also gives its unitary version of (P2): for $M \in \lin(\R^d,\fu(k))$, the realification $M_{\R}(v) = \begin{pmatrix}\Re M(v) & - \Im M(v)\\ \Im M(v) & \Re M(v)\end{pmatrix}$ belongs to $\lin(\R^d,\so(2k))$, and realification commutes with development and expectation. Thus the real analyticity for $M_{\R}$ implies the entrywise real analyticity for $M \in \lin(\R^d,\fu(k))$. These are the conditions used there to establish determinacy within the corresponding analytic class in \cite{ChevyrevLyons2016}.

\begin{remark}\label{rem: VIP}
In this paper, we will find various conditions for various random continuous geometric rough paths $\bX_{[0,T]}$ to get feasible upper bounds $(B_q(n))_{n \ge 0}$ for the $L^q$-norms of $n$-th derivative of $Y^{\lambda M}_T$ for all $n \ge 0$ as in \eqref{eq: upper dev bound in intro}, all of which lead to the  \textbf{Conditions (ROC)},  under which the distribution $\mu_{S(\bX)_{0,T}}$ of its signature is determined by its expected signature $\ESig(\bX_{[0,T]})$ within this class. This means that if there is another random continuous geometric rough path $\bY_{[0,T]}$ which \textbf{also satisfies} \textbf{Conditions (ROC)} and $\ESig(\bX_{[0,T]}) = \ESig(\bY_{[0,T]})$, then it must hold that $\mu_{S(\bX)_{0,T}} = \mu_{S(\bY)_{0,T}}$. However, at this stage, we cannot exclude the possibility that there exists a random continuous geometric rough path $\bY_{[0,T]}$ which violates \textbf{Conditions (ROC)} such that $\ESig(\bX_{[0,T]}) = \ESig(\bY_{[0,T]})$ and $\mu_{S(\bX)_{0,T}} \neq \mu_{S(\bY)_{0,T}}$. This restriction was already observed in \cite[Remark 6.19]{ChevyrevLyons2016}.
\end{remark}

\section{High-order derivatives of unitary development}\label{sect: high-order derivatives}
In this section, we will find a sequence of positive numbers $(B_q(n))_{n \ge 0}$ and a constant $r>0$ for various random continuous geometric rough paths $\bX_{[0,T]}$ (whose expected signature has a positive radius of convergence) such that for any $M \in \bigsqcup_{k = 1}^\infty \lin(\R^d, \so(k))$ and any $n \ge 0$,
\begin{equation}\label{eq: upper bounds for dev in sect. 3}
 \sup_{\lambda \in \R}\frac{1}{n!}\bigg\|\frac{\dd^n}{\dd \lambda^n} Y^{\lambda M}_T \bigg\|_{L^q} \le B_q(n) \quad \text{or} \quad \sup_{\lambda \in \bK} \frac{1}{n!}\bigg\|\frac{\dd^n}{\dd \lambda^n} Y^{\lambda M}_T \bigg\|_{L^q} \le B_q(n)
\end{equation}
(where $\bK$ is any compact subset of $\R$) and $\sum_{n = 0}^\infty B_q(n)h^n$ converges for all $|h| < r$ for some $q \ge 1$. Thanks to the linear dynamics of $Y^{\lambda M}_T$, see \eqref{eq: RDE in introduction}, we can readily obtain the equations satisfied by $\frac{\dd^n}{\dd \lambda^n} Y^{\lambda M}_T$ and then apply different techniques to estimate their $L^q$-norms, see the subsections below.

\subsection{$X_{[0, T]}$ is smooth}\label{subsect: smooth}
In this subsection we will mainly consider smooth paths in $C^\infty([0,T],\R^d)$ for some finite $T > 0$. For a given linear operator $M \in \bigsqcup_{k=1}^\infty \lin(\R^d, \so(k))$ and a $\lambda \in \R$, suppose $Y^\lambda_{[0,T]} = Y^{\lambda M}_{[0,T]}$ is the solution to the following linear controlled differential equation:
\begin{align} \label{eq: CDE}
    \dd Y^\lambda_t = Y^\lambda_t(\lambda M)(\dd X_t) = Y^\lambda_t(\lambda M) (\dot X_t) \dd t, \quad Y^\lambda_0 = I_k.
\end{align}
We will denote the $n$-th radial derivatives of $Y^\lambda_t$ with respect to $\lambda$ (for fixed $t \in [0,T]$) by
\[
Y^{\lambda, (n)}_t : = \frac{\partial^n Y^\lambda_t}{\partial \lambda^n},
\]
with the convention that $Y^{\lambda, (0)}_t : = Y_t^\lambda$. We will collect some important properties of this path $Y^\lambda_{[0,T]}$ into the next theorem. In particular, we will show that $Y^{\lambda, (n)}_t$ exists for all $n \ge 1$ and derive an explicit formula for these partial derivatives.

\begin{theorem}\label{thm: properties of smooth path dev}
For a given linear operator $M \in \lin(\R^d, \so(k))$ and a $\lambda \in \R$, let $Y^\lambda_{[0,T]}$ be defined as in \eqref{eq: CDE}. Then we have
  \begin{enumerate}
      \item For all $t \in [0,T]$, we have $Y^\lambda_t = \widetilde{\lambda M}(S(X)_{0,t}) = \sum_{n = 0}^\infty \lambda^n M^{ \otimes n} (S(X)^n_{0,t}) = \Phi_{X_{[0,t]}}(\lambda M)$.
      \item For all $t \in [0,T]$, $Y^\lambda_t \in O(k)$ is an orthogonal matrix.
            \item For all $t \in [0,T]$, the mapping $\lambda \to Y^\lambda_t$ is real analytic on $\R$ and can be extended to a holomorphic function on $\C$, and the extended function $Y^z_t = \sum_{n = 0}^\infty z^n M^{ \otimes n} (S(X)^n_{0,t}) = \widetilde{zM}(S(X)_{0,t})$ for $z \in \C$ satisfies the linear differential equation
      $$
      \dd Y^z_t = Y^z_t (zM)(\dd X_t), \quad Y^z_0 = I_k.
      $$
Moreover, for any $n \ge 1$, $\lambda \in \R$, it holds that
      \begin{align}\label{eq: nth derivative of dev}
        Y^{\lambda,(n)}_t = n! \int_{\Delta^n_t} \Ad_{Y^\lambda_{s_1}}(M(\dot X_{s_1})) &\Ad_{Y^\lambda_{s_2}} (M(\dot X_{s_{2}})) \nonumber \\
        &\cdots \Ad_{Y^\lambda_{s_n}} (M(\dot X_{s_{n}}))\dd s_1 \dd s_2 \cdots \dd s_{n}Y^\lambda_t,
      \end{align}
where all concatenations are matrices multiplications and $\Ad$ is the Adjoint map defined as in Subsection \ref{subsect: Lie group}.
  \end{enumerate}
\end{theorem}

\begin{proof}
In view of \cite[Theorem 4.5]{Lyons2007}, the unique solution to the linear equation \eqref{eq: CDE} is given by
$$
Y^\lambda_t = \sum_{n = 0}^\infty(\lambda M)^{ \otimes n} (S(X)^n_{0,t}) = \sum_{n = 0}^\infty \lambda^n M^{ \otimes n} (S(X)^n_{0,t})
$$
where $S(X)^n_{0,t} \in V^{ \otimes n}$ is the $n$-th component of the signature of $X$ defined as in \eqref{eq: signature of BV}, and therefore is equal to $\widetilde {\lambda M} (S(X)_{0,t})$. Therefore, the solution path $Y^\lambda_{[0,T]}$ of equation \eqref{eq: CDE} coincides with the unitary development of $X_{[0,T]}$ under $\lambda M$, i.e., $Y^\lambda_t = \Phi_{X_{[0,t]}}(\lambda M)$. It then follows immediately that $Y^\lambda_t \in O(k)$ as $\lambda M$ takes values in $\so(k)$, see Subsection \ref{subsect: PCF} or \cite[Sect. 4]{ChevyrevLyons2016}. One can also use the product rule to check $Y^\lambda_t(Y^\lambda_t)^\dagger = I_k$ for all $t \in [0,T]$ directly, which also verifies our second claim.

Next, as we mentioned in Subsection \ref{subsect: basic notions in rough path theory}, the signature $S(X)_{0,t}$ takes values in the space $E$ and therefore has infinite radius of convergence, it follows that the power series $\sum_{n = 0}^\infty\lambda^n \|M^{ \otimes n} (S(X)^n_{0,t})\|_{\HS}$ also has infinite radius of convergence as we noticed in Subsection \ref{subsect: PCF}, which indeed implies that the function $\lambda \mapsto Y^\lambda_t = \sum_{n = 0}^\infty\lambda^n M^{ \otimes n} (S(X)^n_{0,t})$ is real analytic on $\R$ (in the sense that every entry $(Y^\lambda_t)^{ij}, i,j = 1,\ldots,k$ is a real analytic function on $\R$), and the absolute convergence of this power series also ensures that it can be extended to a holomorphic function on $\C$. Since $Y^z_t = \sum_{n = 0}^\infty z^n M^{ \otimes n} (S(X)^n_{0,t})$ for any  $z \in \C$, using \cite[Theorem 4.5]{Lyons2007} again we see that it is the unique solution to the linear equation $\dd Y^z_t = Y^z_t (zM)(\dd X_t), Y^z_0 = I_k$.

Now let us prove the formula \eqref{eq: nth derivative of dev} for $Y^{\lambda,(n)}_t$. By the classical smoothness result of ODE flows, see e.g. \cite[Theorem D.6]{LeeSM2013}, we know that the flow to the linear equation \eqref{eq: CDE}, $\Theta_M(\lambda,t) : = Y^\lambda_t$, is smooth in $(\lambda,t) \in \R \times [0,T]$. Therefore, we are able to take derivatives on both sides of \eqref{eq: CDE} with respect to $\lambda$ and by induction obtain
\begin{align}\label{eq: induction formula for nth derivative}
    \dd Y^{\lambda, (n)}_t = \Big[Y^{\lambda, (n)}_t  (\lambda M) (\dot X_t) + n Y^{\lambda,(n - 1)}_t M (\dot X_t) \Big] \dd t, \quad Y^{\lambda, (n)}_0 = 0
\end{align}
for all $n \ge 1$. Likewise, we define for $n \ge 0$ a path
\[
Z^{\lambda,(n)}_t : = Y^{\lambda,(n)}_t  (Y^\lambda_t)^{ - 1} = Y^{\lambda,(n)}_t (Y^\lambda_t)^\dagger.
\]
By the product rule and using \eqref{eq: induction formula for nth derivative} and \eqref{eq: CDE}, we can easily check that
\begin{align*}
   \dd Z^{\lambda,(n)}_t & = \dd  Y^{\lambda,(n)}_t (Y^\lambda_t)^\dagger + Y^{\lambda,(n)}_t \dd (Y^\lambda_t)^\dagger \\
   & = \Big[Y^{\lambda, (n)}_t  (\lambda M) (\dot X_t) + n Y^{\lambda,(n - 1)}_t M (\dot X_t) \Big] (Y^\lambda_t)^\dagger\dd t \\
   &\quad + Y^{\lambda,(n)}_t (\lambda M^\dagger)(\dot X_t)  (Y^\lambda_t)^\dagger \dd t \\
   & = n Y^{\lambda,(n - 1)}_t M (\dot X_t)(Y^\lambda_t)^\dagger\dd t,
\end{align*}
where we used the fact $M(\dot X_t) \in \so(k)$ so that $M(\dot X_t) + M^\dagger(\dot X_t) = 0$ to derive the last equality. Putting everything together, we finally obtain:
\begin{align*}
    Y^{\lambda,(n)}_t = Z^{\lambda,(n)}_t Y^\lambda_t = n \int_0^t Y^{\lambda,(n - 1)}_s M (\dot X_s) (Y^\lambda_s)^\dagger \dd s Y^\lambda_t.
\end{align*}
Then, by induction on $n$, it is not hard to see that
\begin{align*}
Y^{\lambda,(n)}_t & = n \int_0^t Y^{\lambda,(n - 1)}_s M (\dot X_s) (Y^\lambda_s)^\dagger \dd s Y^\lambda_t \\
& = n (n - 1) \int_0^t \int_0^{s_{n}} Y^{\lambda,(n - 2)}_{s_{n - 1}} M (\dot X_{s_{n - 1}}) (Y^\lambda_{s_{n - 1}})^\dagger  Y^\lambda_{s_n} M (\dot X_{s_n}) (Y^\lambda_{s_n})^\dagger \dd s_{n - 1} \dd s_n Y^\lambda_t\\
& = \cdots\\
& = n! \int_{\Delta^n_t}  \Big[Y^\lambda_{s_1} M(\dot X_{s_1})(Y^\lambda_{s_1})^\dagger \Big] \cdots \Big[Y^\lambda_{s_n} M(\dot X_{s_n})(Y^\lambda_{s_n})^\dagger \Big] \dd s_1 \cdots \dd s_n Y^\lambda_t\\
& = n! \int_{\Delta^n_t} \Ad_{Y^\lambda_{s_1}}(M(\dot X_{s_1}))
        \cdots \Ad_{Y^\lambda_{s_n}} (M(\dot X_{s_{n}}))\dd s_1 \dd s_2 \cdots \dd s_{n}Y^\lambda_t,
\end{align*}
which gives us the formula \eqref{eq: nth derivative of dev}.
\end{proof}
Since $Y^\lambda_t = \widetilde{\lambda M}(S(X)_{0,t})$, we see that $Y^\lambda_{T} = \Phi_{X_{[0,T]}}(\lambda M)$ is the unitary development of $X_{[0,T]}$ under $\lambda M$, see \eqref{eq: unitary dev}. Because $Y^\lambda_t \in O(k)$ is orthogonal, we can always pick $B_q(0) = \sqrt{k}$ for any $M \in \lin(\R^d, \so(k))$ and for any $q \ge 1$.
Furthermore, using the explicit formula \eqref{eq: nth derivative of dev}, we can easily get a reasonable bound for $\|Y^{\lambda,(n)}_t\|_{\HS}$ for all $n \ge 1$, which is independent of $\lambda$:

\begin{corollary}\label{cor: bound for dev of smooth paths}
Let $X_{[0,T]}$ be a smooth path. For a given linear operator $M \in \lin(\R^d, \so(k))$ and a $\lambda \in \R$, let $Y^\lambda_{[0,T]}$ be the unitary development under $\lambda M$. Then for all $\lambda \in \R$, $n \ge 1$ and $t \in [0,T]$, it holds that
    \begin{equation}\label{eq: bound for dev of smooth paths}
        \|Y^{\lambda,(n)}_t\|_{\HS} \le \|M\|_{\op}^n \|X\|^n_{1\var,[0,t]}.
    \end{equation}
\end{corollary}
\begin{proof}
By Theorem \ref{thm: properties of smooth path dev} we know that for any $\lambda \in \R$, the path $Y^\lambda_{[0,T]}$ lives in the orthogonal group $O(k)$. Then, since multiplication with orthogonal matrices preserves the Hilbert-Schmidt norm, we must have for any $s \in [0,T]$,
$$
\|\Ad_{Y^\lambda_{s}}(M(\dot X_{s})) \|_{\HS} = \|Y^\lambda_sM(\dot X_s)(Y^\lambda_s)^\dagger\|_{\HS} = \|M(\dot X_s)\|_{\HS},
$$
and therefore
\begin{align*}
  \|Y^{\lambda,(n)}_t\|_{\HS} & = n!\bigg\|\int_{\Delta^n_t} \Ad_{Y^\lambda_{s_1}}(M(\dot X_{s_1}))
        \cdots \Ad_{Y^\lambda_{s_n}} (M(\dot X_{s_{n}}))\dd s_1 \dd s_2 \cdots \dd s_{n}Y^\lambda_t \bigg\|_{\HS}   \\
        & \le n! \int_{\Delta^n_t} \|\Ad_{Y^\lambda_{s_1}}(M(\dot X_{s_1}))\|_{\HS}
        \cdots \|\Ad_{Y^\lambda_{s_n}} (M(\dot X_{s_{n}}))\|_{\HS}\dd s_1 \dd s_2 \cdots \dd s_{n} \\
        & = n!\int_{\Delta^n_t} \|M(\dot X_{s_1})\|_{\HS}
        \cdots \|M(\dot X_{s_{n}})\|_{\HS}\dd s_1 \dd s_2 \cdots \dd s_{n}\\
        & \le n!\|M\|_{\op}^n \int_{\Delta^n_t} |\dot X_{s_1}|
        \cdots |\dot X_{s_{n}}|\dd s_1 \dd s_2 \cdots \dd s_{n}\\
        & \le n!\|M\|_{\op}^n\frac{\|X\|^n_{1\var,[0,t]}}{n!}\\
        & = \|M\|_{\op}^n \|X\|^n_{1\var,[0,t]},
\end{align*}
and completes the proof.
\end{proof}

Now suppose that $X_{[0,T]}(\omega)$ is a random smooth path defined on a (deterministic or random) time interval $[0,T]$. Fix $k \ge 1$ and $M \in \bigsqcup_{k=1}^\infty \lin(\R^d, \so(k))$. From the inequality \eqref{eq: bound for dev of smooth paths} we observe that the $L^q$-norm of the $n$-th radial derivative of its unitary development $Y^{\lambda,(n)}_T$ satisfies an upper bound $\sup_{\lambda \in \R}\|Y^{\lambda,(n)}_T\|_{L^q} \le \|M\|_{\op}^n \E[\|X\|^{nq}_{1\var,[0,T]}]^{\frac{1}{q}}$ for all $n \ge 1$ and $q \ge 1$. That is, we can pick $B_q(n) := \frac{1}{n!}\|M\|_{\op}^n \|\|X\|^n_{1\var,[0,T]}\|_{L^q}$ for $n \ge 1$ and $B_q(0) = \sqrt{k}$ in \eqref{eq: upper bounds for dev in sect. 3}.
Thanks to such upper bounds, we now give a sufficient condition to guarantee the expected signature $\ESig(X_{[0,T]})$ determines the characteristic function $\phi_{X_{[0,T]}}$, and therefore, determines the law of the random signature $S(X)_{0,T}$.

\begin{corollary}\label{main result: X_smooth}
  Let $X_{[0,T]}(\omega)$ be a random smooth path. Suppose there exists an $R > 0$ such that
        \begin{equation}\label{eq: main condition for smooth path}
       \E[e^{R\|X\|_{1\var,[0,T]}}] < \infty.
        \end{equation}
Then,
\begin{enumerate}
    \item for all $n \ge 0$, we have $\frac{\dd^n}{\dd \lambda^n}\phi_{X_{[0,T]}}(\lambda M) = \E[Y^{\lambda,(n)}_T]$;
    \item $\ESig(X_{[0,T]})$ determines the distribution  $\mu_{S(X)_{0,T}}$ within the class of random geometric rough paths satisfying the Conditions (ROC) as in Remark \ref{rem: VIP}.
\end{enumerate}
\end{corollary}

\begin{proof}
By condition \eqref{eq: main condition for smooth path}, $\sum_{n = 0}^\infty \frac{\E(\|X\|_{1\var,[0,T]}^n)}{n!} R^n < \infty$. And by the factorial decay of signature \eqref{eq: rough path decay} applied to the case $p = 1$ (e.g., we pick the control function $w(s,t)$ as $\|X\|_{1\var,[s,t]}$ and $\gamma_1 = 1$), $\|S(X(\omega))_{0,T}^n\|_{ \otimes n} \le \frac{\|X(\omega)\|^n_{1\var,[0,T]}}{n!}$ for all $n \ge 0$ and $\omega \in \Omega$. Combining these two facts yields $r_1(X_{[0,T]}) \ge R > 0$, where the definition of $r_1$ can be found in Subsection \ref{subsect: PCF}. As a result, $R(\ESig(X_{[0,T]})) > 0$ by Theorem~\ref{thm: ROC of expected signatures}.

We now fix $k \ge 1$ and pick a non-zero $M \in \lin(\R^d, \so(k))$. By Theorem \ref{thm: properties of smooth path dev}, the random function $f(\lambda,\omega) : = Y^\lambda_T(\omega)$ is analytic on $\R$ for every $\omega \in \Omega$ so that for any $\lambda_0 \in \R$, one has $f^{(n)}(\lambda_0,\omega) = Y^{\lambda_0,(n)}_T(\omega)$ for all $n \ge 0$. By \eqref{eq: bound for dev of smooth paths} together with condition \eqref{eq: main condition for smooth path} we have for $r := \frac{R}{||M||_{\op}}$
    \[
    \sum_{n = 0}^\infty \frac{\E\big[||f^{(n)}(\lambda_0, \cdot)||_{\HS}\big]}{n!} r^n \le \sum_{n = 0}^\infty \frac{\sup_{\lambda \in \R}\|Y^{\lambda,(n)}_T\|_{L^1}r^n}{n!}  \le \E[e^{R\|X\|_{1\var,[0,T]}}] < \infty.
    \]
    Therefore, by Theorem \ref{thm: random_func}, we the function $$\lambda \mapsto \phi_{X_{[0,T]}}(\lambda M) = \E[f(\lambda,\cdot)] = \E[Y^\lambda_T] = \E[\Phi_{X_{[0,T]}}(\lambda M)]$$ is analytic in $(\lambda_0 - r, \lambda_0 + r)$ for all $\lambda_0 \in \R$ and $\frac{\dd^n}{\dd \lambda^n}\phi_{X_{[0,T]}}(\lambda M) = \E[Y^{\lambda,(n)}_T]$ for all $n \ge 0$. Now the Conditions (ROC) in Proposition \ref{prop: main criterion} are fulfilled for the random smooth path $X_{[0,T]}(\omega)$. So our last claim follows from Proposition \ref{prop: main criterion}.
\end{proof}

Using smooth approximation of Lipschitz continuous paths followed by time reparametrisation, we can extend all above results to random path $X_{[0,T]}(\omega) \in \Omega_1(V)$ which are continuous paths with bounded variation. Indeed, first we consider a deterministic path $X_{[0,T]}$ which is Lipschitz continuous (and therefore has bounded variation). By \cite[Proposition 1.32]{Friz2010}, there exists a sequence of smooth paths $X(m)_{[0,T]}, m \ge 1$ such that $X(m)_{[0,T]} \to X_{[0,T]}$ with respect to the $\|\cdot\|_{1\var,[0,T]}$-norm. Let us denote by $Y(m)^\lambda_{[0,T]}$ the development path of $X(m)_{[0,T]}$ under $\lambda M$ (that is, replace $X$ by $X(m)$ in \eqref{eq: CDE}) and by $Y^\lambda_{[0,T]}$ the counterpart driven by $X_{[0,T]}$.  By Theorem \ref{thm: properties of smooth path dev}, for each $m \ge 1$, the function $\lambda \mapsto Y(m)^{\lambda}_t$ can be extended to a holomorphic function on $\C$, which will also be denoted by the same notation. Since the signature of $X_{[0,T]} \in \Omega_1(V)$ also satisfies the factorial decay \eqref{eq: rough path decay}, using the same reasoning we also have $\lambda \in \R \mapsto Y^{\lambda}_t$ can be extended to a holomorphic function on $\C$. Note that by Theorem \ref{thm: properties of smooth path dev} the extended functions $Y^z_t$ and $Y(m)^z_t$ for $z \in \C$ still satisfy the linear equations \eqref{eq: CDE} driven by $X$ and $X(m)$, respectively.
Then, using the local Lipschitz continuity of RDE solution, see Theorem \ref{thm: well posedness of RDE}, applied to the linear vector fields $\bV_z: \C^{k \times k} \to \lin(\R^d, \C^{k \times k})$ defined via $\bV_z(y)(x) : = y(z M)(x)$ for $y \in \C^{k \times k}$ and $x \in \R^d$, we obtain that for each $t \in [0,T]$, $Y(m)^z_{t} \to Y^z_t$ as $m \to \infty$ locally uniformly in $z \in \C$,
which by the Weierstrass convergence theorem in complex analysis ensures that all $n$-th derivatives $Y(m)^{\lambda, (n)}_t$ converge to $Y^{\lambda, (n)}_t$ as $m \to \infty$, for all $\lambda \in \C$. In particular, for $\lambda \in \R$, since each $Y(m)^{\lambda, (n)}_t$ satisfies the formula \eqref{eq: nth derivative of dev}, we must have
\begin{align*}
    Y^{\lambda,(n)}_t & = \lim_{m \to \infty}  n! \int_{\Delta^n_t} \Ad_{Y(m)^\lambda_{s_1}}(M(\dot X(m)_{s_1})) \\
    &\quad \quad \quad \cdots \Ad_{Y(m)^\lambda_{s_n}} (M(\dot X(m)_{s_{n}}))\dd s_1 \cdots \dd s_{n}Y(m)^\lambda_t \\
    & = \lim_{m \to \infty}  n! \int_{\Delta^n_t} \Ad_{Y(m)^\lambda_{s_1}}(M(\dd X(m)_{s_1}))  \cdots \Ad_{Y(m)^\lambda_{s_n}} (M(\dd X(m)_{s_{n}}))Y^\lambda_t.
\end{align*}
Now, let us view $(Y^\lambda_t, Z^{\lambda,(1)}_t, \ldots, Z^{\lambda,(n)}_t) \in (\R^{k \times k})^{n + 1}$ as the unique solution to the following coupled ODE/RDE :
\begin{align}\label{eq: coupled ODE for derivatives}
\dd \begin{bmatrix}
    y^0_t\\
    z^1_t\\
    z^2_t\\
    \vdots\\
    z^n_t
\end{bmatrix} = \begin{bmatrix}
    \lambda y^0_t M (\dd X_t) \\
     \Ad_{y^0_t} M(\dd X_t) \\
    2z^1_t\Ad_{y^0_t} M(\dd X_t)\\
   \vdots\\
   nz^{n - 1}_t \Ad_{y^0_t} M(\dd X_t)
\end{bmatrix},
\qquad y^0_0 = I_k,\quad z^j_0 = 0\quad(1 \le j \le n).
\end{align}
where the vector field $\Ad_\cdot M: \GL(k,\R) \to \lin(\R^d, \R^{k \times k})$ should be read as $\Ad_\cdot M(y) (x) : = \Ad_yM(x) = yM(x)y^{ - 1}$ for $x \in \R^d$ and $y \in \GL(k,\R)$. Since we already know that the solution $y^0_t$ to the equation \eqref{eq: coupled ODE for derivatives} takes values in the compact group $O(k) \subset \GL(k,\R)$, we may replace $\Ad_\cdot M$ by a compactly supported smooth vector field on $\R^{k \times k}$ which agrees with it on a neighbourhood of $O(k)$, without changing the solution. This replacement is in the $\Lip^\gamma$-space for every $\gamma > 1$. Then by Theorem \ref{thm: well posedness of RDE} it follows that
\begin{align}\label{eq: dev formula for BV paths}
Y^{\lambda,(n)}_t
 = n! \int_{\Delta^n_t} \Ad_{Y^\lambda_{s_1}}(M(\dd X_{s_1}))  \cdots \Ad_{Y^\lambda_{s_n}} (M(\dd X_{s_{n}}))Y^\lambda_t.
\end{align}
Furthermore, since by Corollary \ref{cor: bound for dev of smooth paths} we have $$\|Y(m)^{\lambda,(n)}_t\|_{\HS} \le \|M\|_{\op}^n\|X(m)\|_{1\var,[0,t]}^n
$$
holds for all $n,m \ge 1$, we get that
\begin{align*}
    \|Y^{\lambda,(n)}_t\|_{\HS} = \lim_{m \to \infty}\|Y(m)^{\lambda,(n)}_t\|_{\HS}
    & \le \|M\|_{\op}^n \lim_{m \to \infty}\|X(m)\|_{1\var,[0,t]}^n \\
    & = \|M\|_{\op}^n\|X\|_{1\var,[0,t]}^n.
\end{align*}
Using this bound, we can easily generalize Proposition \ref{main result: X_smooth} to from smooth random paths to Lipschitz continuous random paths. Finally, since each continuous path with bounded variation $X_{[0,T]} \in \Omega_1(V)$ can be represented by a time-change as a Lipschitz continuous path on $[0,1]$ with Lipschitz constant at most $\|X\|_{1\var,[0,T]}$, see \cite[Proposition 1.38]{Friz2010}, and since the signature mapping is invariant under time reparametrisation, Proposition \ref{main result: X_smooth} actually holds for $X_{[0,T]} \in \Omega_1(V)$, which we record as the following corollary:

\begin{corollary}\label{cor: main result for BV}
For all $X_{[0,T]} \in \Omega_1(V)$ with bounded variation, all statements in Theorem \ref{thm: properties of smooth path dev} remains valid for $X_{[0,T]}$ with the formula \eqref{eq: nth derivative of dev} replaced by \eqref{eq: dev formula for BV paths}. Moreover, Suppose $X_{[0,T]}(\omega) \in \Omega_1(V)$ is a random path with bounded variation, and assume that there exists an $R > 0$ such that $ \E[e^{R\|X\|_{1\var,[0,T]}}] < \infty$. Then, all statements in Corollary \ref{main result: X_smooth} remains valid for $X_{[0,T]}$.
\end{corollary}

\begin{remark}\label{rem: connection with Ilya-Terry}
In \cite[Theorem 6.13]{ChevyrevLyons2016} the authors provides an useful criterion under which the expected signatures can determine the characteristic functions and the law of the random signature. More precisely, given any geometric $p$-rough path $\bX_{[0,T]} \in G\Omega_p(V)$ and any bounded measurable subset $B \subset E$, let $B(\bX_{[0,T]}) : = \inf\{m \ge 1: S(\bX)_{0,T} = x_1 \otimes \ldots \otimes x_m, x_i \in B, i = 1,\ldots,m\}$ (with $\inf \emptyset : = \infty$) be the minimum positive integer for which the signature $S(\bX)_{0,T}$ can be written as the (tensor) product of elements from $B$. Now suppose  $\bX_{[0,T]}(\omega)$ is a random geometric $p$-rough path, if there exist a bounded and measurable subset $B \subset E$ and a positive $\lambda > 0$ such that $\E[e^{\lambda B(\bX_{[0,T]})}] < \infty$, then $\ESig(\bX_{[0,T]})$ can determine $\mu_{S(\bX)_{0,T}}$ within the class of random geometric rough paths satisfying the Conditions (ROC) as in Remark \ref{rem: VIP}. Now, if $\bX_{[0,T]}(\omega) = X_{[0,T]}(\omega) \in \Omega_1(V)$ is a random bounded variation path, let us take $B : = \{S(X)_{0,T}: X_{[0,T]} \in \Omega_1(V), \|X\|_{1\var,[0,T]} \le 1\}$ which is indeed a bounded subset in $E$. Moreover, define a sequence of stopping times $\tau_0 = 0, \tau_{j + 1} : = \inf\{t \ge \tau_j: \|X\|_{1\var,[\tau_j, t]} = 1\} \wedge T$ for $j \ge 0$ and let $\bN(X_{[0,T]}) : = \sup\{j \ge 0: \tau_j < T\}$, by using the additivity of the $1$-variation norm we certainly obtain that $\bN(X_{[0,T]}) \le \|X\|_{1\var,[0,T]} \le \bN(X_{[0,T]}) + 1$. This observation together with the Chen's relation $S(X)_{0,T} = \prod_{j = 0}^{\bN(X)} S(X_{[\tau_j, \tau_{j + 1}]})$ implies that $B(X_{[0,T]}) \le \bN(X_{[0,T]}) + 1 \le \|X\|_{1\var,[0,T]} + 1$ and therefore the condition $\E[e^{\lambda B(X_{[0,T]})}] < \infty$ for some $\lambda > 0$ used in \cite[Theorem 6.13]{ChevyrevLyons2016} is actually guaranteed by assuming $\E[e^{R \|X\|_{1\var,[0,T]}}] < \infty$ for some $R > 0$, which is exactly our condition \eqref{eq: main condition for smooth path} in Proposition \ref{main result: X_smooth} and Corollary \ref{cor: main result for BV}.
\end{remark}

\subsection{$X_{[0, T]}$ is of finite $p$-variation with $p > 2$}\label{subsect: greedy sequence}

Now let us turn to the case of geometric $p$-rough paths $\bX_{[0,T]} \in G\Omega_p(V)$ for some $p > 2$. Recall that $\bX_t \in T^{\lfloor p \rfloor}(V)$ for all $t \in [0,T]$ and there exists a sequence of  paths $X(m)_{[0,T]} \in \Omega_1(V), m \ge 1$ such that $\lim_{m \to \infty} \|S_{\lfloor p \rfloor}(X(m)); \bX\|_{p\var} = 0$. Let $X_{[0,T]} = \bX^1_{[0,T]}: [0, T] \to V$ denote the first coordinate projection of $\bX_{[0,T]}$ onto $V$, which belongs to $C^{p\var}([0,T],\R^d)$. Conversely, if $X_{[0,T]}(\omega) \in C^{p\var}([0,T],\R^d)$ is a random path with finite $p$-variation in some special class, then one can construct a random geometric $p$-rough path $\bX_{[0,T]}(\omega) \in G\Omega_p(V)$ above $X_{[0,T]}(\omega)$. The most prominent examples include continuous semimartingales (see the next subsection), or centered continuous Gaussian processes with sufficiently regular covariance matrix, see e.g. \cite[Example 6.7]{ChevyrevLyons2016} or \cite[Chapter 15]{Friz2010}, or Markovian rough paths generated by Dirichlet forms, see e.g. \cite[Example 6.8]{ChevyrevLyons2016} or \cite[Chapter 16]{Friz2010}.

As in the smooth case, for a given $k \ge 1$, $M \in \lin(\R^d, \so(k))$ and $\lambda \in \R$, let us consider the unitary development path of $\bX_{[0,T]}$ along $\lambda M$, again denoted by $Y^\lambda_{[0,T]}$, which is now the unique solution to the following linear rough differential equation (RDE):
\begin{equation}\label{eq: linear RDE}
    \dd Y^\lambda_t = Y^\lambda_t (\lambda M)(\dd \bX_t), \quad Y^\lambda_0 = I_k.
\end{equation}
Note that the above RDE should be interpreted as the equation driven by rough path $\bX_{[0,T]}$ along the vector field $V_\lambda(y)x : = y(\lambda M)(x)$ for $y \in \R^{k \times k}$ and $x \in \R^d$. By the methodology for solving RDE, see Subsection \ref{subsect: RDE}, for any sequence of bounded variation paths $X(m)_{[0,T]} \in \Omega_1(V), m \ge 1$
with $$\lim_{m \to \infty} \|S_{\lfloor p \rfloor}(X(m)); \bX\|_{p\var} = 0,
$$
let $Y^\lambda(m)_{[0,T]}$ be the development path along $\lambda M$ for $X(m)_{[0,T]}$, we have $Y^\lambda(m)_{[0,T]} \to Y^\lambda_{[0,T]}$ uniformly on $[0,T]$ as $m \to \infty$. Hence, by using the results for bounded variation paths, and following the same arguments for the extension from smooth paths to bounded variation paths as in the last subsection, we can easily show that

\begin{theorem}\label{thm: properties of dev for rough paths}
Let $p > 2$ and $\bX \in G\Omega_p(\R^d)$. Fix $k \ge 1$ and $M \in \lin(\R^d, \so(k))$, and for each $\lambda \in \R$ let $Y^\lambda_{[0,T]}$ be defined as in \eqref{eq: linear RDE}. Then we have
  \begin{enumerate}
      \item For all $t \in [0,T]$, we have $Y^\lambda_t = \widetilde{\lambda M}(S(\bX)_{0,t}) = \sum_{n = 0}^\infty \lambda^n M^{ \otimes n} (S(\bX)^n_{0,t}) = \Phi_{\bX_{[0,t]}}(\lambda M)$.
      \item For all $t \in [0,T]$, $Y^\lambda_t \in O(k)$ is an orthogonal matrix.
            \item For all $t \in [0,T]$, the mapping $\lambda \to Y^\lambda_t$ is real analytic on $\R$ and can be extended to a holomorphic function on $\C$, and the extended function $Y^z_t = \sum_{n = 0}^\infty z^n M^{ \otimes n} (S(\bX)^n_{0,t}) = \widetilde{zM}(S(\bX)_{0,t})$ for $z \in \C$ satisfies the linear RDE
      $$
      \dd Y^z_t = Y^z_t (zM)(\dd \bX_t), \quad Y^z_0 = I_k.
      $$
Moreover, for any $n \ge 1$, $\lambda \in \R$, it holds that
      \begin{align*}
        Y^{\lambda,(n)}_t = n! \int_{\Delta^n_t} \Ad_{Y^\lambda_{s_1}}(M(\dd \bX_{s_1}))  \cdots \Ad_{Y^\lambda_{s_n}} (M(\dd \bX_{s_{n}}))Y^\lambda_t,
    \end{align*}
where the integral $n! \int_{\Delta^n_t} \Ad_{Y^\lambda_{s_1}}(M(\dd \bX_{s_1}))  \cdots \Ad_{Y^\lambda_{s_n}} (M(\dd \bX_{s_{n}}))$ denotes the $z^n$-component of the unique solution to the coupled RDE \eqref{eq: coupled ODE for derivatives} driven by $\bX_{[0,T]}$.
  \end{enumerate}
\end{theorem}

\begin{remark}\label{rem: iterated rough integral}
For $2 < p < 3$, using Gubinelli's controlled path theory (\cite{Gubinelli2004}) we can also interpret $\int_{\Delta^n_t} \Ad_{Y^\lambda_{s_1}}(M(\dd \bX_{s_1}))  \cdots \Ad_{Y^\lambda_{s_n}} (M(\dd \bX_{s_{n}}))$ as an iterated integral of a matrix-valued controlled path. Indeed, by standard results in controlled path theory, see e.g. \cite[Theorem 3.10]{Friz2018}, the RDE solution $(Y^\lambda_{[0,T]}, \hat \bX^\lambda_{[0,T]})$ to the coupled RDE (see \eqref{eq: coupled ODE for derivatives})
$$
\dd \begin{bmatrix}
    y^0_t\\
    z^1_t
\end{bmatrix} = \begin{bmatrix}
    \lambda y^0_t M(\dd \bX_t)\\
    \Ad_{y^0_t} M(\dd \bX_t)
\end{bmatrix},
$$
is a controlled path relative to $\bX$, where
\begin{equation}\label{eq: hat X}
    \hat \bX^\lambda_t : = \int_0^t \Ad_{Y^\lambda_{s}}(M(\dd \bX_{s})), \quad t \in [0,T].
\end{equation}
Further, by \cite[Remark 2.6]{Allan2023b}, the $n$-th iterated integral $\int_{\Delta^n_t} \cdot \dd \hat \bX^\lambda_{s_1} \ldots \cdot\dd \hat \bX^\lambda_{s_n} \in \R^{k \times k}$ of controlled path $\hat \bX^\lambda_{[0,T]}$ against itself is well-defined and remains an $\bX$-controlled path. Here we use ``$\cdot$'' to emphasize the concatenation is the matrix multiplication instead of the tensor product used in defining the iterated integrals for signature. Clearly this iterated integral  $\int_{\Delta^n_t} \cdot \dd \hat \bX^\lambda_{s_1} \ldots \cdot\dd \hat \bX^\lambda_{s_n}$  satisfies the recursive relation:
\begin{equation}\label{eq: recursive formula for matrix signature}
  D_0(\hat \bX^\lambda_t) = I_k, \quad D_{n + 1}(\hat \bX^\lambda_t) = \int_0^t D_{n}(\hat \bX^\lambda_s) \cdot \dd \hat \bX^\lambda_s.
\end{equation}
Using the stability results on rough integration, see e.g. \cite[Lemma 3.4]{Friz2018}, we can easily see that
\begin{align*}
    D_n(\hat \bX^\lambda_t) & = \int_{\Delta^n_t} \cdot\dd \hat \bX^\lambda_{s_1} \ldots \cdot\dd \hat \bX^\lambda_{s_n}\\
    & = \lim_{m \to \infty}  \int_{\Delta^n_t} \Ad_{Y(m)^\lambda_{s_1}}(M(\dd X(m)_{s_1}))  \cdots \Ad_{Y(m)^\lambda_{s_n}} (M(\dd X(m)_{s_{n}}))\\
    & = \int_{\Delta^n_t} \Ad_{Y^\lambda_{s_1}}(M(\dd \bX_{s_1}))  \cdots \Ad_{Y^\lambda_{s_n}} (M(\dd \bX_{s_{n}})).
\end{align*}
\end{remark}
Therefore, we may call the sequence of matrices $(I_k, \hat \bX^\lambda_t, \ldots, D_n(\hat \bX^\lambda_t), \ldots)$ the matrix signature of the path $\hat \bX^\lambda_{[0,T]}$. Clearly, we need to analyse the decay rate of this matrix signature in order to study the radius of convergence for the power series $\sum_{n = 0}^\infty \frac{1}{n!}\|Y^{\lambda,(n)}_t\|_{\HS} |\xi|^n$ in the increment $\xi$ about the fixed centre $\lambda$.
In fact, by the standard estimation for the $p$-variation of solutions to RDEs (Theorem \ref{thm: well posedness of RDE}) together with the Lyons extension theorem (\cite[Theorem 3.1.2]{Lyons2002}), one can obtain that for every $n \ge 1$, \begin{equation*}\label{eq: ugly bound for matrix signature}
 \frac{1}{n!}\|Y^{\lambda,(n)}_t\|_{\HS} \le C^n k^n\frac{(1 + \|\bX\|_{p\var,[0,T]})^{np}}{(\frac{n}{p})!},
\end{equation*}
where $C = C(\lambda, \|M\|_{\op}, k,p) \ge 1$ is locally bounded in $\lambda$. For $n = 0$, $\|Y^\lambda_t\|_{\HS} = \sqrt{k}$. However, this bound is too coarse for concrete application: for instance, even if $\bX(\omega)_{[0,T]}$ is the Stratonovich lift of a Brownian motion, by the Fernique theorem (\cite[Corollary 13.14]{Friz2010}) one only has $\E[e^{\eta \|\bX\|_{p\var,[0,T]}^2}] < \infty$ for some $\eta > 0$, which cannot guarantee the power series $\sum_{n = 0}^\infty \E[\frac{(1 + \|\bX\|_{p\var,[0,T]})^{np}}{(\frac{n}{p})!}]\lambda^n$ has a positive radius of convergence. To solve this issue, we have to use the greedy sequence introduced in \cite{CassLittererLyons2013}, which was also  heavily used in \cite[Sect. 6]{ChevyrevLyons2016} for investigating the analyticity of characteristic functions of random geometric rough paths, to estimate the $L^q$-norms of $Y^{\lambda,(n)}_t$ in an indirect way instead of studying the decay rate of the matrix signature $D_n(\hat \bX^\lambda_t)$. Recall that given a geometric $p$-rough path $\bX_{[0,T]}$, we define the sequence of (stopping) times:
$$
\tau_0 = 0, \quad \tau_{j + 1} = \inf\{t > \tau_j: \|\bX\|_{p\var,[\tau_j,t]} \ge 1\} \wedge T, \quad j \ge 0,
$$
and then let $\bN = \bN(\bX_{[0,T]}) = \sup\{j \ge 0: \tau_j < T\}$. Note that for all $1 \leq j \le \bN$, we have $\|\bX\|_{p\var, [\tau_{j - 1},\tau_j]} = 1$ and $\|\bX\|_{p\var, [\tau_{\bN},T]} \le 1$.

\begin{lemma}[Uniform bound on the development over greedy intervals]
\label{lem:uniform-greedy-development} Fix $p > 2$ and $k \ge 1$. Let $\bX \in G\Omega_p(\mathbb{R}^d)$ and let $M \in \lin(\mathbb{R}^d,\mathfrak{so}(k))$. Let $(\tau_i)_{i \geq 0}$ be the greedy sequence defined as above and
              $\bN = \bN(\bX_{[0,T]}) : = \sup\{j \geq 0:\tau_j < T\}. $ For $z \in \mathbb{C}$ and $0 \leq i \leq \bN$, let $\widetilde Y^{z,i}$ denote the solution on $[\tau_i,\tau_{i + 1}]$ of\footnote{Note that if $z \in \C\setminus \R$, then $\widetilde Y_t^{z,i}$ may not belong to $O(k)$, because the Lie algebra $\so(k)$ is only an $\R$-vector space. }
\[
    \dd \widetilde Y_t^{z,i} = \widetilde Y_t^{z,i}\,(zM)(\dd \bX_t),
    \qquad
    \widetilde Y_{\tau_i}^{z,i} = I_k.
\]
Fix $\lambda \in \mathbb{R}$. Then, for every $r > 0$, there exists a constant
$C_{\lambda,r,k,M}
$
which is independent of the interval index $i$ and locally bounded in $\lambda$, such that
\begin{equation*}
    \sup_{t \in [\tau_i,\tau_{i + 1}]}
    \|
        \widetilde Y_t^{z,i} - \widetilde Y_t^{\lambda,i}
    \|_{\mathrm{op}} \leq C_{\lambda,r,k,M}|z - \lambda|
\end{equation*}
for every $z \in D_\C(\lambda,r)$ and every $0 \leq i \leq \bN$. Consequently,
\begin{equation}\label{eq: bound on complex development on greedy sequence}
   \sup_{\substack{z \in D_\C(\lambda,r)\\
                    0 \leq i \leq \bN\\
                    t \in [\tau_i,\tau_{i + 1}]}}
    \|\widetilde Y_t^{z,i}\|_{\mathrm{op}} \leq K_{\lambda,r,k,M},
\end{equation}
where $K_{\lambda,r,k,M} : = 1 + C_{\lambda,r,k,M}r. $ In particular, $K_{\lambda,r,k,M} \to 1$ as $r\downarrow 0$ locally uniformly in $\lambda$.
\end{lemma}

\begin{proof}
By the definition of the greedy sequence, we have
\[
    \|\bX\|_{p\text{-var},
        [\tau_i,\tau_{i + 1}]} \leq 1
\]
for every $0 \leq i \leq \bN$. Hence all the restrictions $\bX|_{[\tau_i,\tau_{i + 1}]}$ belong to the same bounded subset of the $p$-variation rough path space.

For $z \in \mathbb{C}$, consider the linear vector field
\[
    \bV_z(y)(x): = y\,(zM)(x),
    \quad
    y \in \mathbb{C}^{k \times k},
    \quad x \in \mathbb{R}^d.
\]
If $|z - \lambda| \leq r$, then
\[
    |z| \leq |\lambda| + r,
\]
and therefore the norm of $\bV_z$ is bounded uniformly over the disk $\{z:|z - \lambda| \leq r\}$ by a constant depending only on $(|\lambda| + r)\|M\|_{\mathrm{op}}$. Moreover,
\[
    \bV_z(y)(x) - \bV_\lambda(y)(x) = y\,(z - \lambda)M(x),
\]
so that
\[
    \|\bV_z - \bV_\lambda\|_{\op} \leq C\,|z - \lambda|\,\|M\|_{\mathrm{op}},
\]
where the constant $C$ depends only on $k$.

Then, applying the local stability estimate for linear rough differential equations on the interval $[\tau_i,\tau_{i + 1}]$ (see Theorem \ref{thm: well posedness of RDE}) gives
\[
    \sup_{t \in [\tau_i,\tau_{i + 1}]}
    \|
        \widetilde Y_t^{z,i} - \widetilde Y_t^{\lambda,i}
    \|_{\mathrm{op}} \leq C_{\lambda,r,k,M}|z - \lambda|,
\]
where the constant $C_{\lambda,r,k,M}$ is locally bounded in $\lambda$. The crucial point here is that the Lipschitz constant $C_{\lambda,r,k,M}$ depends on the driving rough path only through an upper bound for its $p$-variation, so the condition $\|\bX\|_{p\text{-var}; [\tau_i,\tau_{i + 1}]} \leq 1 $ for all $0 \le i \le \bN$, the same constant $C_{\lambda,r,k,M}$ can be chosen for every greedy interval $[\tau_i, \tau_{i + 1}]$.

Now, for a fixed $\lambda \in \mathbb{R}$, recall that $\widetilde Y_t^{\lambda,i} \in O(k)$ for every $0 \le i \le \bN$, we have $\|\widetilde Y_t^{\lambda,i}\|_{\mathrm{op}} = 1$. Therefore, for $z \in D_\C(\lambda,r)$, it holds that
\begin{align*}
    \|\widetilde Y_t^{z,i}\|_{\mathrm{op}} \leq \|\widetilde Y_t^{\lambda,i}\|_{\mathrm{op}} + \|
        \widetilde Y_t^{z,i} - \widetilde Y_t^{\lambda,i}
    \|_{\mathrm{op}} \leq 1 + C_{\lambda,r,k,M}|z - \lambda| \leq 1 + C_{\lambda,r,k,M}r.
\end{align*}
Thus, setting
\[
    K_{\lambda,r,k,M}
    : = 1 + C_{\lambda,r,k,M}r,
\]
we obtain the bound \eqref{eq: bound on complex development on greedy sequence}. Clearly, $K_{\lambda,r,k,M} \to 1$ as $r \to 0$, which is locally uniformly in $\lambda$.
\end{proof}

\begin{proposition}\label{prop: bound on development: greedy sequence}
Let $\bX_{[0,T]}$ be a geometric $p$-rough path with $p > 2$. Fix $k \ge 1$ and $M \in \lin(\R^d,\so(k))$. For any $\eta > 0$, any bounded subset $\bK \subset \R$, there exists an $r = r(\eta, \bK, p,k,M) > 0$ such that for all $n \ge 0$,
\begin{equation}\label{eq: bound on expected development: greedy sequence}
\sup_{\lambda \in \bK} \|Y^{\lambda,(n)}_T\|_{\HS} \le \frac{n!}{r^n} \sqrt{k} e^\eta e^{\eta \bN(\bX_{[0,T]})}.
\end{equation}
\end{proposition}

\begin{proof}
For $z \in \C$, let $Y^z_T : = \widetilde{(zM)}(S(\bX)_{0,T})$, which is a holomorphic function in $z$ by Theorem \ref{thm: properties of dev for rough paths}. By Chen's relation, we have
$$
S(\bX)_{0,T} = S(\bX)_{0, \tau_1} \otimes \cdots \otimes S(\bX)_{\tau_{\bN(\bX_{[0,T]})}, T},
$$
and since $\widetilde M: E \to \C^{k \times k}$ is an algebraic homomorphism, we have
$$
Y^z_T = \widetilde{(zM)}(S(\bX)_{0,T}) = \prod_{i = 0}^{\bN(\bX_{[0,T]})} \widetilde{(zM)}(S(\bX)_{\tau_i,\tau_{i + 1}}).
$$
Further, for all $0 \le i \le \bN(\bX_{[0,T]})$, by Theorem \ref{thm: properties of dev for rough paths} we have $\widetilde{(zM)}(S(\bX)_{\tau_i,\tau_{i + 1}}) = \tilde Y^{z,i}_{\tau_{i + 1}}$, where the paths $\tilde Y^{z,i}_{[\tau_i,\tau_{i + 1}]}$ are defined as in Lemma \ref{lem:uniform-greedy-development}. Hence, we have
$$
 Y_T^z = \widetilde Y^{z,0}_{\tau_1}
    \widetilde Y^{z,1}_{\tau_2}
    \cdots
    \widetilde Y^{z,\bN(\bX_{[0,T]})}_{T};
$$
and by the sub-multiplicativity of the operator norm as well as the bound \eqref{eq: bound on complex development on greedy sequence},
\[
    \|Y_T^z\|_{\mathrm{op}} \leq \prod_{i = 0}^{\bN(\bX_{[0,T]})}
    \|\widetilde Y^{z,i}_{\tau_{i + 1}}\|_{\mathrm{op}} \leq K_{\lambda,r,k,M}^{\bN(\bX_{[0,T]}) + 1},
\]
Using the fact that $\|A\|_{\mathrm{HS}} \leq \sqrt{k}\,\|A\|_{\mathrm{op}}$, by \eqref{eq: bound on complex development on greedy sequence} again it follows that
\begin{equation}\label{eq: uniform bound over complex disk}
     \sup_{z \in D_\C(\lambda,r)}
    \|Y^z_T\|_{\mathrm{HS}} \leq \sqrt{k} K_{\lambda,r,k,M}^{\bN(\bX_{[0,T]}) + 1}.
\end{equation}
Since by Lemma \ref{lem:uniform-greedy-development}, $K_{\lambda,r,k,M} \to 1$ as $r \to 0$ locally uniformly in $\bK$, for the given $\eta > 0$ we can pick $r = r(\eta, \bK,p,k,M) > 0$ small enough such that $\sup_{\lambda \in \bK}K_{\lambda,r,k,M} < e^\eta$. Then, using \eqref{eq: uniform bound over complex disk} and the Cauchy's integral formula for derivatives of holomorphic functions, namely,
\begin{align*}
 Y^{\lambda,(n)}_T = \frac{n!}{2\pi \mathrm i} \oint_{|z - \lambda| = \rho} \frac{Y^z_T}{(z - \lambda)^{n + 1}} \dd z, \quad 0 < \rho < r,
\end{align*}
we can deduce that for all $\lambda \in \bK$,
\begin{align*}
    \|Y^{\lambda,(n)}_T\|_{\HS} \le \sup_{z \in D_\C(\lambda,r)}
    \|Y_T^z\|_{\mathrm{HS}} \frac{n!}{r^n} \le \frac{n!}{r^n} \sqrt{k}\,
    K_{\lambda,r,k,M}^{\bN(\bX_{[0,T]}) + 1} \le \frac{n!}{r^n} \sqrt{k} e^{\eta (\bN(\bX_{[0,T]}) + 1)}
\end{align*}
and then taking expectations on both sides we complete the proof.
\end{proof}

\begin{corollary}\label{cor: exponential tail of N}
Let $\bX_{[0,T]}$ be a random geometric $p$-rough path for a fixed $p > 2$ (where $T$ can also be random). Suppose $\E[e^{\eta \bN(\bX_{[0,T]})}] < \infty$ for some $\eta > 0$. Then for any $M \in \bigsqcup_{k=1}^\infty \lin(\R^d, \so(k))$ and any bounded subset $\bK\subset \R$,
\begin{enumerate}
    \item for any $q \ge 1$, there exists an $r = r(\eta/q, \bK, p,k,M)$ such that for all $n \ge 0$ it holds that
    \begin{equation}\label{eq: upper bounds for greedy seuqence case}
        \sup_{\lambda \in \bK} \frac{\|Y^{\lambda,(n)}_T\|_{L^q}}{n!} \le \frac{\sqrt{k}e^\eta}{r^n} \E[e^{\eta \bN(\bX_{[0,T]})}]^{1/q};
    \end{equation}
    \item for any $M \in \bigsqcup_{k=1}^\infty \lin(\R^d, \so(k))$ and any $n \ge 0$, we have $\frac{\dd^n}{\dd \lambda^n} \phi_{\bX_{[0,T]}}(\lambda M) = \E[Y^{\lambda, (n)}_T] = \E[n!D_n(\hat \bX^\lambda_T)Y^{\lambda}_T]$;
    \item for any $M \in \bigsqcup_{k=1}^\infty \lin(\R^d, \so(k))$, $q \ge 1$, $\lambda \in \bK$ and $N\ge 0$, one has
    $$
    \forall |h| < r(\frac{\eta}{q}, \bK, p,k,M), \quad
    \bigg\|Y^{\lambda+h}_T - \sum_{n=0}^N \frac{Y^{\lambda,(n)}_T}{n!}\bigg\|_{L^q} \le \sum_{n=N+1}^\infty\frac{\sqrt{k}e^\eta|h|^n}{r^n} \E[e^{\eta \bN(\bX_{[0,T]})}]^{1/q}
    $$
    and
    $$
    \bigg\|\phi_{\bX_{[0,T]}}((\lambda+h)M) - \sum_{n=0}^N \frac{\frac{\dd^n}{\dd \lambda^n}\phi_{\bX_{[0,T]}}(\lambda M)}{n!}h^n\bigg\|_{\HS} \le \sum_{n=N+1}^\infty\frac{\sqrt{k}e^\eta|h|^n}{r^n} \E[e^{\eta \bN(\bX_{[0,T]})}]^{1/q}.
    $$
    \item  $\ESig(\bX_{[0,T]})$ determines the distribution  $\mu_{S(\bX)_{0,T}}$ within the class of random geometric rough paths satisfying the Conditions (ROC) as in Remark \ref{rem: VIP}.
\end{enumerate}
\end{corollary}

\begin{proof}
 By \eqref{eq: bound on expected development: greedy sequence} we see that there exists an $r = r(\eta/q, \bK, p,k,M)>0$ such that for all $\lambda \in \bK$,
 $\|Y^{\lambda,(n)}_T\|_{\HS} \le \frac{n!}{r^n} \sqrt{k} e^{\frac{\eta}{q} (\bN(\bX_{[0,T]}) + 1)}$. Taking the $L^q$-norms on the both sides we obtain the desired upper bound \eqref{eq: upper bounds for greedy seuqence case}. Then the condition that $\E[e^{\eta \bN(\bX_{[0,T]})}] < \infty$, we have for all $0 < |h| < r$, where $r = r(\eta/q, \bK,p,k,M)$ as above:
$$
 \sum_{n = 0}^\infty \frac{1}{n!}\|Y^{\lambda_0,(n)}_T\|_{L^q}|h|^n \le \sqrt{k} e^\eta\E[e^{\eta \bN(\bX_{[0,T]})}]^{1/q} \sum_{n = 0}^\infty  \bigg(\frac{|h|}{r}\bigg)^n < \infty.
$$
This of course implies that the function $\lambda \mapsto \phi_{\bX}(\lambda M)$ is analytic in $(\lambda_0 - r,\lambda_0 + r)$ for every $\lambda_0 \in \bK$ and $\frac{\dd^n}{\dd \lambda^n} \phi_{\bX_{[0,T]}}(\lambda M) = \E[Y^{\lambda, (n)}_T]$ by using Theorem \ref{thm: random_func}. Moreover, by \cite[Theorem 6.13]{ChevyrevLyons2016} we notice that the condition $\E[e^{\eta \bN(\bX_{[0,T]})}] < \infty$ also ensures that $\ESig(\bX_{[0,T]})$ has a positive radius of convergence. So $\bX_{[0,T]}$ satisfies the Conditions (ROC) and then by using Proposition \ref{prop: main criterion} it follows that $\phi_{\bX_{[0,T]}}(\cdot)$ is uniquely determined by $\ESig(\bX_{[0,T]})$ and therefore $\mu_{S(\bX)_{0,T}}$ is uniquely determined by $\ESig(\bX_{[0,T]})$.
\end{proof}

\begin{remark}
Our proof largely follows the strategy developed in \cite[Sect. 6]{ChevyrevLyons2016} based on an application of greedy sequences, with two modifications that are useful for our purpose. First, instead of working directly with $\phi_{\bX}(\lambda M)$, we introduce the matrix-valued process $Y^{\lambda}_{[0,T]}$ as an intermediate object. This allows us to isolate the analytic dependence on $\lambda$ at the level of the matrix development. Second, after obtaining the local complex-analytic bound as in \cite[Theorem 6.13]{ChevyrevLyons2016}, we apply Cauchy's estimates to derive explicit bounds for the Hilbert--Schmidt norms of all higher-order derivatives:
\[
\frac{1}{n!}
\left\|
Y_T^{\lambda,(n)}
\right\|_{\mathrm{HS}}
\lesssim
r^{ - n} e^{\eta(\bN(\bX_{[0,T]}) + 1)}.
\]
Thus, while the underlying analyticity mechanism is essentially the same as in the proof of \cite[Theorem 6.13]{ChevyrevLyons2016}, we obtain an explicit and quantitative $L^q$-bound for the higher order derivatives $Y_T^{\lambda,(n)}$, and hence a bound for the higher order derivatives of $\phi_{\bX}(\lambda M)$ in $\lambda$.
\end{remark}

\begin{remark}\label{rem: all exponential moments give entire development}
If $\E[e^{\eta\bN(\bX_{[0,T]})}] < \infty$ for every $\eta > 0$, then, for every fixed $k \ge 1$ and $M \in \lin(\R^d,\so(k))$, the function $z \mapsto \E[Y_T^z]$ is entire. Indeed, for any $R > 0$, the factorial signature estimate \eqref{eq: rough path decay} on each greedy interval, whose rough-path $p$-variation is at most $1$, gives a deterministic constant $B_R \ge 1$, depending only on $p,k,R$ and $M$, such that
\[
 \sup_{|z| \le R}\sup_{0 \le i \le \bN}
 \|\widetilde Y_{\tau_{i + 1}}^{z,i}\|_{\mathrm{op}} \le B_R.
\]
Chen's identity and submultiplicativity therefore imply
\[
 \E\!\left[\sup_{|z| \le R}\|Y_T^z\|_{\HS}\right] \le \sqrt{k}\,\E[B_R^{\bN + 1}] < \infty.
\]
Since $z \mapsto Y_T^z$ is entire pathwise, this locally integrable bound permits expectation to pass through contour integrals. Morera's theorem shows that $z \mapsto \E[Y_T^z]$ is entire. Consequently, its Taylor series about every real centre $\lambda_0$ has infinite radius of convergence. This argument uses bounds on arbitrary complex disks, not only the small disks supplied by Proposition \ref{prop: bound on development: greedy sequence}. Note that the same reasoning shows that $\ESig(\bX_{[0,T]})$ has infinite radius of convergence under the above condition, which was already observed in \cite[Theorem 6.3, Remark 6.4, Corollary 6.5]{ChevyrevLyons2016}.
\end{remark}

\begin{example}[Gaussian processes]\label{example: Gaussian processes}
According to \cite[Theorem 6.3]{CassLittererLyons2013}, if $\omega \mapsto X_{[0,T]}(\omega)$ is a centred $\R^d$-valued continuous Gaussian process (with independent components) defined on a deterministic time interval $[0,T]$ with the associated Cameron-Martin space $\cH$ such that the following conditions hold:
\begin{enumerate}
    \item for almost all $\omega \in \Omega$, $X_{[0,T]}(\omega)$ admits a natural lift to a geometric $p$-rough path $\bX_{[0,T]}(\omega) \in G\Omega_p(V)$ for some $1 \le p < 4$.
    \item there is some $q \in [1,2)$ such that $1/p + 1/q > 1$ and the inclusion $\cH \hookrightarrow C^{q\var}([0,T],\R^d)$ is continuous,
\end{enumerate}
then the random variable $\bN(\bX_{[0,T]}(\omega))^{\frac{1}{q}}$ has a Gaussian tail such that $\E[e^{\eta \bN(\bX_{[0,T]})}] < \infty$ for all $\eta > 0$. If $1 \le p \le 2$, choose $p^\prime \in (2,4)$ sufficiently close to $2$ that $1/p^\prime + 1/q > 1$. The canonical extension of $\bX$ can be viewed as a geometric $p^\prime$-rough path with the same signature and developments. Applying the same tail estimate to its $p^\prime$-variation greedy count allows us to use the preceding results with $p^\prime$ in place of $p$. Therefore, Corollary \ref{cor: exponential tail of N} holds for all these Gaussian processes satisfying the above two conditions, and in this case for any $\lambda_0 \in \R$, the Taylor power series of $\lambda \mapsto \phi_\bX(\lambda M)$ at $\lambda_0$ has infinite radius of convergence by Remark \ref{rem: all exponential moments give entire development}, for every $k \ge 1$ and $M \in \lin(\R^d,\so(k))$. For instance, fractional Brownian motion with Hurst parameter $H \in (\frac{1}{4}, \frac12)$ satisfies the above two conditions for any $p \in (\frac{1}{H},4)$ and $q \in ((H + \frac{1}{2})^{ - 1},p/(p - 1))$, see \cite[Corollary 5.5]{CassLittererLyons2013}.
\end{example}

\begin{example}[Markovian rough paths]\label{ex: Markovian rough paths}
Another important class of random continuous geometric $p$-rough paths are the Markovian rough paths $\omega \mapsto \bX^{a,x}_{[0,T]}(\omega)$ (here $T$ is deterministic) generated by Dirichlet form $\mathcal{E}^a$ on $L^2(\fg)$, where $\fg = \fg^{N}(\R^d)$ with $N = \lfloor p \rfloor \ge 2$ is the free nilpotent Lie algebra over $\R^d$ of level $N$, $x \in \fg$ is the initial value with $\bX^{a,x}_0 = x$, and $a \in \Xi^{N,d}(\Lambda)$ with $\Lambda \ge 1$, see \cite[Chapter 16]{Friz2010} for the concrete definitions of the above notions. As pointed out in \cite[Example 6.8]{ChevyrevLyons2016}, see also \cite[Theorem 5.3]{Cass2017}, the random variable $\bN(\bX^{a,x}_{[0,T]})^{1 - \frac{1}{p}}$ has a Gaussian tail for any $p > 2$, which by Corollary \ref{cor: exponential tail of N} implies that the distribution $\mu_{S(\bX^{a,x})_{0,T}}$ of its signature is determined by the expected signature $\ESig(\bX^{a,x}_{[0,T]})$, and in this case for any $\lambda_0 \in \R$, the Taylor power series of $\lambda \mapsto \phi_{\bX^{a,x}_{[0,T]}}(\lambda M)$ at $\lambda_0$ has infinite radius of convergence, for every $k \ge 1$ and $M \in \lin(\R^d,\so(k))$, by Remark \ref{rem: all exponential moments give entire development}.

Further, in \cite[Example 6.20]{ChevyrevLyons2016} the authors showed that for any $r > 0$ and $x \in \fg$, let $B_r(x) : = \{y \in \fg: d_{\text{CC}}(x,y) \le r\}$ (where $d_{\text{CC}}$ denotes the Carnot-Caratheodory metric on $\fg$, see \cite[Chapter 7]{Friz2010}) and define $T(\omega) : = \inf\{t \ge 0: \bX^{a,x}_t(\omega) \notin B_r(x)\}$ be the first exit time of $\bX^{a,x}$ from the ball $B_r(x)$, then the random variable
$$
\bN(\bX^{a,x}_{[0,T]}(\omega)) : = \sup\{j \ge 0: \tau_j(\omega) < T(\omega)\}
$$
has an exponential tail. By Corollary \ref{cor: exponential tail of N} or \cite[Corollary 6.18]{ChevyrevLyons2016} we see that the distribution $\mu_{S(\bX^{a,x})_{0,T}}$ is determined by the expected signature $\ESig(\bX^{a,x}_{[0,T]})$, but in  this case for any $\lambda_0 \in \R$, the Taylor power series of $\lambda \mapsto \phi_{\bX^{a,x}_{[0,T]}}(\lambda M)$ at $\lambda_0$ may only have a finite radius of convergence.
\end{example}

\begin{remark}\label{rem: a counterexample}
  Note that in \eqref{eq: bound on expected development: greedy sequence} and \eqref{eq: upper bounds for greedy seuqence case} the supremum are only taken over a bounded subset $\bK$ instead of the whole $\R$. In fact, for a general geometric $p$-rough path the upper bounds in \eqref{eq: bound on expected development: greedy sequence} and \eqref{eq: upper bounds for greedy seuqence case} cannot be  taken globally over $\lambda \in \R$. Here is a simple counterexample: Consider a ``pure-area'' geometric rough path $\bX = (1, X, \mathbb X)$ defined on $[0,1]$ with $X_t \equiv 0 \in \R^2$ and $\bbX_{s,t} = (t-s)(e_1 \otimes e_2 - e_2 \otimes e_1)$, we pick a $M \in \lin(\R^2, \so(3))$ such that $[M(e_1),M(e_2)]_{\Lie} =: C \neq 0$, e.g., set $M(e_1) = \begin{bmatrix}
      0 & 1 & 0\\
      -1 & 0 & 0\\
      0 & 0 & 0
  \end{bmatrix}$ and $M(e_2) = \begin{bmatrix}
      0 & 0 & 0\\
      0 & 0 & 1\\
      0 & -1 & 0
  \end{bmatrix}$. This rough path $\bX$ can be approximated by smooth paths $X^n$ in the $p$-variation metric (for any $p \in (2,3)$) with $X^n_t = (r_n (\cos(2\pi n t) - 1), r_n\sin(2\pi n t))$, $r_n:=\sqrt{\frac{1}{\pi n }}$. For each $n \ge 1$, the development $Y^{\lambda,X^n}_1$ of $\lambda M$ along $X^n$ satisfies $Y^{\lambda,X^n}_1 = \exp(\lambda^2 C + O_\lambda(\sqrt{1/n}))$ by invoking the Magnus expansion. It follows then that $Y^{\lambda}_1 = \lim_{n \to \infty} Y^{\lambda,X^n}_1 = \exp(\lambda^2 C)$, where $Y^{\lambda}_1$ is the development of $\lambda M$ along $\bX$. Clearly, we have $Y^{\lambda,(1)}_1 = 2\lambda Ce^{\lambda^2 C}$, which gives that $\|Y^{\lambda,(1)}_1\|_{\HS} = 2\|C\|_{\HS}|\lambda|$, and hence $\sup_{\lambda \in \R}\|Y^{\lambda,(1)}_1\|_{\HS} = \infty$.
\end{remark}

\begin{remark}[Time-augmented case]
Let $\bX_{[0,T]}$ be a random continuous geometric $p$-rough path with $p>2$ and a possibly random $T>0$, then by the Young pairing argument (\cite[Sect. 9.4]{Friz2010}) it admits a unique time-augmented extension $\hat \bX_{[0,T]} \in G\Omega_p(\R^{1+d})$ such that its first level projection $\hat X_t = (t, X_t)$, where $X_{[0,T]}$ is the first level of $\bX_{[0,T]}$. Assume that $\E[e^{\eta \bN(\bX_{[0,T]})}]<\infty$ for some $\eta>0$, where $\bN(\bX_{[0,T]})$ is the counts of the greedy sequence of $\bX_{[0,T]}$ related to the increment quantity $1$. By \cite[Theorem 9.26, Exercise 9.21]{Friz2010} there exists a constant $C = C(p)$ such that $\|\hat \bX\|_{p\var,[s,t]} \le C (\|\bX\|_{p\var,[s,t]} + |t-s|)$ for all $(s,t) \in \overline{\Delta^2_T}$. This implies that the count $\hat \bN_{2C}$ of  the greedy sequence of $\hat \bX_{[0,T]}$ related to the increment quantity $2C$ satisfies that $\hat \bN_{2C} \le \bN(\bX_{[0,T]}) + C(1 + \frac{T}{(2C)^{1/p}})$. So, if in addition there exists a positive $\kappa > 0$ such that $\E[e^{\kappa T}]<\infty$ (e.g., when $T$ is a deterministic real number), then we have $\E[e^{c\hat \bN_{2C}}]<\infty $ for some $c>0$ and the same reasoning as used in the proofs of Lemma \ref{lem:uniform-greedy-development} and Proposition \ref{prop: bound on development: greedy sequence} applied to $\hat \bX_{[0,T]}$ and $\hat \bN_{2C}$ provides that Corollary \ref{cor: exponential tail of N} remains valid for time-augmented process $\hat \bX_{[0,T]}$ by using $\hat M \in \bigsqcup_{k=1}^\infty \lin(\R^{1+d}, \so(k))$ with a possibly different radius $r > 0$. Note that for random continuous geometric $p$-rough paths $\bX_{[0,T]}$ and $\bY_{[0,T]}$ it holds that $\mu_{\bX_{[0,T]}} = \mu_{\bY_{[0,T]}}$ iff $\phi_{\hat \bX_{[0,T]}} = \phi_{\hat \bY_{[0,T]}}$, see Remark \ref{rem: time-augmented paths}.
\end{remark}

\subsection{$X_{[0, T]}$ is a continuous  semimartingale}\label{subsect: semimart}

In this section, we will focus on the case that $X_{[0,T]}(\omega)$ is a continuous semimartingale. Classical rough path theory tells us that $X_{[0,T]}(\omega)$ has finite $p$-variation for all $p \in (2,3)$ and its Stratonovich lift
$$
\bX_{[0,T]}(\omega) = \bigg(X_{[0,T]}(\omega), \int X \otimes (\circ \dd) X (\omega) \bigg)
$$
belongs to $G\Omega_p(V)$ for almost all $\omega \in \Omega$, see \cite[Chapter 14]{Friz2010}. Hence, we can and will identify $X_{[0,T]}(\omega)$ with its Stratonovich lift $\bX_{[0,T]}(\omega)$. It is also noteworthy that for almost all $\omega \in \Omega$, the solution to the (pathwise) RDE $\dd Y^\lambda_t(\omega) = Y^\lambda_t(\omega) (\lambda M)(\dd \bX(\omega)_t)$ coincides with the solution to the following Stratonovich stochastic differential equation (SDE) evaluated at $\omega$:
\begin{equation}\label{eq: SDE}
    dY^\lambda_t = Y^\lambda_t  (\lambda M)(\circ \dd X_t), \quad Y^\lambda_0 = I_k.
\end{equation}
and for any progressively measurable and $X$-integrable process $H(\omega)_{[0,T]}$ such that for almost all $\omega$ it is also an $\bX(\omega)$-controlled path, then the (pathwise) rough integral $\int_0^t H_s(\omega) \dd \bX(\omega)_s$ coincides with the Stratonovich integral $\int_0^t H_s \circ \dd X_s$ evaluated at $\omega$, see e.g. \cite[Lemma 4.35]{Chevyrev2019}. Therefore, in this section, we will only use the Stratonovich type notation and identify $\dd \bX_t$ with $\circ \dd X_t$. Note that since $M \in \lin(\R^d, \so(k))$ is linear, the SDE \eqref{eq: SDE} can also be written as
$$
dY^\lambda_t = \sum_{i = 1}^d Y^\lambda_t  (\lambda M)(e_i)(\circ \dd X^i_t), \quad Y^\lambda_0 = I_k,
$$
for $e_1,\ldots,e_d$ being the canonical basis of $\R^d$. Moreover, in this case we have \begin{align*}
        Y^{\lambda,(n)}_t = n! \int_{\Delta^n_t} \Ad_{Y^\lambda_{s_1}}(M( \circ \dd  X_{s_1}))  \cdots \Ad_{Y^\lambda_{s_n}} (M(\circ \dd  X_{s_{n}}))Y^\lambda_t,
    \end{align*}
where $\int_{\Delta^n_t} \Ad_{Y^\lambda_{s_1}}(M( \circ \dd  X_{s_1}))  \cdots \Ad_{Y^\lambda_{s_n}} (M(\circ \dd  X_{s_{n}}))$ is the iterated Stratonovich integral concatenated by matrix products.

 As in \eqref{eq: hat X}, we define $\hat X^\lambda_t = \int_0^t \Ad_{Y^\lambda_{s}}(M( \circ \dd  X_{s})) = \int_0^t Y^\lambda_s M(\circ \dd X_s) (Y^\lambda_s)^\dagger$. After a careful calculation we can check that
\begin{align}\label{eq: hat X is ito}
    \hat X^\lambda_t & = \int_0^t Y^\lambda_s M(\dd X_s) (Y^\lambda_s)^\dagger + \frac{\lambda}{2}\int_0^t Y^\lambda_s \dd[M(X),M(X)]_s(Y^\lambda_s)^\dagger \nonumber\\
&\quad + \frac{1}{2}\bigg[\int_0^\cdot Y^\lambda_s M(\dd X_s), (Y^\lambda_\cdot)^\dagger\bigg]_t \nonumber\\
& = \int_0^t Y^\lambda_s M(\dd X_s) (Y^\lambda_s)^\dagger + \frac{\lambda}{2}\int_0^t Y^\lambda_s \dd[M(X),M(X)]_s(Y^\lambda_s)^\dagger \nonumber\\
&\quad + \frac{\lambda}{2}\int_0^t Y^\lambda_s \dd[M(X), M(X)^\dagger]_s (Y^\lambda_s)^\dagger \nonumber\\
& = \int_0^t Y^\lambda_s M(\dd X_s) (Y^\lambda_s)^\dagger,
\end{align}
because $M(X) + M(X)^\dagger = 0$ as $M(X) \in \so(k)$. Note that in the above computations we used lots of computation rules for matrix-valued semimartingales, for example here the quadratic variation $[M(X),M(X)]_\cdot$ is also a matrix and obtained by matrix multiplication. For concrete details on stochastic calculus for matrix-valued semimartingales we refer readers to Appendix \ref{sect: matrix stochastic calculus}.

Now, similarly as in \eqref{eq: recursive formula for matrix signature}, we define recursively that
 \begin{align}\label{eq: recursive formula for matrix signature in martingale case}
 &D_0(\hat X^\lambda_t) = I_k, \quad D_1(\hat X^\lambda_t) = \hat X^\lambda_t, \quad \nonumber \\
 &D_{n + 1}(\hat X^\lambda_t) = \int_0^t D_{n}(\hat X^\lambda_s) \circ \dd \hat X^\lambda_s = \int_{\Delta^{n + 1}_t} \circ \dd\hat X^\lambda_{s_1} \ldots  \circ \dd \hat X^\lambda_{s_{n + 1}},
\end{align}
and call them the matrix signature of $\hat X^\lambda_{[0,T]}$. Note that the concatenation is still the matrix multiplication and for notational simplicity we omit the product ``$\cdot$'' in the above expression. Also keep in mind that for all $n \ge 0$ and all $t \in [0,T]$, we have $\frac{1}{n!}Y^{\lambda,(n)}_t = D_n(\hat X^\lambda_t)Y^\lambda_t$ and hence $\frac{1}{n!}\|Y^{\lambda,(n)}_t\|_{\HS} = \|D_n(\hat X^\lambda_t)\|_{\HS}$.

The next lemma is crucial for us to get a proper upper bound for the decay rate of the expected matrix signature of $\hat X^\lambda_{[0,T]}$.

\begin{lemma}\label{lem:L2 norm of hat X}
Assume that $X_{[0,T]}(\omega)$ is an $\R^d$-valued continuous local martingale. Then for all $t \in [0,T]$, $\lambda \in \R$ and $M \in \lin(\R^d,\so(k))$, we have $\llbracket \hat X^\lambda,\hat X^\lambda \rrbracket_t = \llbracket M(X), M(X)\rrbracket_t$. If $\E[\llbracket X,X \rrbracket_T] < \infty$, then
   \begin{align}\label{eq: quadratic variation of hat X}
    \E[\|\hat X^\lambda_t\|_{\HS}^2] = \E\bigg\|\int_0^t Y^\lambda_s  M(\dd X_s)  (Y^\lambda_s)^\dagger\bigg\|_{\HS}^2 = \E[\llbracket M(X), M(X)\rrbracket_t],
\end{align}
and consequently \begin{equation}\label{eq: bound for the norm of the first derivative in local martingale case}
    \E[\|Y^{\lambda,(1)}_t\|_{\HS}^2] \le \|M\|_{\op}^2 \E[\llbracket X,X \rrbracket_t].
\end{equation}
\end{lemma}

\begin{proof}
In view of \eqref{eq: hat X is ito}, we have $\hat X^\lambda_t = \int_0^t Y^\lambda_s  M(\dd X_s)  (Y^\lambda_s)^\dagger$. As a consequence we have (for simplicity, below we write $Y^{ij}$ for $(Y^{\lambda})^{ij}$)
\begin{align*}
      \|\hat X^\lambda_t\|_{\HS}^2 = \sum_{i,j = 1}^k\sum_{p,q = 1}^k\sum_{p^\prime, q^\prime = 1}^k \int_0^t Y^{ip}_sM_{pq}(\dd X_s)Y^{jq}_s \int_0^t  Y^{ip^\prime}_sM_{p^\prime q^\prime}(\dd X_s)Y^{jq^\prime}_s,
    \end{align*}
which further implies that
\begin{align*}
   \E[\|\hat X^\lambda_t\|_{\HS}^2]
        & = \sum_{i,j,p,q,p^\prime,q^\prime = 1}^k \E\bigg[\int_0^t Y^{ip}_s Y^{jq}_s Y^{ip^\prime}_s Y^{jq^\prime}_s d[M_{pq}(X), M_{p^\prime q^\prime}(X)]_s\bigg]\\
& = \sum_{p,q = 1}^k\sum_{p^\prime,q^\prime = 1}^k\E\bigg[      \sum_{i,j = 1}^k \int_0^t Y^{ip}_s Y^{ip^\prime}_s Y^{jq}_s  Y^{jq^\prime}_s d[M_{pq}(X), M_{p^\prime q^\prime}(X)]_s\bigg].
\end{align*}
Since $Y^\lambda_\cdot$ is an orthogonal matrix, we have
$
\sum_{i = 1}^k Y^{ip}_sY^{ip^\prime}_s = \delta_{pp^\prime}$ and $\sum_{j = 1}^k Y^{jq}_sY^{jq^\prime}_s = \delta_{qq^\prime},
$
which in turn implies that
\begin{align*}
 \E[\|\hat X^\lambda_t\|_{\HS}^2] = \sum_{p,q = 1}^k \E[[M_{pq}(X), M_{pq}(X)]_t] = \E[\llbracket M(X), M(X)\rrbracket_t],
\end{align*}
which is \eqref{eq: quadratic variation of hat X}. Independently of the integrability assumption, summing the entrywise covariations computed above and using orthogonality gives $\llbracket \hat X^\lambda,\hat X^\lambda \rrbracket_t = \sum_{p,q = 1}^k [M_{pq}(X),M_{pq}(X)]_t = \llbracket M(X), M(X)\rrbracket_t$. This together with Lemma \ref{lemma: bound on the quadratic variation of M(X)} gives us \eqref{eq: bound for the norm of the first derivative in local martingale case}
\end{proof}

The equation \eqref{eq: quadratic variation of hat X} reveals an important fact that if $X_{[0,T]}$ is a continuous local martingale, then the $L^2(\HS)$-norm of the first level of its matrix signature $\hat X^\lambda_t$ (and therefore the $L^2(\HS)$-norm of the first radial derivative of its unitary development $Y^{\lambda,(1)}_t$) is independent of $\lambda \in \R$. Thanks to the ``iterated integrals'' structure of matrix signature (see \eqref{eq: recursive formula for matrix signature in martingale case}) we can use this fact to show that there exists a nice upper bounds $B_2(n)$, which is globally uniform in $\lambda \in \R$, for the $L^2(\HS)$-norm of the $n$-th level of its matrix signature $D_n(\hat X^\lambda_t)$ (and hence, the $L^2(\HS)$-norm of the $n$-th radial derivative of its unitary development $Y^{\lambda,(n)}_t$) for any $n \ge 1$ inductively. This is one of our main strategies for handling semimartingales, see the proofs of Theorem \ref{cor: Ito martingale case}  and Theorem \ref{thm: bound for matrix signature:stopped BM} below.

\subsubsection{Semimartingales on a deterministic time horizon}\label{subsect: semimartingale on deterministic time}

In this subsection we will investigate the $L^q$-bounds (in particular, for $q = 1,2$) for all $n$-th radial derivatives of the unitary development $Y^\lambda_t(\omega)$  for some continuous semimartingales $X_{[0,T]}(\omega)$ defined on a deterministic time horizon $[0,T]$.

Starting with a simple case, we will establish $L^2(\HS)$ bounds for certain class of It\^o processes $X_t = x + \int_0^t b_s \dd s + \int_0^t H_s \dd W_s, t \in [0,T]$, where $x \in \R^d$, $W$ is an $\R^e$-valued Brownian motion and the process $\omega \mapsto (b_{[0,T]}(\omega), H_{[0,T]}(\omega))$ is a progressively measurable  with sufficient integrability.

\begin{theorem}\label{cor: Ito martingale case}
Assume that process $X_t = x + \int_0^t b_s \dd s + \int_0^t H_s \dd W_s, t \in [0,T]$ is defined on some filtered probability space $(\Omega, \bbF, \bbP)$, where $W$ is an $\R^m$-valued Brownian motion and $(b_t,H_t)_{t \in [0,T]}$ is a progressively measurable process taking values in $\R^d \times \lin(\R^m,\R^d)$ such that $\eta_T : = \int_0^T(\|b_s\|_{L^\infty} + \|H_s\|_{L^\infty}^2)\dd s < \infty$. Here the $L^\infty$-norms use the Euclidean norm for $b_s$ and the operator norm for $H_s$. Then for all $k \ge 1$, $M \in \lin(\R^d,\so(k))$ and $n \geq 1$, $t \in [0,T]$ and for all $\lambda \in \R$,  it holds that
\begin{equation}\label{eq: bound for matrix signature of Ito martingales}
    \frac{1}{n!}\| Y^{\lambda,(n)}_t \|_{L^{2}} \leq \frac{4^n (C(||M||_{\op} \vee 1)^2)^n \eta_t^{\frac{n}{2}}(1 + \eta_t)^{\frac{n}{2}}}{\sqrt{n!}},
\end{equation}
where $C = C(k,d) \ge 1$ is a constant only depending on $k,d$ and $\eta_t = \int_0^ t(\|b_s\|_{L^\infty} + \|H_s\|^2_{L^\infty}) \dd s$.
\end{theorem}

\begin{proof}
Without loss of generality we can and will assume that $\|M\|_{\op} \ge 1$ (Indeed, the case $M = 0$ is immediate. If $0 < a : = \|M\|_{\op} < 1$, then $\lambda M = (a\lambda)(M/a)$, so the chain rule introduces the factor $a^n \le 1$ in the $n$-th derivative and reduces the claim to the case $\|M/a\|_{\op} = 1$) and write $X_t = x + A_t + N_t$ with $A_t : = \int_0^t b_s \dd s$ of bounded variation and $N_t : = \int_0^t H_s \dd W_s$ being a continuous martingale. Since $\frac{1}{n!}Y^{\lambda,(n)}_t = D_n(\hat X^\lambda_t)Y^\lambda_t$, it suffices to show that $\|D_n(\hat X^\lambda_t)\|_{L^2}$ satisfies the upper bound \eqref{eq: bound for matrix signature of Ito martingales}. As before, we will apply an induction on $n \ge 1$. Recall that $\hat X^\lambda_t = \int_0^t Y^\lambda_s M(\dd X_s)(Y^\lambda_s)^\dagger$, so in the current case $\hat X^\lambda_{[0,T]}$ is a continuous semimartingale with the canonical decomposition
$$
\hat X^\lambda_t = \hat A^\lambda_t + \hat N^\lambda_t, \quad \hat A^\lambda_t: = \int_0^t Y^\lambda_s M(\dd A_s)(Y^\lambda_s)^\dagger, \quad \hat N^\lambda_t = \int_0^t Y^\lambda_s M(\dd N_s)(Y^\lambda_s)^\dagger.
$$
For $n = 1$, we have $\|D_1(\hat X^\lambda_t)\|_{L^2} = \|\hat X^\lambda_t\|_{L^2} \le \|\hat A^\lambda_t\|_{L^2} + \|\hat N^\lambda_t\|_{L^2}$. Using Minkowski's inequality and the simple fact that $\|Y^\lambda_sM(b_s)(Y^\lambda_s)^\dagger\|_{L^2} \le \|M\|_{\op}\|b_s\|_{L^\infty}$ for all $s \in [0,t]$, we can easily check that
\begin{align*}
\|\hat A^\lambda_t\|_{L^2} = \bigg\| \int_0^t Y^\lambda_s M(b_s)(Y^\lambda_s)^\dagger \dd s \bigg\|_{L^2} & \le \int_0^t \|Y^\lambda_s M(b_s)(Y^\lambda_s)^\dagger\|_{L^2} \dd s \\
& \le \|M\|_{\op}\eta_t.
\end{align*}
On the other hand, using \eqref{eq: bound for the norm of the first derivative in local martingale case}, we have
$$
\|\hat N^\lambda_t\|_{L^2} \le \|M\|_{\op} \sqrt{\E[ \llbracket N,N \rrbracket_t ]} \le \sqrt{d}\|M\|_{\op}\sqrt{\eta_t}.
$$
Combining the above two upper bounds we get
$$
\|D_1(\hat X^\lambda_t)\|_{L^2} \le \|M\|_{\op}\sqrt{\eta_t}(\sqrt{d} + \sqrt{\eta_t}) \le 2C\|M\|_{\op}^2\sqrt{\eta_t}\sqrt{1 + \eta_t},
$$
where $C = C(k,d)$ from Lemma \ref{lem: bound on quad variation of hat X when X is BM} such that $C \ge d$. Hence the bound \eqref{eq: bound for matrix signature of Ito martingales} holds for $n = 1$.

Now assume that the bound \eqref{eq: bound for matrix signature of Ito martingales} holds for $D_m(\hat X^\lambda_t)$ for all $1 \le m \le n$ with some $n \ge 1$, and we consider the case $n + 1$. Using the Itô-Stratonovich correction \eqref{eq: ito-strat correction for matrix}, we have
\[
D_{n + 1}(\hat{X}^{\lambda}_{t}) = \int_0^t
D_{n}(\hat{X}^{\lambda}_{s})\dd \hat{X}_s^\lambda + \frac12
\int_0^t
\dd\bigl[ D_{n}(\hat{X}^{\lambda}_{\cdot}),\hat{X}^\lambda\bigr]_s,
\]
and then inserting the canonical decomposition $\hat X^\lambda_t = \hat A^\lambda_t + \hat N^\lambda_t$, we can further derive that
$$
D_{n + 1}(\hat{X}^{\lambda}_{t}) = \int_0^t
D_{n}(\hat{X}^{\lambda}_{s})\dd \hat{A}_s^\lambda + \int_0^t
D_{n}(\hat{X}^{\lambda}_{s})\dd \hat{N}_s^\lambda + \frac12
\int_0^t
\dd\bigl[ D_{n}(\hat{X}^{\lambda}_{\cdot}),\hat{N}^\lambda\bigr]_s
$$
and therefore
\begin{align*}
  \|D_{n + 1}(\hat{X}^{\lambda}_{t})\|_{L^2} & \le \underbrace{\bigg\|\int_0^t
D_{n}(\hat{X}^{\lambda}_{s})\dd\hat{A}_s^\lambda\bigg\|_{L^2}}_{: = I} + \underbrace{\bigg\|\int_0^t
D_{n}(\hat{X}^{\lambda}_{s})\dd\hat{N}_s^\lambda\bigg\|_{L^2}}_{: = II} \\
&\quad + \underbrace{\frac12
\bigg\|\int_0^t
d\bigl[ D_{n}(\hat{X}^{\lambda}_{\cdot}),\hat{N}^\lambda\bigr]_s\bigg\|_{L^2}}_{: = III}.
\end{align*}
For the first term $(I)$, we notice that
\begin{align*}
 \bigg\|\int_0^t
D_{n}(\hat{X}^{\lambda}_{s})\dd\hat{A}_s^\lambda\bigg\|_{L^2} & = \bigg\|\int_0^t
D_{n}(\hat{X}^{\lambda}_{s}) Y^\lambda_sM(b_s)(Y^\lambda_s)^\dagger\dd s\bigg\|_{L^2} \\
& \le \int_0^t \|D_{n}(\hat{X}^{\lambda}_{s}) Y^\lambda_sM(b_s)(Y^\lambda_s)^\dagger \|_{L^2} \dd s\\
& \le \|M\|_{\op}  \int_0^t \|D_{n}(\hat{X}^{\lambda}_{s}) \|_{L^2} \|b_s\|_{L^\infty} \dd s\\
& \leq \|M\|_{\op}  \int_0^t \|D_{n}(\hat{X}^{\lambda}_{s}) \|_{L^2}  \dd \eta_s.
\end{align*}
Consequently, by the induction hypothesis on $\|D_{n}(\hat{X}^{\lambda}_{s}) \|_{L^2}$, we get that (as $(1 + \eta_s)^{\frac{n}{2}} \le (1 + \eta_t)^{\frac{n}{2}}$ for all $s \le t$)
\begin{align*}
     \bigg\|\int_0^t
D_{n}(\hat{X}^{\lambda}_{s})\dd\hat{A}_s^\lambda\bigg\|_{L^2} & \le \|M\|_{\op}\frac{4^n(C\|M\|_{\op}^2)^n}{\sqrt{n!}}\int_0^t \eta_s^{\frac{n}{2}}(1 + \eta_s)^{\frac{n}{2}} \dd \eta_s\\
& \le \|M\|_{\op}\frac{4^n(C\|M\|_{\op}^2)^n}{\sqrt{n!}} \frac{2\eta_t^{\frac{n}{2} + 1}}{n + 2}(1 + \eta_t)^{\frac{n}{2}}\\
& \le 2\frac{4^{n}(C\|M\|_{\op}^2)^{n + 1}\eta_t^{\frac{n + 1}{2}}}{\sqrt{(n + 1)!}} (1 + \eta_t)^{\frac{n}{2}}\sqrt{\eta_t} \\
& \le 2\frac{4^{n}(C\|M\|_{\op}^2)^{n + 1}\eta_t^{\frac{n + 1}{2}}}{\sqrt{(n + 1)!}} (1 + \eta_t)^{\frac{n + 1}{2}}
\end{align*}
For the second term $(II)$, using the It\^o isometry \eqref{eq: matrix Ito isometry}, we can apply the induction hypothesis to check that
\begin{align*}
 \bigg\|\int_0^t
D_{n}(\hat{X}^{\lambda}_{s})\dd\hat{N}_s^\lambda\bigg\|_{L^2}    & \le C\|M\|_{\op}\sqrt{\int_0^t \|D_n(\hat X^\lambda_s) \|_{L^2}^2 \|H_s\|^2_{L^\infty}\dd s} \\
& \le C\|M\|_{\op}\frac{4^n(C\|M\|_{\op}^2)^n}{\sqrt{n!}}  \sqrt{\int_0^t \eta_s^n(1 + \eta_s)^n \dd \eta_s}\\
& \le \frac{4^{n}(C\|M\|_{\op}^2)^{n + 1}\eta_t^{\frac{n + 1}{2}}}{\sqrt{(n + 1)!}} (1 + \eta_t)^{\frac{n + 1}{2}}.
\end{align*}
For the last term $(III)$, we have
\begin{align*}
    \frac12
\bigg\|\int_0^t
d\bigl[ D_{n}(\hat{X}^{\lambda}_{\cdot}),\hat{N}^\lambda\bigr]_s\bigg\|_{L^2} = \frac12\bigg\|\int_0^t D_{n - 1}(\hat X^\lambda_s) \dd [\hat N^\lambda, \hat N^\lambda]_s
\bigg\|_{L^2}
\end{align*}
For $n = 1$, the quadratic-correction term can be estimated directly. Writing $e_1,\ldots,e_e$ for the canonical basis of $\R^e$, its density satisfies
$$
\bigg\|\frac{\dd[\hat N^\lambda,\hat N^\lambda]_s}{\dd s}\bigg\|_{\HS} \le \sum_{r = 1}^e \|M(H_se_r)\|_{\HS}^2 \le d\|M\|_{\op}^2\|H_s\|_{\op}^2.
$$
Consequently, $III \le \frac{d}{2}\|M\|_{\op}^2\eta_t \le C\|M\|_{\op}^2\eta_t(1 + \eta_t)/\sqrt{2}$, which is the bound below for $n = 1$. For $n \ge 2$, using Lemma \ref{lem: bound on quad variation of hat X when X is BM}, we can readily check that
\begin{align*}
 \frac12\bigg\|\int_0^t D_{n - 1}(\hat X^\lambda_s) \dd [\hat N^\lambda, \hat N^\lambda]_s\bigg\|_{L^2} & \le \frac12 C\|M\|_{\op}^2\int_0^t \| D_{n - 1}(\hat X^\lambda_s) \|_{L^2} \dd \eta_s \\
 & \le \frac12 4^{n - 1} \frac{(C\|M\|^2_{\op})^n}{\sqrt{(n - 1)!}} \int_0^t \eta_s^{\frac{n - 1}{2}}(1 + \eta_s)^{\frac{n - 1}{2}} \dd \eta_s \\
 & \le 4^{n - 1} \frac{(C\|M\|^2_{\op})^n}{\sqrt{(n + 1)!}} \eta_t^{\frac{n + 1}{2}}(1 + \eta_t)^{\frac{n + 1}{2}}.
\end{align*}
From all above we finally get that
\begin{align*}
\|D_{n + 1}(\hat{X}^{\lambda}_{t})\|_{L^2} & \le \frac{(C\|M\|^2_{\op})^{n + 1}}{\sqrt{(n + 1)!}} \eta_t^{\frac{n + 1}{2}}(1 + \eta_t)^{\frac{n + 1}{2}} (2 \times 4^n + 4^n + 4^{n - 1})\\
& \le 4^{n + 1}\frac{(C\|M\|^2_{\op})^{n + 1}}{\sqrt{(n + 1)!}} \eta_t^{\frac{n + 1}{2}}(1 + \eta_t)^{\frac{n + 1}{2}},
\end{align*}
which completes the induction proof of \eqref{eq: bound for matrix signature of Ito martingales}.
\end{proof}

It is easy to see that the idea of the proof of the above theorem  remains valid for any continuous local martingale with deterministic initial value and deterministic quadratic variation, or equivalently speaking, Gaussian local martingales with deterministic initial value (see e.g. \cite{YorRevuz1999}). The proof is very similar to the proof for Theorem \ref{cor: Ito martingale case}, so we omit the concrete details here.

\begin{corollary}\label{cor: Gaussian martingale}
Let $X = (X_t)_{t \in [0,T]}$ be an $\R^d$-valued continuous Gaussian local martingale with deterministic initial value and $[X^i,X^j]_t = q^{ij}_t$ for $i,j = 1,\ldots,d$ and $t \in [0,T]$, where each $q^{ij}_{[0,T]}$ is a deterministic function with bounded variation.  Then for all $k \ge 1$, $M \in \lin(\R^d,\so(k))$, $n \geq 1$, $t \in [0,T]$ and $\lambda \in \R$, let $Y^\lambda_{[0,T]}$ be the random development path induced by $\lambda M$ and $X$, it holds that
\begin{equation*}
    \frac{1}{n!}\| Y^{\lambda,(n)}_t \|_{L^2} = \frac{1}{n!}\E[\|Y^{\lambda,(n)}_t\|^2_{\HS}]^{\frac{1}{2}} \leq \frac{2^n (C(||M||_{\op} \vee 1)^2)^n \llbracket X,X \rrbracket_t^{\frac{n}{2}}}{\sqrt{n!}},
\end{equation*}
where $C = C(k,d) \ge 1$ is a constant only depending on the dimensions $k$ and $d$ and $\llbracket X, X \rrbracket_t = \sum_{i = 1}^d [X^i,X^i]_t = \sum_{i = 1}^d q^{ii}_t$.
\end{corollary}

\begin{remark}\label{remark: ito semimart has infinite ROC}
Let $|\cdot|_{2}$ denote the usual $\ell^2$-Euclidean/Hilbert norm on $V^{ \otimes n}$ for any $n \ge 0$, and set $f_n(t) := (\E[|S(X)^n_{0,t}|_{2}^2])^{1/2}$. The tensor-valued It\^o--Stratonovich recursion gives, for any continuous semimartingale $X_{[0,T]}$ and $n \ge 1$,
$$
S(X)^{n + 1}_{0,t} = \sum_{i = 1}^d\int_0^t S(X)^n_{0,s} \otimes e_i\,\dd X_s^i + \sum_{i,j=1}^d\frac12\int_0^t S(X)^{n - 1}_{0,s} \otimes (e_i \otimes e_j)\dd [X^i,X^j]_s.
$$
For a Gaussian martingale as in Corollary \ref{cor: Gaussian martingale}, put $a_t := \llbracket X,X\rrbracket_t$. The deterministic covariance measure is positive semidefinite and its trace is $\dd a_t$. The tensor It\^o isometry and the variation bound for this covariance measure give $f_{n + 1}(t) \le (\int_0^t f_n(s)^2\dd a_s)^{1/2} + \frac12\int_0^t f_{n - 1}(s)\dd a_s$ for $n \ge 1$. With $f_0(t) = 1$ and $f_1(t) = \sqrt{a_t}$, induction yields $f_n(t) \le 2^na_t^{n/2}/\sqrt{n!}$ for every $n \ge 0$.

For an It\^o process as in Theorem \ref{cor: Ito martingale case}, the corresponding recursion gives
$$
f_{n + 1}(t) \le \int_0^t f_n(s)\dd\eta_s + \sqrt{d}\bigg(\int_0^t f_n(s)^2\dd\eta_s\bigg)^{1/2} + \frac{d}{2}\int_0^t f_{n - 1}(s)\dd\eta_s, \qquad n \ge 1.
$$
Here we used $\|H_s\|_{\mathrm{HS}}^2 \le d\|H_s\|_{\op}^2$ and $\|H_sH_s^\dagger\|_{\mathrm{HS}} \le \|H_s\|_{\mathrm{HS}}^2$. Starting from $f_0(t) = 1$ and $f_1(t) \le \eta_t + \sqrt{d\eta_t}$, induction yields $f_n(t) \le (4d)^n[\eta_t(1 + \eta_t)]^{n/2}/\sqrt{n!}$. The projective-to-Hilbert tensor norm comparison (that is, $\|\cdot\|_{\otimes n} \le d^{\frac{n}{2}}|\cdot|_2$) therefore gives
$$
\|\E[S(X)^n_{0,t}]\|_{ \otimes n} \le \|S(X)^n_{0,t}\|_{L^2} \le \frac{(4d^{3/2})^n[\eta_t(1 + \eta_t)]^{n/2}}{\sqrt{n!}}.
$$
These tensor estimates show that the expected signatures have infinite radius of convergence. Separately, the bounds for $Y^{\lambda,(n)}_t$ and Theorem \ref{thm: random_func} show that the Taylor power series of $\lambda \mapsto \phi_{X_{[0,T]}}(\lambda M)$ at every real centre has infinite radius of convergence.
\end{remark}

For some continuous semimartingales with unbounded drifts, the $L^2$-bounds for derivatives of unitary developments can have the similar decay rate as in \eqref{eq: bound for matrix signature of Ito martingales}. As an example, we now
consider a $d$-dimensional Ornstein-Uhlenbeck (OU) process
\begin{equation}\label{eq: OU process}
\dd X_t = K( m - X_t) \dd t + \Sigma \dd W_t, \quad  t \in [0,T],
\end{equation}
where $K, \Sigma$ are real $d \times d$-matrices, $m \in \R^d$, $W_{[0,T]}(\omega)$ is a $d$-dimensional Brownian motion defined on some filtered probability space $(\Omega, \bbF, \bbP)$, and $X_0$ is independent of $W_{[0,T]}(\omega)$ with a Gaussian distribution. Moreover, we fix a linear mapping $M \in \lin(\R^d, \so(k))$ for some $k \ge 1$. As usual, for a $\lambda \in \R$, we use $Y^\lambda_{[0,T]}(\omega)$ to denote the development process of the above OU process $X_{[0,T]}(\omega)$ with respect to $M$:
$$
\dd Y^\lambda_t = Y^\lambda_t (\lambda M)(\circ \dd X_t), \quad Y^\lambda_0 = I_k.
$$

\begin{theorem}\label{thm: bounds for OU}
There exist constants $A$ and $B$ which only depends on $K$, $m$, $\Sigma$, $X_0$, $k$, $\|M\|_{\op}$ and $T$ such that for all $n \ge 1$, it holds that for all $t \in [0,T]$
\begin{equation}\label{eq: bounds for OU}
    \sup_{\lambda \in \R} \frac{1}{n!}\|Y^{\lambda,(n)}_t\|_{L^2} \le A \frac{B^n}{\sqrt{n!}}.
\end{equation}
\end{theorem}

\begin{proof}
For $t \in [0,T]$ we define $B_t := M(K(m - X_t))$ and for $r=1,\ldots,d$, we set $A_r := M(\Sigma e_r)$ with $e_r, r= 1,\ldots,d$ being the canonical basis for $\R^d$. Note that these matrices are skew-Hermitian. For $z = u + \mathrm i v \in \C$ we consider the development $Y^z_{[0,T]}$ such that
$$
\dd Y^z_t = Y^z_t (zM)(\circ \dd X_t), \quad Y^z_0 = I_k.
$$
One can easily check that $Y^z_{[0,T]}$ satisfies the following It\^o dynamics:
$$
\dd Y^z_t = Y^z_t\bigg(zB_t + \frac{z^2}{2}\sum_{r=1}^d A_r^2\bigg) \dd t + z \sum_{r=1}^d Y^z_tA_r \dd W^r_t.
$$
For simplicity, below we will write $Y_t$ for $Y^z_t$ and define $F_t := \|Y_t\|_{\HS}^2 = \Tr(Y_tY^\dagger_t)$. By the product rule, we have
$$
\dd F_t = \Tr(  (\dd Y_t)Y^\dagger_t + Y_t (\dd Y^\dagger_t) + \dd[Y,Y^\dagger]_t), \quad F_0 = k.
$$
Inserting the It\^o dynamics of $\dd Y_t$ and $\dd Y^\dagger_t$ into the above expression of $\dd F_t$ and after some tedious computation, we can obtain that
\begin{align*}
  \dd F_t &= (z - \bar z) \Tr(Y_t B_t Y^\dagger_t) \dd t   + \bigg(\frac{z^2 + \bar z^2}{2} - |z|^2\bigg)\sum_{r=1}^d \Tr(Y_t A^2_r Y^\dagger _t) \dd t \\
  & \quad \quad \quad \quad \quad \quad \quad \quad \quad \quad + \sum_{r=1}^d (z- \bar z) \Tr(Y_t A_r Y^\dagger_t) \dd W^r_t\\
  &=2\Re(z \Tr(Y_t B_t Y^\dagger_t)) \dd t + 2v^2\sum_{r=1}^d \|Y_tA_r\|^2_{\HS} \dd t + 2 \sum_{r=1}^d \Re(z\Tr(Y_t A_r Y^\dagger_t)) \dd W^r_t
  \end{align*}
because $Y_tB_tY^\dagger_t$ and $Y_tA_rY^\dagger_t$ are skew-Hermitian and consequently their traces are pure imaginary.

Now by setting $a^2 := \sum_{r=1}^d \|A_r\|_{\op}^2$ and $b_t := \|B_t\|_{\op}$, we indeed have $|\Re(z \Tr(Y_t B_t Y^\dagger_t))| \le |v|\|B_t\|_{\op}F_t$ and $\sum_{r=1}^d \|Y_tA_r\|^2_{\HS} \le a^2 F_t$, which in turn implies that if we write
$$
\dd F_t = F_t \beta_t \dd t + F_t \sum_{r=1}^d \eta^r_{t} \dd W^r_t
$$
with $\beta_t := \frac{2\Re(z \Tr(Y_t B_t Y^\dagger_t)) + 2v^2\sum_{r=1}^d \|Y_tA_r\|^2_{\HS}}{F_t}$ and $\eta^r_{t} := 2\frac{\Re(z\Tr(Y_t A_r Y^\dagger_t))}{F_t}$, then
$$\beta_t \le 2|v|b_t + 2v^2a^2, \quad \sum_{r=1}^d |\eta^r_t|^2 \le 4v^2a^2.
$$
Here we note that since $M(\cdot)$ is traceless, the Stratonovich chain rule applied to the determinant yields $\det(Y_t) = 1$ for all $t \in [0,T]$, which ensures that $Y_t$ is invertible and $F_t > 0$ for all $t \in [0,T]$.

Next, we apply the It\^o's formula to $\log F_t$ to get
\begin{align*}
    \dd \log F_t &= \frac{1}{F_t} \dd F_t - \frac{1}{2F_t^2} \dd [F,F]_t\\
    &=\beta_t \dd t + \sum_{r=1}^d \eta^r_t \dd W^r_t - \frac{1}{2} \sum_{r=1}^d |\eta^r_t|^2 \dd t\\
    &\le (2|v|b_t + 2v^2a^2) \dd t + \dd L_t,
\end{align*}
where $L_t := \sum_{r=1}^d \int_0^t \eta^r_s \dd W^r_s$. Hence, let $V:= \int_0^T b_t \dd t$, we have shown that
$$
F_T \le k\exp(2|v|V + 2v^2a^2T + L_T)
$$
with $[L,L]_T = \sum_{r=1}^d \int_0^T |\eta^r_t|^2 \dd t \le 4v^2a^2T$. By the Cauchy-Schwarz inequality, it follows that
$$
\E[F_T] \le ke^{2v^2a^2 T} \sqrt{\E[e^{4|v|V}]} \sqrt{\E[e^{2L_T}]}.
$$
Note that the upper bound for $[L,L]_T$ ensures that the stochastic exponential $\mathcal{E}(2L_t), t \in [0,T]$ is a true martingale and consequently $\E[e^{2L_T}] = \E[\mathcal{E}(2L_T)e^{2[L,L]_T}] \le e^{8v^2a^2T}$, whence
$$
\E[F_T] \le ke^{6v^2a^2T}\sqrt{\E[e^{4|v|V}]}.
$$
For independent Gaussian $X_0$, the continuity and Gaussian property of the OU process $X_{[0,T]}(\omega)$ together with the Fernique theorem implies that for $V = \int_0^T \|B_t\|_{\op} \dd t = \int_0^T \|M(K(m-X_t))\|_{\op} \dd t$ there is a $\varepsilon > 0$ such that
$
H := \E[e^{\varepsilon V^2}] < \infty
$
and using $4|v|V \le \varepsilon V^2 + \frac{4v^2}{\varepsilon}$, we conclude that
$$
\|Y^z_T\|_{L^2} = \sqrt{\E[F_T]} \le A e^{Gv^2/2}, \quad v = \Im(z).
$$
where $A:=\sqrt{k}H^{1/4}$ and $G:= 6a^2T+\frac{2}{\varepsilon}$. Note that the above upper bound for $\|Y^z_T\|_{L^2}$ does not rely on $\Re(z)$.

Since $z \mapsto Y^z_T$ is holomorphic on $\C$ (see Theorem \ref{thm: properties of dev for rough paths}), the Cauchy's formula gives, for any $\lambda \in \R$ and any $R>0$,
\begin{align*}
    \frac{1}{n!}Y^{\lambda,(n)}_T &= \frac{1}{2 \pi \mathrm{i}}\oint_{|z - \lambda| = R} \frac{Y^z_T}{(z - \lambda)^{n+1}} \dd z \\
    &=\frac{1}{2 \pi R^n} \int_0^{2\pi} Y^{\lambda + Re^{\mathrm{i}\theta}}_T e^{-\mathrm{i} n \theta} \dd \theta,
\end{align*}
and then by Minkowski integral inequality, we get
\begin{align*}
    \frac{1}{n!}\| Y^{\lambda,(n)}_T \|_{L^2} &\le \frac{1}{2 \pi R^n} \int_0^{2\pi} \|Y^{\lambda + Re^{\mathrm{i}\theta}}_T\|_{L^2} \dd \theta \\
    &\le \frac{A}{2\pi R^n} \int_0^{2\pi} e^{GR^2\sin^2 \theta/2} \dd \theta \quad (\Im(\lambda + Re^{\mathrm{i}\theta}) = R \sin \theta)\\
    &\le AR^{-n}e^{GR^2/2},
\end{align*}
where the right-hand side of the above inequality does not depend on $\lambda \in \R$. Since the minimiser for the function $R^{-n}e^{GR^2/2}$ for $R>0$ is $\sqrt{\frac{n}{G}}$, inserting such $R$ into the right-hand side of the above inequality, we finally obtain that
$$
\sup_{\lambda \in \R} \frac{1}{n!}\| Y^{\lambda,(n)}_T \|_{L^2} \le A\bigg(\frac{e G}{n}\bigg)^{\frac{n}{2}}.
$$
Our claim then following from the fact that $n! \le n^n$, by setting $B:= \sqrt{e G}$. Note that constants $A$ and $B$ only depends on $K$, $m$, $\Sigma$, $X_0$, $k$, $\|M\|_{\op}$ and $T$.
\end{proof}

\begin{remark}
In the proof of the above theorem \ref{thm: bounds for OU} we actually obtained an interesting result which is valid for all continuous semimartingales, namely for any $M \in \bigsqcup_{k=1}^\infty \lin(\R^d, \so(k))$ and $z \in \C$, $\|Y^{zM}_t\|_{\HS}^2$ only depends on the imaginary part $\Im(z)$ of $z$ thanks to the anti-symmetry of $M(\cdot)$. This is crucial to establish a uniform bound for $\frac{1}{n!}\|Y^{\lambda,(n)}_t\|_{L^2(\HS)}$ over all $\lambda \in \R$ through the Cauchy's formula.
\end{remark}

Using the similar method, we can easily prove that the expected signature of any OU process (with independent Gaussian initial value) defined on a deterministic time horizon $[0,T]$ has infinite radius of convergence.

\begin{theorem}\label{thm: OU has infinite ROC}
Let $X_{[0,T]}$ be a $d$-dimensional OU process as in \eqref{eq: OU process} such that $X_0$ is independent of the driven signal $W_{[0,T]}$ with a Gaussian distribution, where $T > 0$ is deterministic. Then for any $u>0$, there exists a constant $C_u = C(T,u, K, \Sigma, X_0) > 0$ such that for all $n \ge 1$,
$$
\E[\|S(X)^n_{0,T}\|_{\otimes n}] \le C_u \bigg(\frac{\sqrt{d}}{u}\bigg)^n.
$$
In particular, we have $r_2(X_{[0,T]}) = \infty$.
\end{theorem}

\begin{proof}
Let $V_T := \int_0^T |K(m-X_t)| \dd t$ and $K_T := \|\Sigma\|_{\HS}^2 T$. For independent Gaussian $X_0$, the continuity and Gaussian property of the OU process $X_{[0,T]}(\omega)$ together with the Fernique theorem implies that $V_T$ is sub-Gaussian and therefore for all $\alpha > 0$ it holds that $\E[e^{\alpha V_T}]<\infty$. Obviously, this implies that for all $\alpha > 0$ and all $\gamma > 0$ one has $H(\alpha,\gamma,T):= \E[e^{\alpha V_T + \gamma K_T}]<\infty$. In Theorem \ref{thm: ROC of stopped semimart} in the next subsection we can prove that for all $u>0$ such that $2u\le \alpha$ and $2u^2 \le \gamma$, it holds
$$
\forall n \ge 1, \quad \E[\|S(X)^n_{0,T}\|_{\otimes n}] \le \sqrt{H(\alpha,\gamma,T)}\bigg(\frac{\sqrt{d}}{u}\bigg)^n,
$$
which in turn implies that $r_2(X_{[0,T]}) \ge \frac{1}{\sqrt{d}}\bigg(\frac{\alpha}{2} \wedge \sqrt{\frac{\gamma}{2}}\bigg)$. Since here $\alpha$ and $\gamma$ can be arbitrarily large, we have $r_2(X_{[0,T]}) = \infty$.
\end{proof}

Thanks to Theorem \ref{cor: Ito martingale case}, Remark \ref{remark: ito semimart has infinite ROC}, Theorem \ref{thm: bounds for OU} and Theorem \ref{thm: OU has infinite ROC}, we can apply Theorem \ref{thm: random_func} and Proposition \ref{prop: main criterion}  to deduce the following results for all continuous semimartingales $X_{[0,T]}$ on a deterministic time horizon $[0,T]$ appeared in this subsection:

\begin{corollary}\label{cor: OU}
 Let $X_{[0,T]}(\omega)$ be an It\^o semimartingale with bounded drift and volatility as in Theorem \ref{cor: Ito martingale case}, or a Gaussian martingale as in Corollary \ref{cor: Gaussian martingale}, or
 a $d$-dimensional OU process as in \eqref{eq: OU process} with independent Gaussian initial value $X_0$, where $T>0$ is deterministic. Then there exist constants $A$ and $B$, which depend only on the intrinsic parameters of $X$, $M \in \bigsqcup_{k = 1}^\infty \lin(\R^d, \so(k))$ and $T$, such that
 \begin{itemize}
     \item For any $M \in \bigsqcup_{k = 1}^\infty \lin(\R^d, \so(k))$ and $\lambda \in \R$,
     $$
     \forall n \ge 0, \quad \frac{\dd^n}{\dd \lambda^n}\phi_{X_{[0,T]}}(\lambda M) = \E[Y^{\lambda,(n)}_T] = n!\E[D_n(\hat X_t)Y^\lambda_t].
     $$
     \item For any $M \in \bigsqcup_{k = 1}^\infty \lin(\R^d, \so(k))$ and and any $n \ge 1$, one has
     $$
     \sup_{\lambda \in \R} \bigg\|\frac{\frac{\dd^n}{\dd \lambda^n}\phi_{X_{[0,T]}}(\lambda M)}{n!}\bigg\|_{\HS} \le A\frac{B^n}{\sqrt{n!}}.
     $$
     \item For any $M \in \bigsqcup_{k = 1}^\infty \lin(\R^d, \so(k))$, $\lambda, h \in \R$ and $N \ge 0$, one has
     $$
     \bigg\|Y^{\lambda+h}_T - \sum_{n=0}^N \frac{Y^{\lambda,(n)}_T}{n!}\bigg\|_{L^2}  \le A\sum_{n=N+1}^\infty \frac{(B|h|)^n}{\sqrt{n!}}
     $$
     and $$
     \bigg\|\phi_{X_{[0,T]}}((\lambda+h)M) - \sum_{n=0}^N \frac{\frac{\dd^n}{\dd \lambda^n}\phi_{X_{[0,T]}}(\lambda M)}{n!}h^n\bigg\|_{\HS}  \le A\sum_{n=N+1}^\infty \frac{(B|h|)^n}{\sqrt{n!}}.
     $$
     \item The expected signature $\ESig(X_{[0,T]})$ uniquely determines the law of its distribution $\mu_{S(X)_{0,T}}$.
 \end{itemize}
 \end{corollary}

\begin{remark}[Time-augmented case]
  Let $X_{[0,T]}(\omega)$ be an It\^o semimartingale with bounded drift and volatility as in Theorem \ref{cor: Ito martingale case}, or a Gaussian martingale as in Corollary \ref{cor: Gaussian martingale}, or
 a $d$-dimensional OU process as in \eqref{eq: OU process} with independent Gaussian initial value $X_0$, where $T>0$ is deterministic.   Then one can readily apply the same argumets in this subsection to show that Corollary \ref{cor: OU}  remains valid for the (Stratonovich lift of) time-augmented process $\hat X_t = (t, X_t), t \in [0,T]$ and $\hat M \in \bigsqcup_{k=1}^\infty \lin(\R^{1+d},\so(k))$ with possibly different choices of $A$ and $B$.
\end{remark}

\subsubsection{Semimartingales on a random time horizon}\label{subsect: stopped BM}
In this subsection, we will give a feasible $L^q$-bound ($q=1$ or $q=2$) for high-order radial derivatives of unitary developments for some continuous semimartingales on a random time horizon, which are independent of the base point $\lambda \in \R$.

We consider now a $d$-dimensional continuous semimartingale $X_t, t \ge 0$ defined on a filtered probability space $(\Omega,\bbF,\bbP)$, whose canonical decomposition is
$
X_t = X_0 + B_t + N_t, t \ge 0,
$
where $(N_t)_{t \ge 0}$ is a continuous local martingale and $(B_t)_{t \ge 0}$ is a bounded variation process with $N_0 = B_0 = 0$. For the later use, we also define $V_t := \|B\|_{1\var,[0,t]}$ and $K_t := \llbracket N, N \rrbracket_t = \sum_{i=1}^d [N^i, N^i]_t$ for $t \ge 0$.

\begin{theorem}\label{thm: stopped semimart}
Let $\tau$ be a stopping time such that $\bbP(\tau<\infty) = 1$, which does not need to be the first exit time of $X$ from a domain. Assume there exist constants $\alpha > 0, \gamma > 0$ such that $H := \E[e^{\alpha V_\tau + \gamma K_\tau}]<\infty$. Then for any $M \in \lin(\R^d, \so(k))$, there exists a positive $R>0$ which only depends on $\|M\|_{\op}, \alpha,$ and $\gamma$ such that for any $n \ge 0$,
\begin{equation}\label{eq: bound for stopped semimart}
    \sup_{\lambda \in \R} \frac{\|Y^{\lambda,(n)}_\tau\|_{L^2}}{n!} \le \sqrt{k}H^{\frac{1}{4}}R^{-n}.
\end{equation}
\end{theorem}

\begin{proof}
For $z  = u + \mathrm{i} v \in \C$, we consider the process $Y^z_t, t \ge 0$ which is the solution to the linear SDE $\dd Y^z_t = Y^z_t (zM)(\circ \dd X_t), Y^z_0 = I_k$. As in the proof of Theorem \ref{thm: bounds for OU} we set $F_t := \|Y^z_t\|_{\HS}^2 = \Tr(Y^z_t(Y^z_t)^\dagger)$. Then applying It\^o's formula to $(F_t)_{t \ge 0}$ and following the same calculations as in the proof of Theorem \ref{thm: bounds for OU}, we can readily check that
$$
\dd F_t \le 2|v|\ell F_t \dd V_t + 2v^2\ell^2 F_t \dd K_t + \dd L_t,
$$
where $\ell := \sup_{|x|=1}\|M(x)\|_{\op}$, $(L_t)_{t \ge 0}$ is a continuous local martingale such that $[L,L]_t \le 4v^2 \ell^2 F_t^2 \dd K_t$. Using the above estimates for the dynamics of $(F_t)_{t \ge 0}$ we can use It\^o's formula to $G_t := F_t^2 = \|Y^z_t\|^4_{\HS}$ to get
$$
\dd G_t \le 4|v|\ell G_t \dd V_t + 8v^2\ell^2G_t \dd K_t + \dd \tilde L_t,
$$
where $(\tilde L_t)_{t \ge 0}$ is a continuous local martingale. Next, define $A^v_t:= 4|v|\ell V_t + 8v^2\ell^2K_t, t \ge 0$, applying It\^o's formula to the process $(e^{-A^v_t}G_t)_{t \ge 0}$ we then use the above upper bound for $\dd G_t$ to obtain that
$$
\dd (e^{-A^v_t}G_t) \le e^{-A^v_t} \dd \tilde L_t
$$
which implies that the process $(e^{-A^v_t}G_t)_{t \ge 0}$ is a nonnegative local supermartingale. Then by a localizing argument together with optional sampling theorem and Fatou's lemma, we can further deduce that
$$
\E[e^{-A^v_\tau}G_\tau] = \E[e^{-A^v_\tau} \|Y^z_\tau\|^4_{\HS}] \le k^2.
$$
Then the Cauchy-Schwarz inequality gives that
\begin{align*}
    \E[\|Y^z_\tau\|^2_{\HS}] = \E[e^{A^v_\tau/2} e^{-A^v_\tau/2} \|Y^z_\tau\|^2_{\HS}  ] &\le \sqrt{\E[e^{A^v_\tau}]}\sqrt{\E[e^{-A^v_\tau}G_\tau]}\\
    &\le k\sqrt{\E[e^{A^v_\tau}]}.
\end{align*}
Now, we pick a positive $R>0$ such that $4\ell R \le \alpha$ and $8\ell^2R^2 \le \gamma$, then for all $z \in \C$ with $|v| = |\Im(z)| \le R$, we have
$A^v_\tau:= 4|v|\ell V_\tau + 8v^2\ell^2K_\tau \le \alpha V_\tau + \gamma K_\tau$, and consequently
$$
\sup_{|\Im(z)| \le R} \|Y^z_\tau\|_{L^2} \le \sqrt{k}H^{\frac{1}{4}}.
$$
Now we can use the above bound to get the desired claim \eqref{eq: bound for stopped semimart} which is the same as what we did in the proof of Theorem \ref{thm: bounds for OU}: for any $\lambda \in \R$ and $n \ge 0$, by the Cauchy's formula and Minkowski's integral inequality it yields that
\begin{align*}
    \frac{1}{n!}\| Y^{\lambda,(n)}_\tau \|_{L^2} &\le \frac{1}{2 \pi R^n} \int_0^{2\pi} \|Y^{\lambda + Re^{\mathrm{i}\theta}}_T\|_{L^2} \dd \theta \\
    &\le \frac{\sqrt{k}H^{\frac{1}{4}}}{2\pi R^n} \int_0^{2\pi}  \dd \theta \quad (\Im(\lambda + Re^{\mathrm{i}\theta}) = R \sin \theta)\\
    &\le \sqrt{k}H^{\frac{1}{4}}R^{-n}.
\end{align*}
\end{proof}

Using the same approach we can prove that for any continuous semimartingale $X_t = X_0 + B_t + N_t, t \ge 0$ and stopping time $\tau$ satisfying the assumptions imposed in Theorem \ref{thm: stopped semimart}, its expected signature $\ESig(X_{[0,\tau]})$ has a positive radius of convergence.

\begin{theorem}\label{thm: ROC of stopped semimart}
 Let $\tau$ be a stopping time such that $\bbP(\tau<\infty) = 1$. Assume there exist constants $\alpha > 0, \gamma > 0$ such that $H := \E[e^{\alpha V_\tau + \gamma K_\tau}]<\infty$.
 Then for all $n\ge 1$,  it holds that
 $$
 \E[\|S(X)^n_{0,\tau}\|_{\otimes n}] \le \sqrt{H}\bigg(\frac{\sqrt{d}}{u}\bigg)^n,
 $$
 for any $u>0$ such that $2u \le \alpha$ and $2u^2 \le \gamma$. In particular, we have $$r_2(X_{[0,\tau]}) \ge \frac{1}{\sqrt{d}}\bigg(\frac{\alpha}{2} \wedge \sqrt{\frac{\gamma}{2}}\bigg)>0.
 $$
\end{theorem}

\begin{proof}
For simplicity, we write $S_t = S(X)_{0,t}$ for the Stratonovich signature of $X_{[0,t]}$, and then $S^n_t \in V^{\otimes n}$ denotes the $n$-th component of $S_t$. Furthermore, for $J \ge 1$ and $u>0$, we define $Z_J(u)_t := (1, uS^1_t, u^2 S^2_t, \ldots, u^JS^J_t)$ and $F_J(u)_t := \|Z_J(u)_t\|_2^2 := \sum_{n=0}^J u^{2n} |S^n_t|_2^2$, where $|\cdot|_2$ denotes the $\ell^2$-Euclidean norm on $V^{\otimes n}$ for every $n \ge 0$. For a given $x \in V$, we define a linear operator $\iota^J(x): T^J(V) \to T^J(V)$ such that $\iota^J(x)(z^0,z^1,\ldots,z^J) := (0,z^0 \otimes x, \ldots, z^{J-1}\otimes x)$, which clearly satisfies that $\|\iota^J(x)\|_{\op}\le |x|$ (holds for any $J\ge 1$) when $T^J(V)$ is equipped with the $\ell^2$-norm $\|\cdot\|_2$.

Recalling that the Stratonovich signature $S_t$ satisfies the equation $\dd S_t = S_t \otimes (\circ \dd X_t)$, we can easily see that the process $Z_J(u)_t, t \ge 0$ satisfies the equation $\dd Z_J(u)_t = \iota^J(\circ \dd(uX_t))Z_J(u)_t = u \iota^J(\circ \dd X_t)Z_J(u)_t$, which can be rewritten as the It\^o equation
$$
\dd Z_J(u)_t = u\sum_{i=1}^d \iota^J_iZ_J(u)_t \dd B^i_t + u\sum_{i=1}^d \iota^J_iZ_J(u)_t \dd N^i_t + \frac{u^2}{2}\sum_{i,j=1}^d \iota^J_i\iota^J_jZ_J(u)_t \dd[N^i,N^j]_t.
$$
where $\iota^J_i := \iota^J(e_i)$ for the canonical basis $e_1,\ldots,e_d$ for $V$. More concretely, for any $n = 2,\ldots,J$ and $i_1,\ldots,i_n \in \{1,\ldots,d\}$, the above equation reads
$$
\dd Z^{i_1\ldots,i_n}_J(u)_t = u Z^{i_1\ldots i_{n-1}}_J(u)_t \dd B^{i_n}_t + u Z^{i_1\ldots i_{n-1}}_J(u)_t \dd N^{i_n}_t + \frac{u^2}{2}Z^{i_1\ldots i_{n-2}}_J(u)_t \dd[N^{i_{n-1}},N^{i_{n}}]_t
$$
where $Z^{i_1\ldots i_n}_J(u)_t$ denotes the coordinate of $Z_J(u)_t$ in front of the basis vector $e_{i_1} \otimes \ldots \otimes e_{i_n}$ and we set $Z^{i_0}_J(u)_t := Z^0_J(u)_t = 1$. Applying It\^o's formula to $F_J(u)_t = \|Z_J(u)_t\|_2^2 = \sum_{n=0}^J \sum_{i_1,\ldots,i_n = 1}^d |Z^{i_1\ldots i_n}_J(u)_t|^2$, and noting that $\dd([N^i,N^j])^{i,j=1,\ldots,d}_t = C_t\dd K_t$ for some nonnegative definite symmetric matrix-valued process $(C_t)_{t \ge 0}$ with $\Tr(C_t) = 1$ as well as $\|\iota^J_i\|_{\op} \le 1$ for all $i=1,\ldots,d$, we can easily check that
\begin{equation}\label{eq: bound for l2 norm from Ito formula}
    \dd F_J(u)_t \le F_J(u)_t(2u \dd V_t + 2u^2 \dd K_t) + \dd L_J(u)_t
\end{equation}
where $L_J(u)_t, t \ge 0$ is a continuous local martingale. Consequently, if we set $A(u)_t:= 2uV_t + 2u^2 K_t$ for $t \ge 0$, by the product rule and the inequality \eqref{eq: bound for l2 norm from Ito formula} we can easily see that $\dd e^{-A(u)_t}F_J(u)_t \le e^{-A(u)_t}\dd L_J(u)_t$,  which in turn implies that $(e^{-A(u)_t}F_J(u)_t)_{t \ge 0}$ is a nonnegative local supermartingale. Since $F_J(u)_0 = 1$, by a localizing argument, optional sampling theorem and Fatou's lemma, it follows that $\E[e^{-A(u)_\tau}F_J(u)_\tau] \le 1$. So, for all $n \le J$, due to $u^{2n}|S^n_t|_2^2 \le F_J(u)_t$ we have $\E[e^{-A(u)_\tau} |S^n_\tau|_2^2] \le u^{-2n}$.

Now by Cauchy-Schwarz, we note that
$$
\E[|S^n_\tau|_2] = \E[e^{A(u)_\tau/2} e^{-A(u)_\tau/2}|S^n_\tau|_2^2] \le \sqrt{\E[e^{A(u)_\tau}]}\sqrt{\E[e^{-A(u)_\tau}|S^n_\tau|_2^2]]} \le u^{-n}\sqrt{\E[e^{A(u)_\tau}]}.
$$
Hence, if $u>0$ satisfies $2u\le \alpha$ and $2u^2 \le \gamma$, we must have $A_\tau = 2uV_\tau + 2u^2K_\tau \le \alpha V_\tau + \gamma K_\tau$, which implies that $\sqrt{\E[e^{A(u)_\tau}]} \le \sqrt{\E[e^{\alpha V_\tau + \gamma K_\tau}]} = \sqrt{H}<\infty$, whence $\E[|S^n_\tau|_2] \le \sqrt{H}u^{-n}$. Finally, since $\|\cdot\|_{\otimes n}$ (the projective norm on $V^{\otimes n}$) satisfies $\|\cdot\|_{\otimes n} \le d^{\frac{n}{2}}|\cdot|_2$, we actually obtain
$$
\E[|S^n_\tau|_{\otimes n}] \le \sqrt{H} \bigg( \frac{\sqrt{d}}{u}\bigg)^n
$$
for all $n \le J$. Since $J$ can be arbitrarily large, from above we can conclude that $r_2(X_{[0,\tau]}) \ge \frac{1}{\sqrt{d}}\bigg(\frac{\alpha}{2} \wedge \sqrt{\frac{\gamma}{2}}\bigg) > 0$.
\end{proof}

By Theorem \ref{thm: stopped semimart} and Theorem \ref{thm: ROC of stopped semimart}, we obtain that any stopped continuous semimartingale $X_{[0,\tau]}$ satisfies the Conditions (ROC), as long as the integrability condition $\E[e^{\alpha V_\tau + \gamma K_\tau}]<\infty$ for some $\alpha>0,\eta>0$ is fulfilled. Therefore, we have the following results:

\begin{corollary}\label{cor: determinacy for stopped semimart}
Let $X_t, t \ge 0$ be a $d$-dimensional continuous semimartingale, and $\tau$ is a stopping time with $\bbP(\tau < \infty) = 1$. If there exist $\alpha>0,\eta>0$ such that $H := \E[e^{\alpha V_\tau + \gamma K_\tau}]<\infty$, then we have
\begin{enumerate}
  \item For any $M \in \bigsqcup_{k = 1}^\infty \lin(\R^d, \so(k))$ and $\lambda \in \R$,
     $$
     \forall n \ge 0, \quad \frac{\dd^n}{\dd \lambda^n}\phi_{X_{[0,\tau]}}(\lambda M) = \E[Y^{\lambda,(n)}_\tau] = n!\E[D_n(\hat X^\lambda_\tau)Y^\lambda_\tau],
     $$
     where $D_n(\hat X^\lambda_t)$ is the matrix signature as in \eqref{eq: recursive formula for matrix signature in martingale case}.
     \item For any $M \in \bigsqcup_{k = 1}^\infty \lin(\R^d, \so(k))$, $\lambda \in \R$ and $N \ge 0$, one has
     $$
     \bigg\|Y^{\lambda+h}_\tau - \sum_{n=0}^N \frac{Y^{\lambda,(n)}_\tau}{n!}\bigg\|_{L^2}  \le \sqrt{k}H^{\frac{1}{4}}\sum_{n=N+1}^\infty \bigg(\frac{|h|}{R}\bigg)^n
     $$
     and $$
     \bigg\|\phi_{X_{[0,\tau]}}((\lambda+h)M) - \sum_{n=0}^N \frac{\frac{\dd^n}{\dd \lambda^n}\phi_{X_{[0,\tau]}}(\lambda M)}{n!}h^n\bigg\|_{\HS}  \le \sqrt{k}H^{\frac{1}{4}}\sum_{n=N+1}^\infty \bigg(\frac{|h|}{R}\bigg)^n
     $$
     for all $|h|<R$, where $R>0$ is a constant only depending on $\|M\|_{\op}$, $\alpha$, and $\gamma$.
       \item The distribution $\mu_{S(X)_{0,\tau}}$ is uniquely determined by $\ESig(X_{[0,\tau]})$ within the class of random geometric rough paths satisfying the Conditions (ROC).
\end{enumerate}
\end{corollary}

\begin{remark}[Time-augmented case]
Let $X_{[0,\tau]}$ be a continuous semimartingale with the canonical decomposition $X_t = X_0 + B_t + N_t$, $\tau$ be a stopping time such that there exist constants $\alpha>0, \gamma>0, \eta>0$ such that $\E[e^{\alpha V_\tau + \gamma K_\tau}] + \E[e^{\eta \tau}]<\infty$, where $V_t = \|B\|_{1\var,[0,t]}$ and $K_t = \llbracket N, N \rrbracket_t$. Then the time-augmented process $\hat X_t = (t, X_t)$ admits the canonical decomposition $\hat X_t = \hat X_0 + \hat B_t + \hat N_t$ with $\hat B_t = (t, B_t)$ and $\hat N_t = (0, N_t)$, which implies that $\hat V_\tau = \|\hat B\|_{1\var,[0,\tau]} \lesssim \tau + V_\tau$ and $K_\tau = \llbracket \hat N, \hat N \rrbracket_\tau = \llbracket N, N \rrbracket_\tau$. Therefore, one can find constants $a>0$ and $b>0$ depending on $\alpha, \gamma, \eta$ such that $\E[e^{a \hat V_\tau + b \hat K_\tau}]<\infty$. Therefore, Corollary \ref{cor: determinacy for stopped semimart} remains valid for $\hat X_{[0,\tau]}$ and $\hat M \in \bigsqcup_{k=1}^\infty \lin(\R^{1+d},\so(k))$ with $H$ replaced by $\E[e^{a \hat V_\tau + b \hat K_\tau}]$.
\end{remark}

Now we will give some concrete and important examples of stopped continuous semimartingales which satisfy the conditions in Theorem \ref{thm: stopped semimart} and Theorem \ref{thm: ROC of stopped semimart}.

\paragraph{Stopped Brownian motion from a regular bounded domain}

    Let $\D \subset \R^d$ be a strongly Lipschitz bounded domain of the class $C^m$ with $m \ge \lfloor \frac{d}{2} \rfloor + 1$, let $X = (X^1_t, \ldots, X^d_t)_{t \in [0,\infty)}$ be a Brownian motion started from some $x \in \D$. Define
$$
\tau_\D = \inf\{t \ge 0: X_t \notin \D\}
$$
be the first exit time of $X$ from $\D$.

We remark that if $2 \le d \le 8$ and $\D$ is additionally of class $C^{2,\gamma}$ for some $\gamma \in (0,1)$, then the expected signature $\ESig(X_{[0,\tau_\D]})$ of stopped Brownian motion $X^{\tau_\D}$ has only finite radius of convergence (see \cite{finiteROC2022}), and consequently one cannot use \cite[Corollary 6.6]{ChevyrevLyons2016}, which asks for an infinite radius of convergence of expected signature, to deduce that the distribution $\mu_{S(X)_{0,\tau_\D}}$ is uniquely determined by $\ESig(X_{[0,\tau_\D]})$.

In view of \cite{LyonsHao2015}, we know that for any $x \in \D$ and any $\alpha \in (0, \mu_1(\D))$, where $\mu_1(\D)$ denotes the principal Dirichlet eigenvalue of $ - \frac{1}{2}\Delta$ on $\D$, it holds that $K_\alpha(x) : = \E_x[e^{\alpha \tau_\D}] < \infty$. Then, since for $d$-dimensional Brownian motion we have
$V_t \equiv 0$ (there is no drift term for Brownian motion) and $K_t = d t$, the condition $\E_x [e^{a V_{\tau_\D} + \gamma K_{\tau_\D}}] = \E[e^{\gamma K_{\tau_\D}}] < \infty$ by taking $\gamma := \frac{\alpha}{d}$ for any $\alpha \in (0, \mu_1(\D))$.
Therefore, such stopped Brownian motion meets the conditions in Theorem \ref{thm: stopped semimart} and Theorem \ref{thm: ROC of stopped semimart}. In particular, Corollary \ref{cor: determinacy for stopped semimart} remains valid for such $X_{[0,\tau_\D]}$.

Below we will give an alternative proof for establishing $L^1$-bounds for derivatives of unitary developments of Brownian motion stopped by $\tau_\D$,
which is based on an inductive argument applied to matrix signature $D_n(\hat{X}^{\lambda}_{t})
: = \int_0^t
D_{n - 1}(\hat{X}^{\lambda}_{s}) \circ \dd\hat{X}_s^\lambda$ with respect to $n$ (very similar to the proof of Theorem \ref{cor: Ito martingale case}).  The key tool here is the usual Burkholder-Davis-Gundy inequality instead of the Cauchy's formula used in the proof of Theorem \ref{cor: Ito martingale case}.

\begin{theorem}\label{thm: bound for matrix signature:stopped BM}
For every integer $k \ge 1$, there exists a constant $\bar C = \bar C(k,d) \ge 1$ such that for every $M \in \lin(\R^d,\so(k))$, $\lambda \in \R$, $n \ge 1$ and $T \ge 0$,
\begin{equation}\label{eq: bound for supremum of D_n}
    \bigg\| \sup_{t \in [0,T]}\|D_n(\hat X^\lambda_t)\|_{\HS}   \bigg\|_{L^2} \le \frac{2^n (\bar C(||M||_{\op} \vee 1)^2)^n T^{\frac{n}{2}}}{\sqrt{n!}}.
\end{equation}
Moreover, we also have for all $n \ge 1$,
\begin{equation}\label{eq: bound on the high order derivative for stopped BM}
  \frac{1}{n!}\|Y^{\lambda,(n)}_{\tau_\D}\|_{L^1} \le e^3K_\alpha(x)^{\frac{1}{2}} \bigg( 2\bar C (\|M\|_{\op} \vee 1)^2 \sqrt{\frac{2}{\alpha}} \bigg)^n
 \end{equation}
as long as $\alpha \in (0, \mu_1(\D))$.
\end{theorem}

\begin{proof}
Again, we may assume that $\|M\|_{\op} \ge 1$.
\begin{enumerate}
    \item The proof of the bound \eqref{eq: bound for supremum of D_n} is analogous to the proof of Theorem \ref{cor: Ito martingale case}. We still apply an induction argument for $n \ge 1$.  The case $n = 1$ is clear: by the Burkholder-Davis-Gundy inequality (see \cite{MarinelliRockner2016}), we obtain immediately that
    $$
     \bigg\| \sup_{t \in [0,T]}\|\hat X^\lambda_t\|_{\HS}   \bigg\|_{L^2} = \E\bigg[ \sup_{t \in [0,T]}\|\hat X^\lambda_t\|_{\HS}^2\bigg]^{\frac{1}{2}} \le \sqrt{\beta_2}\E[\llbracket \hat X^\lambda, \hat X^\lambda   \rrbracket_T]^{\frac{1}{2}}.
    $$
In view of \eqref{eq: quadratic variation of hat X} we know that $\llbracket \hat X^\lambda, \hat X^\lambda \rrbracket_t = \llbracket M(X), M(X) \rrbracket_t$, and therefore by Lemma \ref{lemma: bound on the quadratic variation of M(X)}, we get that
$$
 \bigg\| \sup_{t \in [0,T]}\|\hat X^\lambda_t\|_{\HS}   \bigg\|_{L^2} \le \sqrt{\beta_2d}\|M\|_{\op}T^{\frac{1}{2}},
$$
which means that \eqref{eq: bound for supremum of D_n} is satisfied for $n = 1$ if we take $\bar C = \bar C(k,d) = C(k,d)\beta_2$, where $C(k,d)$ is as in Lemma \ref{lem: bound on quad variation of hat X when X is BM}.

Now assume that the bound \eqref{eq: bound for supremum of D_n} holds for $D_m(\hat X^\lambda)$ for all $1 \le m \le n$ with some $n \ge 1$, and we consider the case $n + 1$. As in the proof of Theorem \ref{cor: Ito martingale case}, we have
\[
D_{n + 1}(\hat{X}^{\lambda}_{t}) = \int_0^t
D_{n}(\hat{X}^{\lambda}_{s})\dd\hat{X}_s^\lambda + \frac12
\int_0^t
\dd\bigl[ D_{n}(\hat{X}^{\lambda}_{\cdot}),\hat{X}^\lambda\bigr]_s,
\]
and thus
\begin{align}\label{eq: sum of two terms stopped BM}
  \bigg\| \sup_{t \in [0,T]}\|D_{n + 1}(\hat X^\lambda_t)\|_{\HS}   \bigg\|_{L^2} & \le \bigg\| \sup_{t \in [0,T]}\bigg\|\int_0^t
D_{n}(\hat{X}^{\lambda}_{s})\dd\hat{X}_s^\lambda\bigg\|_{\HS}   \bigg\|_{L^2}  \nonumber \\
& \quad + \frac{1}{2} \bigg\| \sup_{t \in [0,T]}\bigg\|  \int_0^t
\dd\bigl[ D_{n}(\hat{X}^{\lambda}_{\cdot}),\hat{X}^\lambda\bigr]_s \bigg\|_{\HS}   \bigg\|_{L^2}.
\end{align}
We apply the Burkholder-Davis-Gundy inequality again to the martingale $I_t : = \int_0^t D_n(\hat X_s^\lambda)\dd\hat X_s^\lambda$ and obtain
$
\|\sup_{t \in [0,T]}\|I_t\|_{\HS}\|_{L^2} \le \sqrt{\beta_2}\E[\llbracket I,I\rrbracket_T]^{1/2}.
$
From \eqref{eq: predictable compensator} we see that
$$
\bigg\llbracket \int_0^\cdot D_{n}(\hat{X}^{\lambda}_{s})\dd\hat{X}_s^\lambda,\,
\int_0^\cdot D_{n}(\hat{X}^{\lambda}_{s})\dd\hat{X}_s^\lambda \bigg\rrbracket_T = \int_0^T
\text{Tr}\bigg(D_{n}(\hat{X}^{\lambda}_{s})\dd[\hat{X}^\lambda, (\hat{X}^\lambda)^\dagger]_s D_{n}(\hat{X}^{\lambda}_{s})^\dagger\bigg)
$$
and therefore using the induction hypothesis and Lemma \ref{lem: bound on quad variation of hat X when X is BM} exactly as in the proof of Theorem \ref{cor: Ito martingale case}, we can further deduce that
\begin{align}\label{eq: bound 1 for stopped BM}
  \bigg\| \sup_{t \in [0,T]}\bigg\|\int_0^t
D_{n}(\hat{X}^{\lambda}_{s})\dd\hat{X}_s^\lambda\bigg\|_{\HS}   \bigg\|_{L^2} & \le 2^n(\bar C\|M\|^2_{\op})^{n + 1}\frac{T^{\frac{n + 1}{2}}}{\sqrt{(n + 1)!}},
\end{align}
because $\bar C = C\beta_2$.

For the second term in \eqref{eq: sum of two terms stopped BM}, the argument for $n \ge 2$ uses the induction hypothesis at level $n - 1$. When $n = 1$, we instead use $D_1(\hat X^\lambda) = \hat X^\lambda$ and the identity
\[
\dd[\hat X^\lambda,\hat X^\lambda]_s = Y^\lambda_s\left(\sum_{r = 1}^d M(e_r)^2\right)(Y^\lambda_s)^\dagger\dd s.
\]
Orthogonal invariance and submultiplicativity of the Hilbert--Schmidt norm give $\|\sum_{r = 1}^d M(e_r)^2\|_{\HS} \le d\|M\|_{\op}^2$. Consequently,
\[
\frac{1}{2}\bigg\|\sup_{t \in [0,T]}\|[\hat X^\lambda,\hat X^\lambda]_t\|_{\HS}\bigg\|_{L^2} \le \frac{d}{2}\|M\|_{\op}^2 T \le \frac{\bar C\|M\|_{\op}^2 T}{\sqrt{2}}.
\]
Thus, for every $n \ge 1$, we obtain
\begin{equation}\label{eq: bound 2 for stopped BM}
  \frac{1}{2}\bigg\| \sup_{t \in [0,T]}\bigg\|  \int_0^t
\dd\bigl[ D_{n}(\hat{X}^{\lambda}_{\cdot}),\hat{X}^\lambda\bigr]_s \bigg\|_{\HS}   \bigg\|_{L^2} \le 2^{n - 1} (\bar C\|M\|_{\op}^2)^n\frac{T^{\frac{n + 1}{2}}}{\sqrt{(n + 1)!}}.
\end{equation}
Now, inserting \eqref{eq: bound 1 for stopped BM} and \eqref{eq: bound 2 for stopped BM} into the inequality \eqref{eq: sum of two terms stopped BM}, we finally obtain that
\begin{align*}
     \bigg\| \sup_{t \in [0,T]}\|D_{n + 1}(\hat X^\lambda_t)\|_{\HS}   \bigg\|_{L^2} & \le (2^n + 2^{n - 1})(\bar C\|M\|^2_{\op})^{n + 1}\frac{T^{\frac{n + 1}{2}}}{\sqrt{(n + 1)!}} \\
 & \le 2^{n + 1} (\bar C\|M\|_{\op}^2)^{n + 1}\frac{T^{\frac{n + 1}{2}}}{\sqrt{(n + 1)!}},
\end{align*}
which completes the induction proof and gives \eqref{eq: bound for supremum of D_n} for all $n \ge 1$.
\item Now we turn to the proof of \eqref{eq: bound on the high order derivative for stopped BM}. Clearly by definition we have
$$
\frac{1}{n!}\|Y^{\lambda,(n)}_{\tau_\D}\|_{\HS} = \|D_n(\hat X^\lambda_{\tau_\D})\|_{\HS}
$$
and therefore $\frac{1}{n!}\|Y^{\lambda,(n)}_{\tau_\D}\|_{L^1} = \|D_n(\hat X^\lambda_{\tau_\D})\|_{L^1}$. Now, we pick some $\alpha \in (0,\mu_1(\D))$, and note that by Markov's inequality, for any $u \ge 0$, it holds that for all $x \in \D$,
\begin{equation}\label{eq: Markov ineq}
    \bbP_x(\tau_\D \ge u) \le K_\alpha(x) e^{ - \alpha u}.
\end{equation}
Next, we set $h : = \frac{2}{\alpha} > 0$, and decompose $[0,\infty)$ by equidistant intervals with the length $h$. Since the stopped process $D_n(\hat X^{\lambda}_{\cdot \wedge \tau_\D})$ is constant after time $\tau_\D$, and for almost every $\omega \in \Omega$, there exists a unique $m = m(\omega) \ge 0$ such that $\tau_\D(\omega) \in [mh, (m + 1)h)$ and hence $\|D_n(\hat X^\lambda_{\tau_\D})(\omega)\|_{\HS} \le \sup_{t \in [0, (m + 1)h]} \|D_n(\hat X^\lambda_t)(\omega)\|_{\HS}$, we must have
\begin{align*}
    \E_x[\|D_n(\hat X^\lambda_{\tau_\D})\|_{\HS}] \le \sum_{m = 0}^\infty \E_x\bigg[ 1_{\{\tau_\D \ge mh\}}     \sup_{t \in [0, (m + 1)h]} \|D_n(\hat X^\lambda_t)\|_{\HS}   \bigg].
\end{align*}
By the Cauchy-Schwarz inequality, for each $m \ge 0$ it holds that
\begin{align*}
    \E_x\bigg[ 1_{\{\tau_\D \ge mh\}}     \sup_{t \in [0, (m + 1)h]} \|D_n(\hat X^\lambda_t)\|_{\HS}   \bigg] & \le \bbP_x(\tau_\D \ge mh)^{\frac{1}{2}} \\
     & \quad \times \bigg\| \sup_{t \in [0,(m + 1)h]} \|D_n(\hat X^\lambda_t)\|_{\HS}\bigg\|_{L^2}.
\end{align*}
By \eqref{eq: Markov ineq} we have $\bbP_x(\tau_\D \ge mh)^{\frac{1}{2}} \le K_\alpha(x)^{\frac{1}{2}}e^{ - \frac{\alpha mh}{2}} = K_\alpha(x)^{\frac{1}{2}}e^{ - m}$ (recall $h = 2/\alpha$). By \eqref{eq: bound for supremum of D_n} we have
$$
\bigg \|\sup_{t \in [0,(m + 1)h]} \|D_n(\hat X^\lambda_t)\|_{\HS}\bigg\|_{L^2} \le \frac{2^n (\bar C||M||_{\op}^2)^n ((m + 1)h)^{\frac{n}{2}}}{\sqrt{n!}}.
$$
Putting two bounds together, we find that
$$
\E_x[\|D_n(\hat X^\lambda_{\tau_\D})\|_{\HS}] \le K_\alpha(x)^{\frac{1}{2}} \frac{2^n (\bar C||M||_{\op}^2)^n h^{\frac{n}{2}}}{\sqrt{n!}} \sum_{m = 0}^\infty e^{ - m}(m + 1)^{\frac{n}{2}}.
$$
Let us give a proper upper bound to $\sum_{m = 0}^\infty e^{ - m}(m + 1)^{\frac{n}{2}}$. It is clear that
\begin{align*}
    \sum_{m = 0}^\infty e^{ - m}(m + 1)^{\frac{n}{2}} & \le \int_0^\infty e^{ - x + 1}(x + 2)^{\frac{n}{2}} \dd x
     \le e^3 \int_0^\infty y^{\frac{n}{2}} e^{ - y} \dd y = e^3\Gamma \bigg(\frac{n}{2} + 1\bigg),
\end{align*}
where $\Gamma(z) = \int_0^\infty e^{ - t}t^{z - 1} \dd t$ is the Gamma function. Then, by the fact that $\ln \Gamma(z)$ is convex in $z > 0$, we can easily obtain that
\begin{align*}
    \ln \Gamma \bigg(\frac{n}{2} + 1\bigg)  \le \frac{1}{2}\ln \Gamma(n + 1) + \frac{1}{2}\ln \Gamma(1)
    \le \frac{1}{2}\ln (n!)
\end{align*}
as $\ln \Gamma(1) = 0$ and $\Gamma(n + 1) = n!$. Hence, we have $\Gamma (\frac{n}{2} + 1) \le \sqrt{n!}$ and whence $\sum_{m = 0}^\infty e^{ - m}(m + 1)^{\frac{n}{2}} \le e^3 \sqrt{n!}$. Finally, from all above bounds we can arrive at
\begin{align*}
    \E_x[\|D_n(\hat X^\lambda_{\tau_\D})\|_{\HS}]  & \le K_\alpha(x)^{\frac{1}{2}} \frac{2^n (\bar C||M||_{\op}^2)^n h^{\frac{n}{2}}}{\sqrt{n!}} \sum_{m = 0}^\infty e^{ - m}(m + 1)^{\frac{n}{2}}\\
    & \le K_\alpha(x)^{\frac{1}{2}} \frac{2^n (\bar C||M||_{\op}^2)^n h^{\frac{n}{2}}}{\sqrt{n!}} \times e^3 \sqrt{n!} \\
    & \le e^3K_\alpha(x)^{\frac{1}{2}} \bigg( 2\bar C \|M\|_{\op}^2 \sqrt{\frac{2}{\alpha}} \bigg)^n,
\end{align*}
which is exactly \eqref{eq: bound on the high order derivative for stopped BM}.
\end{enumerate}
\end{proof}

\begin{remark}
\begin{enumerate}
    \item The determinacy conclusion in Corollary \ref{cor: determinacy for stopped semimart} applied to the above stopped Brownian motion, namely, $\ESig(X_{[0,\tau_\D]})$ uniquely determines $\mu_{S(X)_{0,\tau_\D}}$ within the class of random geometric rough paths satisfying the Conditions (ROC), can also be obtained from \cite[Example 6.20, Remark 6.24, Proposition 6.25 and Remark 6.26]{ChevyrevLyons2016}. Indeed, let $\pi_1 : \fg^4(\R^d) \to \R^d$ denote the first-level projection. The Brownian lift exits the cylindrical domain $\pi_1^{ - 1}(\D)$ precisely when $X$ exits $\D$. Although this cylinder is unbounded, its bounded first-level projection satisfies the geometric exit condition (6.3) in that reference. Remark 6.24 transfers this condition to the cylinder, and Proposition 6.25 gives the required exponential tail of the greedy count for the degree-four Brownian lift. The resulting determinacy statement has the same restriction to the analytic class of laws discussed in Remark \ref{rem: VIP}. Theorem \ref{thm: stopped semimart} and Theorem \ref{thm: bound for matrix signature:stopped BM} provide a direct quantitative derivative estimate and a radius bound uniform in $\lambda_0$, without using the greedy sequence.
    \item  Our setup for stopped semimartingale cases is different from \cite[Section 6.2]{ChevyrevLyons2016}, as we do not assume $\tau$ to be the first exit time of the underlying process from a domain, neither the underlying process to be a Markovian rough path. Moreover, the condition $\E[e^{\alpha V_\tau + \gamma K_\tau}]<\infty$ in Corollary \ref{cor: determinacy for stopped semimart} might be directly verifiable without computing the random counts of greedy sequences as in \cite[Section 6.2]{ChevyrevLyons2016}.
\end{enumerate}
\end{remark}

\paragraph{Stopped It\^o semimartingale with bounded drift and volatility}
Consider a $d$-dimensional It\^o semimartingales with bounded drifts and volatilities as in Corollary \ref{cor: Ito martingale case}, namely
$X_t = x + \int_0^t b_s \dd s + \int_0^t H_s \dd W_s, t \ge 0$ is defined on some filtered probability space $(\Omega, \bbF, \bbP)$, where $W$ is an $\R^e$-valued Brownian motion and $(b_t,H_t)_{t \ge 0}$ is a progressively measurable process taking values in $\R^d \times \lin(\R^e,\R^d)$ such that $\sup_{t \ge 0} \|b_t\|_{L^\infty} \le \beta < \infty$ and $\sup_{t \ge 0} \sum_{r=1}^e\||H^r_t|_2\|^2_{L^\infty} \le a^2$ (where $H^r_t$ denotes the $r$-th column of $H_t$). Note that thanks to the boundedness of drifts and volatilities, for any  stopping time $\tau$ with $\bbP(\tau < \infty) = 1$ we have
$$
V_\tau \le \beta \tau, \quad K_\tau \le a^2 \tau.
$$
Hence, if $\tau$ satisfies $\E[e^{\eta \tau}]<\infty$ for some $\eta > 0$, then by taking $\alpha := \frac{\eta}{2\beta}$ and $\gamma := \frac{\eta}{2a^2}$,
we have $\E[e^{\alpha V_\tau + \gamma K_\tau}] \le \E[e^{\eta \tau}]<\infty$, and therefore the stopped It\^o semimartingale $X_{[0,\tau]}$ meets all conditions in Theorem \ref{thm: stopped semimart} and Theorem \ref{thm: ROC of stopped semimart}, and consequently Corollary \ref{cor: determinacy for stopped semimart} holds for all It\^o semimartingales with bounded drifts and volatilities as in Theorem \ref{cor: Ito martingale case} and all stopping times $\tau$ with $\E[e^{\eta \tau}]<\infty$ for some $\eta > 0$.

\paragraph{Stopped OU processes}

Consider a $d$-dimensional OU process $\dd X_t = K(m - X_t) \dd t + \Sigma \dd W_t$ as in \eqref{eq: OU process}. Additionally, we will assume that
\begin{itemize}
    \item the matrix $A := -K$ is a $d \times d$-Hurwitz matrix in the sense that there exist positive definite matrices $P$ and $Q$ such that
    \begin{equation}\label{eq: Hurwitz condition}
        A^\dagger P + PA  = -Q;
    \end{equation}
    \item there exists a $\varepsilon_0 > 0$ such that
    \begin{equation}\label{eq: initial bound}
        K_0 := \E[e^{\varepsilon_0(X_0 - m)^\dagger P (X_0 - m)}] < \infty,
    \end{equation}
    \item $\tau$ is a stopping time such that $\E[e^{\eta \tau}]<\infty$ for some $\eta > 0$.
\end{itemize}
Clearly, for such semimartingale $X$, we have $V_t = \int_0^t |A(X_s - m)|\dd s$ and $K_t = \|\Sigma\|_{\HS}^2 t$. Now let us use the above assumptions to find $\alpha > 0, \gamma > 0$ such that $\E[e^{\alpha V_\tau + \gamma K_\tau}] < \infty$.

First, let $f_\varepsilon(x) := e^{\varepsilon(x - m)^\dagger P (x - m)}$ be the Lyapunov function for $\varepsilon > 0$. By It\^o's formula, we can easily check that
$$
\dd f_\varepsilon(X_t) = \mathcal L f_\varepsilon(X_t) \dd t + 2\varepsilon f_\varepsilon(X_t) \sum_{r=1}^d \langle P(X_t - m), \Sigma e_r\rangle \dd W^r_t,
$$
where $\mathcal L f_\varepsilon(x)$ satisfies
$$
\frac{\mathcal L f_\varepsilon(x)}{f_\varepsilon(x)} = \varepsilon T_0 - \varepsilon (x - m)^\dagger Q (x - m) + 2\varepsilon^2 |\Sigma^\dagger P (x - m)|^2, \quad  T_0:= \Tr(\Sigma \Sigma^\dagger P).
$$
Let $q_0 := \lambda_{\min}(Q)>0$ be the smallest positive eigenvalue of $Q$, $s:=\|\Sigma^\dagger P\|_{\op}$, we can see that if $\varepsilon>0$ is small enough such that $2\varepsilon s^2 \le \frac{q_0}{2}$, then it holds that
$$
\frac{\mathcal L f_\varepsilon(x)}{f_\varepsilon(x)} \le \varepsilon T_0 - \frac{\varepsilon q_0}{2}|x - m|^2.
$$
It follows then that for $\delta > 0$,
\begin{align*}
   \frac{\mathcal L f_\varepsilon(x)}{f_\varepsilon(x)} + \delta|A(x - m)| &\le  \varepsilon T_0 - \frac{\varepsilon q_0}{2}|x - m|^2 + \delta \|A\|_{\op} |x - m| \\
   &\le \varepsilon T_0 + \frac{\delta^2\|A\|_{\op}^2}{2\varepsilon q_0}.
\end{align*}
If we take $\delta = \varepsilon$ (with $2\varepsilon s^2 \le \frac{q_0}{2}$), then the above inequality becomes
\begin{equation}\label{eq: Lyapunov bound}
    \frac{\mathcal L f_\varepsilon(x)}{f_\varepsilon(x)} + \delta|A(x - m)| \le \delta\bigg(T_0 + \frac{\|A\|^2_{\op}}{2q_0}\bigg).
\end{equation}
Set $c_\delta := \delta(T_0 + \frac{\|A\|^2_{\op}}{2q_0})$ and define $Z_t := e^{\delta V_t - c_\delta t}, t \ge 0$. By applying It\^o's formula to $Z_t f_\delta(X_t)$, we obtain
\begin{align*}
    \dd (Z_t f_\delta(X_t)) &= Z_tf_\delta(X_t) \dd (\delta V_t - c_\delta t) + Z_t \mathcal L f_\delta(X_t) \dd t + L_t \\
    &= Z_tf_\delta(X_t)\bigg(\frac{\mathcal L f_\delta(X_t)}{f_\varepsilon(X_t)} + \delta|A(X_t - m)| - c_\delta\bigg) \dd t + L_t,
\end{align*}
where $L_t, t \ge 0$ is a continuous local martingale. Hence, by $Z_tf_\delta(X_t) > 0$ and \eqref{eq: Lyapunov bound} we see that the process $(Z_t f_\delta(X_t))_{t \ge 0}$ is a nonnegative local supermartingale. As before, using a localizing argument, optional sampling theorem and Fatou's lemma, we arrive at
$$
\E[Z_\tau f_\delta(X_\tau)] = \E[ e^{\delta V_\tau - c_\delta \tau}f_\delta(X_\tau)] \le \E[f_\delta(X_0)] = K_0 < \infty,
$$
provided $\delta < \varepsilon_0$.
Since $f_\varepsilon(X_\tau) \ge 1$, we have $\E[ e^{\delta V_\tau - c_\delta \tau}] \le K_0 < \infty$.
Finally, using Cauchy-Schwarz, for any $\gamma > 0$,
\begin{align*}
    \E[e^{\frac{\delta}{2}V_\tau + \gamma K_\tau}] &=  \E[e^{\frac{\delta V_\tau - c_\delta \tau}{2}} e^{(\frac{c_\delta}{2} + \gamma \|\Sigma\|_{\HS}^2)\tau }] \\
    &\le \sqrt{\E[e^{\delta V_\tau - c_\delta \tau}]} \sqrt{\E[e^{(c_\delta + 2\gamma \|\Sigma\|_{\HS}^2)\tau }]}.
\end{align*}
So, by choosing $\gamma > 0$ small enough such that $c_\delta + 2\gamma \|\Sigma\|_{\HS}^2 \le \eta$, we get that
\begin{equation}\label{eq: bound for stopped OU}
 \E[e^{\frac{\delta}{2}V_\tau + \gamma K_\tau}] \le \sqrt{K_0}\E[e^{\eta \tau}] < \infty.
\end{equation}
As a result, any stopped OU process $X_{[0,\tau]}$ satisfying \eqref{eq: Hurwitz condition} and \eqref{eq: initial bound}, where $\tau$ is any stopping time such that $\E[e^{\eta \tau}]<\infty$ for some $\eta > 0$, meets the conditions in Theorem \ref{thm: stopped semimart} and Theorem \ref{thm: ROC of stopped semimart}. In particular, Corollary \ref{cor: determinacy for stopped semimart} remains valid for such $X_{[0,\tau]}$.

\section{Numerical experiments}\label{sect: numerical experiments}

In this section, we conduct numerical experiments on two-sample hypothesis testing for laws of stochastic processes using the radial derivatives of PCF and matrix-signature \eqref{eq: recursive formula for matrix signature} studied in previous sections, and illustrate the benefits of augmenting the PCF with its derivatives as complementary features for distinguishing laws on path space. We first describe the hypothesis testing problem and define distances induced by PCF between laws of stochastic processes in Subsection~\ref{subsect: numerical population distances}. Next, in Subsections~\ref{subsect: numerical paths}--\ref{subsect: numerical features}, we specify processes and models used in our experiments. The experiment procedure is introduced in Subsection~\ref{subsect: numerical testing}, and finally, experiment results are presented in Subsection~\ref{subsect: numerical results}. The code of the experiments is available at \url{https://github.com/DeepIntoStreams/Radial_Derivative_PCF}. The Monte-Carlo and Taylor approximation error bounds for the numerical computation of radial derivatives of PCF are recorded in Appendix~\ref{sect: numerical error bounds}.

\subsection{Two-sample hypothesis testing and distances between stochastic path laws}\label{subsect: numerical population distances}

Let $P$ and $Q$ be laws of two stochastic processes on a common path space $\cX$. Given independent samples from the two laws, we wish to test
\[
    H_0:P=Q \qquad\text{against}\qquad H_1:P\ne Q.
\]
Our starting point is that if $P=Q$, then every integrable function of a path will have the same expectation under the two laws. Therefore, to distinguish the two laws, it suffices to find a function whose expectations under the two laws differ. A natural example is the terminal unitary development in \eqref{eq: unitary dev}, whose expectation gives the path characteristic function (PCF) which distinguishes different path laws.

Motivated by this example, we consider a family of path features indexed by a parameter $a\in\Theta$, e.g., the linear map $M$ in \eqref{eq: unitary dev}. For each parameter $a$ and path $X$, we record several matrix-valued quantities of $X$ and collect them into a single feature map
\[
    F(X;a)=\bigl(F_1(X;a),\ldots,F_B(X;a)\bigr).
\]
In the rest, we refer to the $B$ components as \emph{blocks}, and for a path law $A\in\{P,Q\}$, we write $\mu_b^A(a)=\E_A[F_b(X;a)]$ for the mean of the $b$-th block.

For the sake of hypothesis testing, we compare the block means under $P$ and $Q$. To obtain a single measure of separation, we average the squared differences over the blocks and over $a$ with respect to a probability measure $\nu$ on $\Theta$. Assuming that the block means exist and are square-integrable with respect to $\nu$, we define the population distance, following the PCF construction in \cite{PCFGAN2023}, by
\begin{equation}\label{eq: numerical population distance}
    d_{F,\nu}(P,Q)^2
    =\frac{1}{B}\sum_{b=1}^{B}\int_{\Theta}
       \|\mu_b^P(a)-\mu_b^Q(a)\|_{\HS}^2\,\nu(\dd a).
\end{equation}
Here, the measure $\nu$ specifies how the discrepancies are weighted. We would like the distance in \eqref{eq: numerical population distance} to distinguish any two distinct path laws, that is, to satisfy $d_{F,\nu}(P,Q)>0$ whenever $P\ne Q$. However, this is not guaranteed by definition, so $d_{F,\nu}$ is in general only a pseudometric. In the aforementioned example, the terminal development in \eqref{eq: unitary dev} gives a case in which this distance is a true metric. To see this, consider time-augmented geometric $p$-rough paths in $\R^{d+1}$ with a common fixed initial point. Take $F$ to consist of the single block $Y_T^M$, and let $\nu$ have full support on $\Theta=\bigsqcup_{k\geq 1}\lin(\R^{d+1},\so(k))$. If $d_{F,\nu}(P,Q)=0$, then the two PCFs will agree $\nu$-almost everywhere. Consequently, Theorem~\ref{thm: PCF determines the law} gives equality of the signature laws, and Remark~\ref{rem: time-augmented paths} gives $P=Q$. 

In practice, to evaluate the population distance in \eqref{eq: numerical population distance}, we approximate $\nu$ by the discrete measure $\nu_K=\sum_{h=1}^K\delta_{a_h} / K$ where $a_1,\ldots,a_K \in \Theta$, and replace the block means by sample means. Specifically, let $X_1,\ldots,X_m$ and $Y_1,\ldots,Y_n$ be i.i.d. sample paths under $P$ and $Q$, respectively. Denote by $P_m$ and $Q_n$ the empirical laws. Then, the approximate population distance writes
\begin{equation}\label{eq: numerical empirical distance}
    d_{F,\nu_K}(P_m,Q_n)^2 =\frac{1}{BK}\sum_{b=1}^{B}\sum_{h=1}^{K} \bigl\|\overline F^P_{b,h}-\overline F^Q_{b,h}\bigr\|_{\HS}^2,
\end{equation}
where
\begin{equation*}
    \overline F^P_{b,h} = \frac{1}{m}\sum_{i=1}^m F_b(X_i;a_h),
    \qquad
    \overline F^Q_{b,h} = \frac{1}{n}\sum_{j=1}^n F_b(Y_j;a_h).
\end{equation*}

\subsection{Two Ornstein--Uhlenbeck laws and their path representation}\label{subsect: numerical paths}

In our experiments, we consider the following Ornstein--Uhlenbeck process:
\begin{equation}\label{eq: numerical OU}
    \dd X_t = ( - 2 I_2 + \omega R) X_t\,\dd t + 2\,\dd W_t, \qquad R = \begin{pmatrix}0 & - 1\\1 & 0\end{pmatrix}, \qquad X_0 \sim N(0,I_2),
\end{equation}
where $W$ is a standard two-dimensional Brownian motion independent of $X_0$. Here, $\omega$ is a model hyperparameter that determines the direction and speed of rotation.

Let $P$ and $Q$ be the laws of \eqref{eq: numerical OU} on $[0,1]$ with $\omega = 0.2$ and $\omega = - 0.2$, respectively. Notice that the two laws have the same marginal distribution $N(0,I_2)$ at every time, while having different temporal dependence. To see the later, we calculate
\begin{equation*}
    \forall 0\le s\le t\le1, \qquad \E[X_t X_s^\top] = e^{ - 2(t - s)}\bigl(\cos(\omega(t - s)) I_2 + \sin(\omega(t - s)) R\bigr).
\end{equation*}

For numerical purposes, denote by $N\ge1$ be the number of time steps on $[0,1]$. On the time grid $t_j=j\Delta$ ($j=0,\ldots,N$) with $\Delta=1/N$, we sample \eqref{eq: numerical OU} by its exact transition formula:
\begin{equation}\label{eq: numerical OU transition}
    X_{t_{j + 1}} = e^{ - 2\Delta} e^{\omega\Delta R} X_{t_j} + \sqrt{1 - e^{ - 4\Delta}}\,\xi_j, \qquad j=0,\ldots,N-1,
\end{equation}
where $\{\xi_j\}_{j=0}^{N-1}$ are independent $N(0,I_2)$ vectors and are independent of $X_0$. Then, we construct time-augmented polygonal paths with vertices
\begin{equation}\label{eq: numerical augmented path}
    Z_0=0, \qquad Z_{j+1}=\bigl(\tfrac{j+1}{N+1},X^1_{t_j},X^2_{t_j}\bigr), \qquad j=0,\ldots,N.
\end{equation}
Notice that the time coordinate in \eqref{eq: numerical augmented path} records the observation order, and that the initial segment from $Z_0$ to $Z_1$ encodes $X_0$. We use the time-augmented path $Z$ as the input of our models.

\subsection{Learned maps and matrix features}\label{subsect: numerical features}

In the notation of Subsection~\ref{subsect: numerical population distances}, each model in our experiments specifies a set $\Theta$ of admissible maps and scales and a choice of feature blocks $F_b(Z;a)$ for $a\in\Theta$. The PCFD, OPCFD and RPCFD models use the same terminal-development feature but different classes of maps, and the Jet, MS and finite-difference models share a common parametrization but differ in their feature blocks.

\paragraph*{Common construction of a head}

For a fixed matrix order $k\ge2$, a \emph{head} takes one input path $Z$ and returns a collection of $k\times k$ matrices, called feature blocks. It uses a linear map $M\in\lin(\R^3,\fu(k))$ and a scale $\lambda$. Write
\[
    M(v)=\sum_{i=1}^3v_iM_i,
    \qquad
    A_j=M(Z_j-Z_{j-1}).
\]
The unitary development \eqref{eq: unitary dev} along the polygonal path $Z$ is obtained by solving \eqref{eq: linear RDE} on each segment:
\begin{equation}\label{eq: numerical development}
    U_0=I_k,
    \qquad
    U_j=U_{j-1}\exp(\lambda A_j),
    \qquad j=1,\ldots,N+1.
\end{equation}
Its terminal value and expectation are
\begin{equation}\label{eq: numerical PCF}
    J_0=U(\lambda)=U_{N+1},
    \qquad
    \phi_{Z_{[0,1]}}(\lambda M)=\E[J_0].
\end{equation}
Note that $J_0$ is a feature of an individual path, and its expectation is the PCF. All the feature blocks below are obtained from this development and its dependence on $\lambda$.

\paragraph*{Admissible maps and scales}

The PCFD, OPCFD and RPCFD heads learn the coefficient matrices $M_i$ directly and take $\lambda=1$, so the magnitude of the development map is contained in the coefficients. No coefficient normalization or separate scale constraint is imposed on these heads. Their distinction is the admissible matrix class. PCFD uses complex skew-Hermitian matrices, $M_i^\dagger=-M_i$, as in \cite[Definition~3.3]{PCFGAN2023}. OPCFD uses real skew-symmetric matrices, $M_i^\top=-M_i$, as in \cite[Definition~6.1]{RestrictedPCF2024}. RPCFD further restricts these real skew-symmetric matrices to the first upper and lower off-diagonals, requiring $(M_i)_{ab}=0$ whenever $|a-b|\ne1$, as in \cite[Remark~6.2]{RestrictedPCF2024}.

The remaining six heads, Jet-1, Jet-2, MS-1, MS-2, MS-3 and FD-Jet-1, all use the same normalized skew-Hermitian parametrization and a bounded trainable scale. For trainable matrices $C_i\in\C^{k\times k}$ and a scalar $\rho\in\R$, set
\begin{equation}\label{eq: numerical map parametrization}
\begin{gathered}
    S_i=\frac{C_i-C_i^\dagger}{2},
    \qquad
    M_i=\frac{S_i}{
    \bigl(\sum_{\ell=1}^3\|S_\ell\|_{\HS}^2\bigr)^{1/2}},\\
    M(v)=\sum_{i=1}^3v_iM_i,
    \qquad
    \lambda=0.1+\frac{0.8}{1+e^{-\rho}}.
\end{gathered}
\end{equation}
We initialize each head with a nonzero skew-Hermitian coefficient stack. This parametrization separates the normalized coefficients, satisfying $\sum_{i=1}^3\|M_i\|_{\HS}^2=1$, from the scale $0.1<\lambda<0.9$.

\paragraph*{Feature blocks}

A PCFD, OPCFD or RPCFD head returns only the terminal-development block $J_0$. The difference between these three heads therefore lies entirely in their admissible maps.

A radial-derivative head retains $J_0$ and adjoins derivatives with respect to the scale. For a fixed path $Z$ and map $M$, define
\begin{equation}\label{eq: numerical jets}
    J_r=\frac{1}{r!}\partial_\lambda^rU(\lambda),
    \qquad r\ge1.
\end{equation}
A Jet-1 head returns $(J_0,J_1)$, while a Jet-2 head returns $(J_0,J_1,J_2)$. Here, the factor $1/r!$ is a normalization of the derivative order. With $J_{0,0}=I_k$ and $J_{r,0}=0$ for $r\ge1$, these blocks are computed by the exact recurrence
\begin{equation}\label{eq: numerical jet recurrence}
    J_{r,j}
    =\sum_{\ell=0}^r
    J_{r-\ell,j-1}\frac{A_j^\ell}{\ell!}\exp(\lambda A_j),
    \qquad
    J_r=J_{r,N+1}.
\end{equation}
Only one matrix exponential per segment is required for all retained orders. Since the vertices in \eqref{eq: numerical augmented path} are jointly Gaussian, the total variation of $Z$ has exponential moments. Corollary~\ref{cor: main result for BV} therefore identifies $\E[J_r]$ with the normalized radial derivative of the PCF in \eqref{eq: numerical PCF}.

The FD-Jet-1 head has the same two-block structure as Jet-1, but replaces $J_1$ by the centered finite difference
\begin{equation}\label{eq: numerical finite difference}
    \widetilde J_1
    =\frac{U(\lambda+0.1)-U(\lambda-0.1)}{0.2}.
\end{equation}
It therefore returns $(J_0,\widetilde J_1)$. Both terminal developments in \eqref{eq: numerical finite difference} use the same path and map $M$. The bound in \eqref{eq: numerical map parametrization} applies to the learned base scale $\lambda$, while the finite difference is evaluated at the shifted scales $\lambda\pm0.1$.

A matrix-signature head returns levels of the matrix signature of the moving-frame path \eqref{eq: hat X}. For the polygonal input $Z$, its increments are
\begin{equation}\label{eq: numerical moving frame}
    B_j=U_{j-1}A_jU_{j-1}^\dagger.
\end{equation}
Indeed, $A_j$ commutes with $\exp(u\lambda A_j)$ for every $u$, so the moving-frame path is linear on each segment. Its matrix-signature levels are computed from \eqref{eq: recursive formula for matrix signature} by
\begin{equation}\label{eq: numerical signature recurrence}
\begin{gathered}
    D_{r,j}
    =\sum_{\ell=0}^rD_{r-\ell,j-1}\frac{B_j^\ell}{\ell!},\\
    D_{0,0}=I_k,
    \qquad
    D_{r,0}=0,
    \qquad
    D_r=D_{r,N+1}, \qquad r \geq 1.
\end{gathered}
\end{equation}
An MS-$L$ head returns $(D_1,\ldots,D_L)$, with $L=1,2,3$ for MS-1, MS-2 and MS-3, respectively. By \eqref{eq: dev formula for BV paths} and \eqref{eq: numerical jets},
\begin{equation}\label{eq: numerical jet signature relation}
    D_r=J_rU(\lambda)^\dagger,
    \qquad r\ge1.
\end{equation}
Thus, for the same path, map and scale, each $D_r$ is obtained from $J_r$ by multiplication by the terminal factor $U(\lambda)^\dagger$. This factor is random under a path law, so $\E[D_r]$ generally differs from $\E[J_r]$. The MS heads retain these moving-frame levels without adjoining $J_0$, whereas the Jet heads retain the terminal development together with its radial derivatives.

\paragraph*{Models with multiple heads}

In our experiments, a model consists of $K\ge1$ copies of its corresponding head, all applied to the same input path. Head $h$ evaluates the features at $a_h=(M^{(h)},\lambda_h)$, where $\lambda_h=1$ for the direct parametrizations and is learned for the normalized models. The $K$ heads thus determine the discrete measure $\nu_K$ in \eqref{eq: numerical empirical distance}. Vectorizing all returned blocks and concatenating them over the heads gives the model's feature map, called a \emph{critic}. For complex blocks, real and imaginary parts are included as separate real coordinates. All of our nine models are summarized in Table~\ref{tab: numerical features}.

\begin{table}[ht]
\centering
\caption{Model construction for matrix order $k$ and $K$ heads. Each row specifies the coefficient class and the blocks returned by one head. The two groups use different map and scale parametrizations. Dimensions are ambient real dimensions after vectorization and concatenation over all heads, with real and imaginary parts separated for complex blocks.}
\label{tab: numerical features}
\begin{tabular}{llll}
\toprule
Model & Coefficient matrices & Blocks per head & Dimension\\
\midrule
\multicolumn{4}{l}{Direct coefficients and fixed $\lambda=1$}\\
PCFD & $\fu(k)$ & $[J_0]$ & $2Kk^2$\\
OPCFD & $\so(k)$ & $[J_0]$ & $Kk^2$\\
RPCFD & tridiagonal $\so(k)$ & $[J_0]$ & $Kk^2$\\
\midrule
\multicolumn{4}{l}{Normalized coefficients and trainable $0.1<\lambda<0.9$}\\
Jet-1 & $\fu(k)$ & $[J_0,J_1]$ & $4Kk^2$\\
Jet-2 & $\fu(k)$ & $[J_0,J_1,J_2]$ & $6Kk^2$\\
FD-Jet-1 & $\fu(k)$ & $[J_0,\widetilde J_1]$ & $4Kk^2$\\
MS-1 & $\fu(k)$ & $[D_1]$ & $2Kk^2$\\
MS-2 & $\fu(k)$ & $[D_1,D_2]$ & $4Kk^2$\\
MS-3 & $\fu(k)$ & $[D_1,D_2,D_3]$ & $6Kk^2$\\
\bottomrule
\end{tabular}
\end{table}

\begin{remark}
    The Jet-1 and Jet-2 models in Table~\ref{tab: numerical features} augment the terminal development $U(\lambda)$ with its normalized radial derivatives $J_r=(r!)^{-1}\partial_\lambda^rU(\lambda)$, defined in \eqref{eq: numerical jets}. For a map and scale fixed independently of the sampled paths, empirical averages of these derivative blocks over i.i.d. paths estimate the corresponding normalized radial derivatives of the PCF. Under the corresponding $L^2$ bounds, Lemma~\ref{lemma: error bounds}(1) in Appendix~\ref{sect: numerical error bounds} controls their sampling error. The experiments below investigate the discriminative performance of these derivative features in two-sample testing.
\end{remark}

\subsection{Fitting and held-out two-sample tests}\label{subsect: numerical testing}

We compare the two OU laws $P$ and $Q$ defined in Subsection~\ref{subsect: numerical paths}. All nine models use $N=32$ time steps, matrix order $k=5$ and $K=8$ heads. For each model, the experiment consists of fitting a critic, selecting its parameters on a separate validation sample, and evaluating the selected critic on held-out paths.

\paragraph*{Data}

We generate three mutually independent samples from each law: $10{,}000$ training paths, $1{,}024$ validation paths and a pool of $10{,}000$ test paths. The training sample size counts distinct trajectories, each observed on the full time grid, and the nine models share these samples.

\paragraph*{Fitting and validation}

Let $\theta=(\theta_1,\ldots,\theta_K)$ collect the trainable model coefficients, where $\theta_h$ denotes the coefficients of head $h$. Through the parametrization in Subsection~\ref{subsect: numerical features}, $\theta_h$ determines the location $a_h(\theta_h)=(M^{(h)},\lambda_h)$: the direct models learn the map with $\lambda_h=1$, while the normalized models learn both the map and the scale. Thus the $K$ locations of $\nu_K$ are learned as part of fitting $\theta$. We write $F_{b,h}^{\theta}(X)=F_b(X;a_h(\theta_h))$ for block $b$ of head $h$, with $B$ blocks per head as specified in Table~\ref{tab: numerical features}. For a batch of $m$ paths from law $A\in\{P,Q\}$, write
\begin{equation}\label{eq: numerical minibatch moments}
    \overline F^A_{b,h}
    =\frac{1}{m}\sum_{i=1}^mF_{b,h}^{\theta}(X_i^A).
\end{equation}
The fitting objective is the squared empirical distance in \eqref{eq: numerical empirical distance}, evaluated on the two training batches at the current head parameters:
\begin{equation}\label{eq: numerical training score}
    S(\theta)
    =\frac{1}{BK}\sum_{b=1}^B\sum_{h=1}^K
    \|\overline F^P_{b,h}-\overline F^Q_{b,h}\|_{\HS}^2.
\end{equation}
We maximize \eqref{eq: numerical training score} using the Adam optimizer \cite[Algorithm~1]{KingmaBa2015}. At each update, we randomly select $m=1{,}024$ distinct paths from the training set of each law.

After each update, we evaluate the current critic on the full validation samples. Let $v=1{,}024$, and let $\overline F^{A,\mathrm{val}}_{b,h}$ denote the block means defined as in \eqref{eq: numerical minibatch moments}, using the $v$ validation paths from law $A\in\{P,Q\}$. Define their sample variances by
\[
    \widehat\sigma_{A,b,h}^{\,2}
    =\frac{1}{v-1}\sum_{i=1}^v
    \left\|
    F_{b,h}^{\theta}(X_i^{A,\mathrm{val}})
    -\overline F^{A,\mathrm{val}}_{b,h}
    \right\|_{\HS}^2.
\]
For a fixed critic, subtracting the estimated sampling variances of the two means gives an unbiased estimate of $d_{F,\nu_K}(P,Q)^2$, the squared population distance at the current heads:
\begin{equation}\label{eq: numerical validation distance}
    \widehat d_{\rm u}^{\,2}
    =\frac{1}{BK}\sum_{b=1}^B\sum_{h=1}^K
    \left[
    \left\|
    \overline F^{P,\mathrm{val}}_{b,h}
    -\overline F^{Q,\mathrm{val}}_{b,h}
    \right\|_{\HS}^2
    -\frac{\widehat\sigma_{P,b,h}^{\,2}
    +\widehat\sigma_{Q,b,h}^{\,2}}{v}
    \right].
\end{equation}
We select the final iterate using the validation score
\begin{equation}\label{eq: numerical validation score}
    V(\theta)
    =\frac{\widehat d_{\rm u}^{\,2}}
    {\widehat{\rm se}_{\rm reg}},
\end{equation}
where $\widehat{\rm se}$ is the plug-in estimate of the first-order asymptotic standard error under a nondegenerate alternative, see \cite[Section~6, p.~740]{Gretton2012}.

\paragraph*{Held-out two-sample tests}

For each fitted critic, we perform $20$ test trials using the held-out pools. In each trial, we draw one group of $n=200$ distinct paths from $P$ and two disjoint groups of $n=200$ paths from $Q$. Comparing the $P$ group with the first $Q$ group measures model power, while comparing the two $Q$ groups measures type I error. For either comparison, denote the two groups by $X_1,\ldots,X_n$ and $Y_1,\ldots,Y_n$. With the selected heads now fixed, the observed statistic is the squared empirical distance in \eqref{eq: numerical empirical distance} between these two groups:
\begin{equation}\label{eq: numerical permutation test}
    T_{\mathrm{obs}}
    =\frac{1}{BK}\sum_{b=1}^B\sum_{h=1}^K
    \left\|
    \frac{1}{n}\sum_{i=1}^nF_{b,h}^{\theta}(X_i)
    -
    \frac{1}{n}\sum_{i=1}^nF_{b,h}^{\theta}(Y_i)
    \right\|_{\HS}^2.
\end{equation}
Following \cite[Algorithm~1]{RestrictedPCF2024}, we pool the $2n$ paths and compute $500$ permutation statistics by randomly relabelling them into two groups of size $n$. We reject equality of laws when $T_{\mathrm{obs}}$ is strictly greater than the empirical $95\%$ quantile of these permutation statistics. 

\begin{remark}
    The training objective and the test statistic both measure the squared distance between empirical feature means. During training, we maximise this distance to learn features that separate the two laws. During testing, we compare the observed distance with its permutation distribution to determine whether the separation is statistically significant. The validation score serves a different purpose: it compares candidate iterates before the final test. We use a bias-corrected estimate of the squared mean separation divided by its estimated standard error, which accounts for sampling variability and serves as a proxy for test power. 
\end{remark}

\paragraph*{Repetitions}

The complete procedure above, consisting of data generation, fitting, validation and the $20$ test trials for each model, forms one repetition of the experiment. We run five repetitions, using fresh seeds for data generation, initialization, optimization and testing in each repetition.

\subsection{Experiment results}\label{subsect: numerical results}

For each fitted critic, its empirical model power is the proportion of the $20$ $P/Q$ trials that reject equality of laws, and its empirical type I error is the corresponding proportion for the $20$ $Q/Q$ trials. Table~\ref{tab: numerical OU powers} reports the mean and sample standard deviation of each rejection frequency across the five complete repetitions. 

\begin{table}[ht]
\centering
\caption{Power and type I error for the OU laws with $\omega=\pm0.2$, $N=32$, $k=5$ and $K=8$. Entries are percentages, reported as mean $\pm$ sample standard deviation across five complete repetitions. Each fitted critic is evaluated in $20$ $P/Q$ and $20$ $Q/Q$ tests, with $200$ paths per group and $500$ permutations per test. The nominal level is $5\%$.}
\label{tab: numerical OU powers}
\begin{tabular}{lcrr}
\toprule
Model & Learning rate & Power (\%) & Type I error (\%)\\
\midrule
PCFD & $0.05$ & $36.0\pm12.4$ & $8.0\pm4.5$\\
OPCFD & $0.05$ & $32.0\pm16.8$ & $7.0\pm4.5$\\
RPCFD & $0.5$ & $33.0\pm11.0$ & $6.0\pm6.5$\\
\midrule
Jet-1 & $0.05$ & $45.0\pm17.0$ & $5.0\pm3.5$\\
Jet-2 & $0.05$ & $48.0\pm13.5$ & $5.0\pm3.5$\\
FD-Jet-1 & $0.05$ & $44.0\pm15.2$ & $4.0\pm4.2$\\
MS-1 & $0.05$ & $48.0\pm10.4$ & $4.0\pm2.2$\\
MS-2 & $0.05$ & $60.0\pm9.4$ & $4.0\pm2.2$\\
MS-3 & $0.05$ & $62.0\pm10.4$ & $3.0\pm4.5$\\
\bottomrule
\end{tabular}
\end{table}

\begin{figure}[t]
\centering
\includegraphics[width=\linewidth]{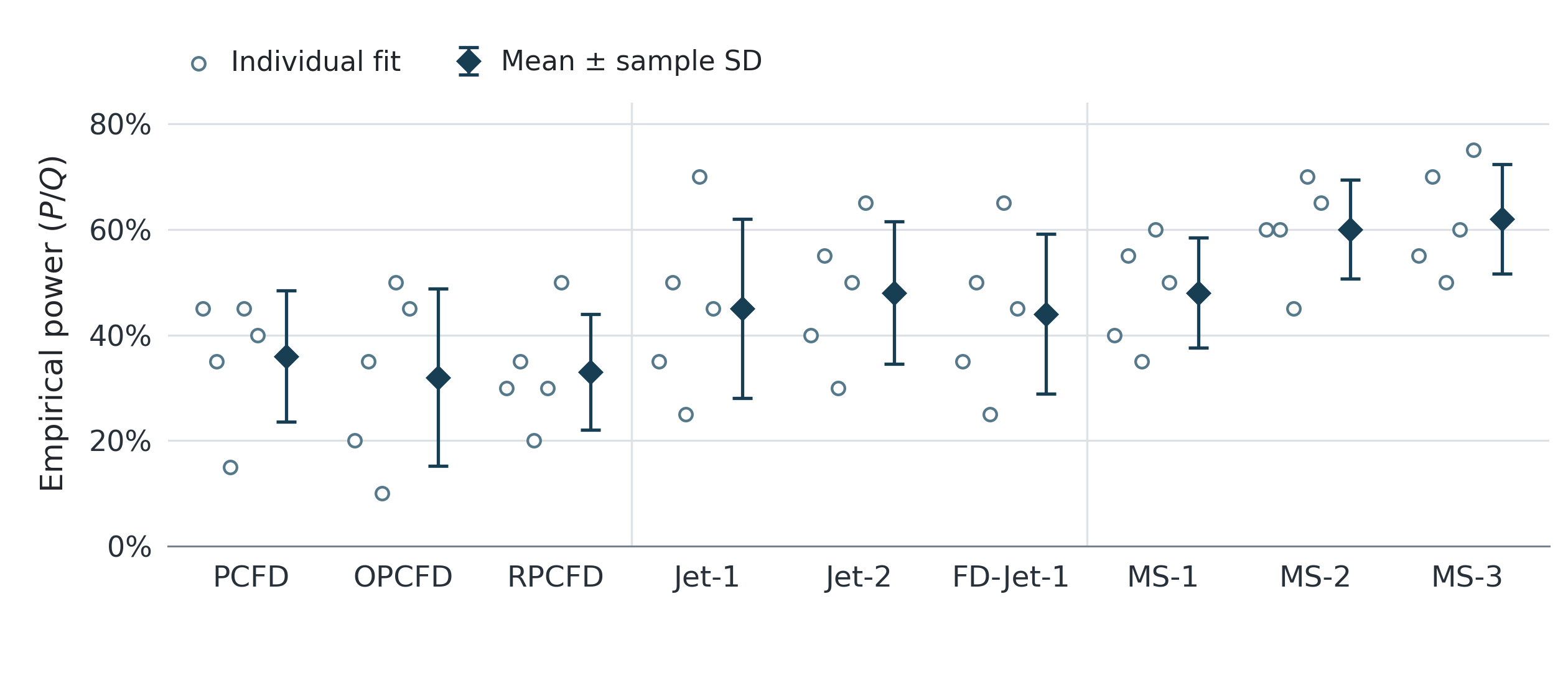}
\caption{Empirical power for the OU laws with $\omega=\pm0.2$, $k=5$ and $K=8$. Small points show the rejection frequencies for the five fitted critics of each model. Large points and bars show their mean and sample standard deviation. Each individual estimate uses $20$ $P/Q$ tests with $200$ paths per group and $500$ permutations per test.}
\label{fig: numerical OU powers}
\end{figure}

\begin{figure}[t]
\centering
\includegraphics[width=\linewidth]{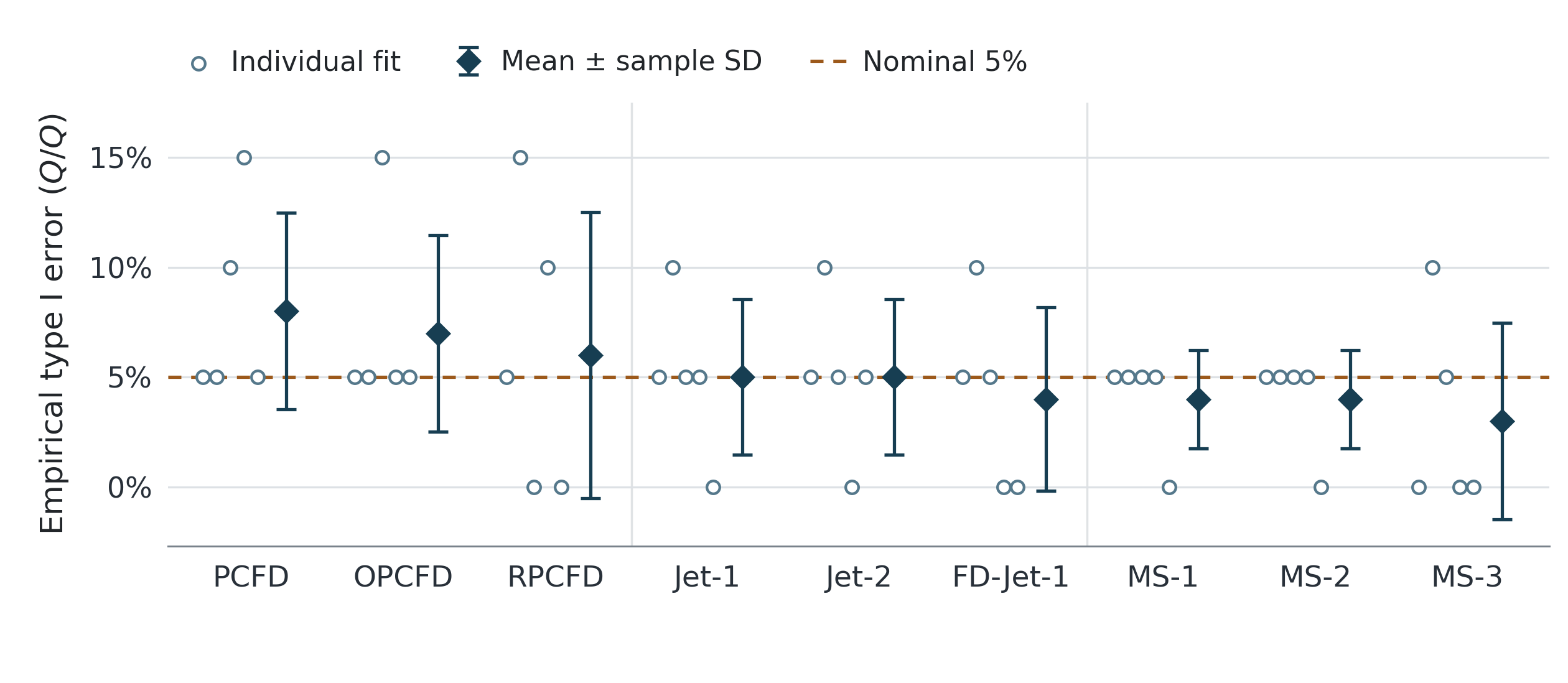}
\caption{Empirical type I error from the $Q/Q$ comparisons. Small points show the rejection frequencies for the five fitted critics of each model. Large points and bars show their mean and sample standard deviation. Each individual estimate uses $20$ tests. The dashed line indicates the nominal level $5\%$.}
\label{fig: numerical OU null}
\end{figure}

As shown, MS-2 and MS-3 attain the highest mean powers, which are $60\%$ and $62\%$, respectively. The largest mean power among the PCF baselines is $36\%$ for PCFD, followed by $33\%$ for RPCFD and $32\%$ for OPCFD, as plotted in Figure~\ref{fig: numerical OU powers}.

Within the group sharing normalized maps and bounded scales, the mean powers of MS-1, MS-2 and MS-3 are $48\%$, $60\%$ and $62\%$, respectively. MS-2 exceeds MS-1 in all five repetitions. The additional increase from MS-2 to MS-3 is smaller, and MS-3 exceeds MS-2 in three of the five repetitions. Jet-1 and Jet-2 attain mean powers of $45\%$ and $48\%$, respectively. FD-Jet-1 has mean power $44\%$, close to the $45\%$ of Jet-1, and the two methods have the same empirical power in four repetitions.

The mean type I errors range from $3\%$ to $8\%$, with $4\%$ for MS-2 and $3\%$ for MS-3. Figure~\ref{fig: numerical OU null} displays the individual-fit estimates relative to the nominal level $5\%$. These individual estimates range from $0\%$ to $15\%$ and provide a coarse assessment of null rejection rates, given the $20$ trials per fit.

\appendix
\section{Stochastic Calculus of Matrix-valued Continuous Semimartingales} \label{sect: matrix stochastic calculus}

Unless otherwise stated, all processes in this section are $\R^{k \times k}$-valued continuous semimartingales in the sense that each $(i,j)$-entry of them are $\R$-valued continuous semimartingale for all $i,j = 1,\ldots,k$.

\paragraph{Stochastic Integration and Quadratic Variation of Matrix-valued Semimartingales.}

For two continuous semimartingales $A = (A_t)_{t \in [0,T]}$ and $B = (B_t)_{t \in [0,T]}$, their Stratonovich integral $\int_0^t A_s \circ \dd B_s$ is another $k \times k$-matrix valued continuous semimartingale whose $(i,j)$-entry is given by
$$
\bigg(\int_0^t A_s \circ \dd B_s \bigg)_{ij} = \sum_{\ell = 1}^k \int_0^t A^{i\ell}_s \circ \dd B^{\ell j}_s
$$
for all $i,j = 1,\ldots,k$. The It\^o integral $\int_0^t A_s \dd B_s$ will be interpreted in the same way by replacing those $\R$-valued Stratonovich integrals $\int_0^t A^{i\ell}_s \circ \dd B^{\ell j}_s$ through It\^o integrals $\int_0^t A^{i\ell}_s \dd B^{\ell j}_s$. Similarly, the Stratonovich integral $\int_0^t \circ \dd A_s  B_s$ means the above matrix-valued integral of $B$ from the right-hand side against $A$, that is, the $(i,j)$-entry of $\int_0^t \circ \dd A_s  B_s$ is given by
$$
\bigg(\int_0^t \circ \dd A_s  B_s \bigg)_{ij} = \sum_{\ell = 1}^k \int_0^t B^{\ell j}_s \circ dA^{i \ell}_s.
$$
The same calculation also holds for the It\^o integral $\int_0^t (\dd A_s) B_s$. It is clear that using the above definition we have the following product rule for matrix-valued continuous semimartingales:
\begin{equation*}
    A_tB_t = A_0B_0 + \int_0^t A_s \circ \dd B_s + \int_0^t \circ \dd A_s B_s.
\end{equation*}
Furthermore, suppose that $C = (C_t)_{t \in [0,T]}$ is another continuous semimartingale, then
  thanks to the associativity law of $\R$-valued Stratonovich integral, it follows that
$\int_0^t \circ \dd(\int A_s \circ \dd B_s)  C_s = \int_0^t A_s \circ \dd(\int \circ \dd B_s  C_s)$.

We will use $[A,B]_\cdot$ to denote the $k \times k$-matrix-valued continuous process of bounded variation whose $(i,j)$-entry is the sum of the scalar quadratic covariations associated with the products $A^{i\ell}B^{\ell j}$. More precisely,
$$
[A,B]_t = ([A,B]_t^{ij})_{i,j = 1,\ldots,k} = \bigg(\sum_{\ell = 1}^k [A^{i\ell}, B^{\ell j}]_t\bigg)_{i,j = 1,\ldots,k},
$$
where $[A^{i\ell}, B^{\ell j}]$ is the usual quadratic covariation process between $\R$-valued semimartingales $A^{i\ell}$ and $ B^{\ell j}$. Therefore we will call $[A,B]_\cdot$ the quadratic covariation of $A$ and $B$. Note that the usual It\^o-Stratonovich correction still holds for matrix-valued continuous semimartingales:
\begin{equation}\label{eq: ito-strat correction for matrix}
    \int_0^t A_s \circ \dd B_s = \int_0^t A_s \dd B_s + \frac{1}{2}[A,B]_t, \quad \int_0^t \circ \dd B_s   A_s = \int_0^t (dB_s) A_s + \frac{1}{2}[B,A]_t
\end{equation}
Moreover, the usual computation rule for the quadratic variation of a stochastic integral with another continuous semimartingale remains true in the matrix case: if $A,B,$ and $C$ are $k \times k$-matrix valued continuous semimartingales, then
\begin{equation}\label{eq: associative law quadratic variation of stochastic integral}
    \bigg[\int_0^\cdot A_s \circ \dd B_s, C\bigg]_t = \bigg[\int_0^\cdot A_s \dd B_s, C\bigg]_t = \int_0^t A_s \dd[B,C]_s.
\end{equation}
However, in general we do not have $[A,B] = [B,A]$. On the other hand, given a continuous matrix-valued local martingale $A$, we use
$$
\llbracket A, A \rrbracket_t : = \sum_{i,j = 1}^k [A^{ij}, A^{ij}]_t
$$
to denote the compensator of the $\R$-valued local submartingale $(\|A_t - A_0\|^2_{\HS})_{t \in [0,T]}$, i.e., $\llbracket A, A \rrbracket_\cdot$ is the unique continuous bounded variation process (starting from $0$) such that $(\|A_t - A_0\|^2_{\HS} - \llbracket A, A \rrbracket_t)_{t \in [0,T]}$ is a continuous local martingale.

\paragraph{It\^o Isometry for the Matrix-valued Stochastic Integrals.}

Now we consider a continuous square integrable matrix- valued local martingale $B$ and a progressively measurable matrix-valued process $A$ which is integrable with respect to $B$ in the sense that for all $i,j,\ell = 1,\ldots,k$, $A^{i\ell}$ is $B^{\ell j}$-integrable. Assume in addition that $\E[\int_0^T \text{Tr}(A_s\dd[B,B^\dagger]_sA_s^\dagger)] < \infty$. Recall that the It\^o integral $\int_0^t A_s \dd B_s$ is a matrix-valued continuous local martingale such that
$$
\bigg(\int_0^t A_s dB_s\bigg)_{ij} = \sum_{\ell = 1}^k \int_0^t A^{i\ell}_s dB^{\ell j}_s.
$$
It is easy to see that for all $t \in [0,T]$, as a consequence of the classical It\^o isometry, it holds that
\begin{align}\label{eq: matrix Ito isometry}
    \E\bigg[\bigg\|\int_0^t A_s \dd B_s\bigg\|_{\HS}^2\bigg]
     = \E\bigg[\sum_{i,j = 1}^k \sum_{\ell,\ell^\prime = 1}^k \int_0^t A^{i\ell}_s A^{i\ell^\prime}_s \dd[B^{\ell j},B^{\ell^\prime j}]_s
    \bigg]
    = \E\bigg[ \int_0^t \text{Tr}(A_s \dd[B,B^\dagger]_s A^\dagger_s)      \bigg]
\end{align}
which is the It\^o isometry for matrix-valued stochastic integral. Also note that the entrywise quadratic-covariation rule also implies that the compensator $\llbracket \int_0^\cdot A_s \dd B_s,  \int_0^\cdot A_s \dd B_s \rrbracket_t, t \in [0,T]$ satisfies that
 \begin{align}\label{eq: predictable compensator}
     \bigg\llbracket \int_0^\cdot A_s \dd B_s,  \int_0^\cdot A_s \dd B_s \bigg\rrbracket_t & = \sum_{i,j = 1}^k \bigg[ \bigg(\int_0^\cdot A_s \dd B_s\bigg)^{ij},  \bigg(\int_0^\cdot A_s \dd B_s\bigg)^{ij}\bigg]_t \nonumber\\
     & = \int_0^t \text{Tr}(A_s \dd[B,B^\dagger]_s A^\dagger_s).
 \end{align}

\paragraph{Matrix valued semimartingale as linearly transformed $\R^d$-valued semimartingales}
Let $X$ be a continuous $\R^d$-valued semimartingale defined on some filtered probability space $(\Omega, \bbF, \bbP)$ and $M \in \lin(\R^d, \so(k))$ be a given linear operator from $\R^d$ to $\so(k)$ for some $k \ge 1$. Then clearly the process $(M(X_t))_{t \in [0,T]}$ is a $\so(k)$-valued continuous semimartingale. Thanks to the linearity of $M$, we indeed have $\int_0^\cdot A_t \circ \dd (M(X)_t) = \int_0^\cdot A_t M(\circ \dd X_t) = \sum_{i = 1}^d \int_0^\cdot A_t M(e_i) \circ \dd X^i_t$.

For a given linear operator $M \in \lin(\R^d, \so(k))$, we use $\|M\|_{\op}$ to denote the operator norm of $M$ relative to the Hilbert-Schmidt norm on $\so(k)$ and the usual Euclidean norm on $\R^d$, that is,
$$
\|M\|_{\op} = \sup_{x \in \R^d; |x| \le 1} \|M(x)\|_{\HS} = \sup_{x \in \R^d; \|x\|_2 \le 1} \bigg(
\sum_{i,j = 1}^k |M_{ij}(x)|^2 \bigg)^{\frac{1}{2}}.
$$
\begin{lemma}\label{lemma: bound on the quadratic variation of M(X)}
Assume that $X$ is an $\R^d$-valued continuous local martingale and $M \in \lin(\R^d, \so(k))$. Then we have for all $t \in [0,T]$,
    $$
   \llbracket M(X), M(X) \rrbracket_t = \sum_{i,j = 1}^k [M_{ij}(X), M_{ij}(X)]_t \le \|M\|_{\op}^2 \llbracket X, X \rrbracket_t,
    $$
where $\llbracket X, X \rrbracket_t = \sum_{i = 1}^d [X^i,X^i]_t$.
\end{lemma}
\begin{proof}
By the definition of the quadratic variation, for any sequence of deterministic partitions $\cP_n, n \ge 1$ of $[0,t]$ whose mesh tends to zero, it holds that for any $i,j = 1,\ldots,k$,
$$
[M_{ij}(X), M_{ij}(X)]_t = \lim_{n \to \infty} \sum_{[s,u] \in \cP_n} |M_{ij}(X_u) - M_{ij}(X_s)|^2,
$$
where the limit is taken with respect to the convergence in probability. Note that for every $n$, we have
\begin{align*}
    \sum_{i,j = 1}^k \sum_{[s,u] \in \cP_n} |M_{ij}(X_u) - M_{ij}(X_s)|^2 & \le \|M\|_{\op}^2 \sum_{[s,u] \in \cP_n} \|X_u - X_s\|_2^2 \\
    & = \|M\|_{\op}^2 \sum_{[s,u] \in \cP_n} \sum_{i = 1}^d|X^i_u - X^i_s|^2.
\end{align*}
Hence, letting $n \to \infty$ and taking the limit in the convergence of probability, we get that
$$
\sum_{i,j = 1}^k [M_{ij}(X), M_{ij}(X)]_t \le \|M\|_{\op}^2 \sum_{i = 1}^d [X^i, X^i]_t,
$$
as claimed.
\end{proof}

\section{Proofs of Some Auxiliary Results}\label{sect: proof of auxiliary results}

\begin{lemma}\label{lem: bound on quad variation of hat X when X is BM}
Let $X = (X_t)_{t \in [0,T]}$ be an $\R^d$-valued Brownian motion, and $M \in \lin(\R^d, \so(k))$ be a given linear operator from $\R^d$ to $\so(k)$ for some $k \ge 1$. Then there exists a constant $C = C(k,d) \ge \sqrt{d}$ such that for every $\lambda \in \R$, $t \in [0,T]$ and any progressively measurable $k \times k$-matrix valued process $A = (A_t)_{t \in [0,T]}$ satisfying $\E[\int_0^T \|A_s\|_{\HS}^2\dd s] < \infty$, one has
$$
\int_0^t\text{Tr}(A_s\dd[\hat{X}^\lambda, (\hat{X}^\lambda)^\dagger]_sA_s^\dagger) \le C\|M\|_{\op}^2\int_0^t \|A_s\|_{\HS}^2 \dd s,
$$
and
$$
\bigg\|\int_0^t A_s \dd[\hat X^\lambda,\hat X^\lambda]_s   \bigg\|_{L^2} \le C\|M\|_{\op}^2 \int_0^t \|A_s\|_{L^2} \dd s.
$$
\end{lemma}

\begin{proof}
Let $e_1, \ldots,e_d$ be the canonical basis for $\R^d$, then for a given $M \in \lin(\R^d, \so(k))$, we have $M(X_t) = X^1_tM(e_1) + \ldots + X^d_t M(e_d)$. Since $X$ is an $\R^d$-valued Brownian motion (so that $[X^i, X^j]_t = \delta_{ij}t$), it is easy to see that for any pairs $(p,q)$ and $(p^\prime,q^\prime)$ in $\{1,\ldots,k\}^2$,
$$
[M_{pq}(X), M_{p^\prime q^\prime}(X)]_t = \bigg(\sum_{r = 1}^d M_{pq}(e_r)M_{p^\prime q^\prime}(e_r)\bigg)t.
$$
Further, since $\hat X^\lambda_t = \int_0^t Y^\lambda_s M(\dd X_s)(Y^\lambda_s)^\dagger$, for any pair $(\ell,\ell^\prime) \in \{1,\ldots,k\}^2$, we obtain that
\begin{align*}
    [\hat{X}^\lambda, (\hat{X}^\lambda)^\dagger]^{\ell \ell^\prime}_t = \int_0^t \sum_{j,p,q,p^\prime,q^\prime = 1}^k Y^{\lambda, \ell p}_sY^{\lambda, jq}_s Y^{\lambda, \ell^\prime p^\prime}_sY^{\lambda, jq^\prime}_s \sum_{r = 1}^d M_{pq}(e_r)M_{p^\prime q^\prime}(e_r) \dd s.
\end{align*}
Let $\mathcal C^{\ell \ell^\prime}_{M}(s)$ denote the integrand inside the last integral above, so that $$\dd [\hat{X}^\lambda, (\hat{X}^\lambda)^\dagger]^{\ell \ell^\prime}_t = \mathcal C^{\ell \ell^\prime}_{M}(t) \dd t
$$ holds. Then since for all $\lambda \in \R$ and $s \in [0,T]$ we have $Y^\lambda_s \in O(k)$ is orthogonal, it follows that for all $(i,j) \in \{1,\ldots,k\}^2$, $|Y^{\lambda, ij}_s| \le 1$. Since also $|M_{pq}(e_r)| \le \|M\|_{\op}$ for all $p,q = 1,\ldots,k$ and $r = 1,\ldots,d$, we indeed get that for all $s \in [0,T]$,
\begin{align*}
    |\mathcal C^{\ell \ell^\prime}_M(s)| & \le \sum_{j,p,q,p^\prime,q^\prime = 1}^k |Y^{\lambda, \ell p}_s||Y^{\lambda, jq}_s||Y^{\lambda, \ell^\prime p^\prime}_s||Y^{\lambda, jq^\prime}_s|\sum_{r = 1}^d |M_{pq}(e_r)||M_{p^\prime q^\prime}(e_r)| \le dk^5\|M\|_{\op}^2.
\end{align*}
Thanks to the above bound, we can further deduce that
\begin{align*}
   \int_0^t\text{Tr}(A_s\dd[\hat{X}^\lambda, (\hat{X}^\lambda)^\dagger]_sA_s^\dagger) = \sum_{i = 1}^k \sum_{\ell,\ell^\prime = 1}^k \int_0^t A^{i\ell}_sA^{i\ell^\prime}_s  \mathcal C_M^{\ell \ell^\prime}(s) \dd s
\end{align*}
and therefore
\begin{align*}
  \int_0^t\text{Tr}(A_s\dd[\hat{X}^\lambda, (\hat{X}^\lambda)^\dagger]_sA_s^\dagger)  \le dk^6\|M\|_{\op}^2 \int_0^t \|A_s\|_{\HS}^2 \dd s,
\end{align*}
where we applied the Cauchy-Schwarz inequality to deduce $\sum_{i,\ell,\ell^\prime = 1}^k |A^{i\ell}_s||A^{i\ell^\prime}_s| \le k\|A_s\|_{\HS}^2$ in the above last inequality. As a result, we may take $C_1 = C_1(k,d) = dk^6$, which satisfies our requirement for the first inequality.

For the second inequality, note that by the same argument as above, we can derive that there exists a progressively measurable $k \times k$-matrix valued process $\mathcal D_M(t), t \in [0,T]$ such that $\dd[\hat X^\lambda, \hat X^\lambda]_t = \mathcal D_M(t) \dd t$ and $|\mathcal D^{\ell \ell^\prime}_M(t)| \le dk^5\|M\|_{\op}^2$ for all $t$ and all $\ell,\ell^\prime = 1,\ldots,k$. It follows that
$$
\int_0^t A_s \dd[\hat X^\lambda,\hat X^\lambda]_s = \int_0^t A_s \mathcal D_M(s) \dd s
$$
and therefore by the Minkowski inequality,
$$
\bigg\|\int_0^t A_s \dd[\hat X^\lambda,\hat X^\lambda]_s   \bigg\|_{L^2} \le \int_0^t \|A_s \mathcal D_M(s)\|_{L^2} \dd s.
$$
Note that for each $s \in [0,t]$, one has (using Cauchy-Schwarz inequality again)
\begin{align*}
 \|A_s  \mathcal D_M(s)\|^2_{L^2}  = \E[\|A_s  \mathcal D_M(s)\|_{\HS}^2     ]
  &\le (dk^5\|M\|_{\op}^2)^2 k^2 \E\bigg[ \sum_{i,j,\ell = 1}^k |A^{i\ell}_s|^2 \bigg] \\
  &\le (dk^5\|M\|_{\op}^2)^2 k^3 \E[\|A_s\|_{\HS}^2],
\end{align*}
which implies that
$$
\bigg\|\int_0^t A_s \dd[\hat X^\lambda,\hat X^\lambda]_s   \bigg\|_{L^2} \le \int_0^t \|A_s \mathcal D_M(s)\|_{L^2} \dd s \le dk^{5 + \frac{3}{2}}\|M\|_{\op}^2 \int_0^t \|A_s\|_{L^2} \dd s.
$$
As a result, we may take $C_2 = C_2(k,d) = dk^{5 + \frac{3}{2}}$ so that the second inequality holds. Since $C_2 \ge C_1$, finally we can choose $C = C_2$ so that both inequalities are satisfied with $C$.
\end{proof}

\section{Error bounds for Monte Carlo and Taylor approximation}\label{sect: numerical error bounds}

Let $\bX_{[0,T]}$ be a random geometric $p$-rough path with $p \ge 1$ on a possibly random time horizon $T>0$. Fix $M \in \bigsqcup_{k = 1}^\infty \lin(\R^d, \fu(k))$ and suppose that a sequence of positive numbers $(B(n))_{n \ge 0}$ satisfies
\begin{equation}\label{eq: dev bounds in 4.1}
    \sup_{\lambda \in \R} \frac{1}{n!}\|Y^{\lambda,(n)}_T\|_{L^2} \le B(n) \quad  \text{or} \quad  \sup_{\lambda \in \bK} \frac{1}{n!}\|Y^{\lambda,(n)}_T\|_{L^2} \le B(n)
\end{equation}
for all $n \ge 0$, where $Y^{\lambda}_T = Y^{\lambda M}_T$ is the unitary development \eqref{eq: unitary dev} with respect to $\lambda M$ along $X_{[0,T]}$, and
$\bK$ is a compact subset of $\R$. The derivative formulas and bounded-variation bounds used here extend to $\fu(k)$, with the same proofs and with unitarity replacing orthogonality. Assume that differentiation commutes with expectation, as under the hypotheses of Theorem~\ref{thm: random_func}, and write
\begin{equation}\label{eq: numerical Taylor coefficient}
    a_n(\lambda) = \frac{1}{n!}\E[Y^{\lambda,(n)}_T]
    = \frac{1}{n!}\frac{\dd^n}{\dd\lambda^n}\phi_{\bX_{[0,T]}}(\lambda M), \qquad n\ge 0.
\end{equation}

For $m \ge 1$, take i.i.d. copies $(\bX^j_{[0,T_j]},T_j)$, $j = 1,\ldots,m$, of $(\bX_{[0,T]},T)$. For each copy, let $Y^{\lambda,j}$ solve \eqref{eq: linear RDE} driven by $\bX^j_{[0,T_j]}$. For $n \ge 0$, write $D_n(\hat \bX^{\lambda,j}_t)$, $t \in [0,T_j]$, for the $n$-th matrix-signature level of the moving-frame path \eqref{eq: hat X}, defined recursively by \eqref{eq: recursive formula for matrix signature}.
By the derivative formula in Theorem~\ref{thm: properties of dev for rough paths}(3), the random variables $Y^{\lambda,(n)}_j = n!D_n(\hat \bX^{\lambda,j}_{T_j})Y^{\lambda,j}_{T_j}$ ($j = 1,\ldots,m$) are i.i.d. copies of the $n$-th radial derivative $Y^{\lambda,(n)}_T$. For $h\in\R$ and $N\ge 1$, the Monte-Carlo estimator of \eqref{eq: numerical Taylor coefficient} and its associated Taylor surrogate are
\begin{equation}\label{eq: numerical Taylor estimators}
    \hat a_{n,m}(\lambda) := \frac{1}{m}\sum_{j=1}^m \frac{Y^{\lambda,(n)}_j}{n!},
    \qquad
    \hat\phi_{N,m}(\lambda+h) := \sum_{n=0}^N h^n \hat a_{n,m}(\lambda).
\end{equation}
In the following, Lemma~\ref{lemma: error bounds} uses \eqref{eq: dev bounds in 4.1} to control the Monte Carlo and truncation errors in \eqref{eq: numerical Taylor estimators}.

\begin{lemma}\label{lemma: error bounds}
   Assume that for a given $M \in \bigsqcup_{k = 1}^\infty \lin(\R^d, \fu(k))$, there exists a sequence of positive numbers $(B(n))_{n \ge 0}$ such that the bounds in \eqref{eq: dev bounds in 4.1} hold for all $n \ge 0$. Let $I \in \{\R, \bK\}$ denote the corresponding domain, then
\begin{enumerate}
    \item for all $n \ge 0$ and $m \ge 1$,

    \begin{equation}\label{eq: numerical Monte Carlo bound}
        \sup_{\lambda\in I}\bigl\|\hat a_{n,m}(\lambda)-a_n(\lambda)\bigr\|_{L^2}
        \le \frac{B(n)}{\sqrt{m}};
    \end{equation}
    
    \item if $r>0$ satisfies
    \[
        \sum_{n=0}^{\infty}B(n)r^n<\infty,
    \]
    then, for all $N\ge1$, $m\ge1$ and $h\in\R$ with $|h|\le r$,

    \begin{equation}\label{eq: numerical Taylor bound}
        \sup_{\lambda\in I}\bigl\|\hat \phi_{N,m}(\lambda+h)-\phi_{\bX_{[0,T]}}((\lambda+h)M)\bigr\|_{L^2}
        \le \frac{1}{\sqrt{m}}\sum_{n=0}^N B(n)r^n+\sum_{n=N+1}^\infty B(n)r^n.
    \end{equation}
\end{enumerate}
\end{lemma}

\begin{proof}
For a square-integrable matrix-valued random variable $Z$, write
\[
    \mathrm{Var}(Z):=\E\bigl[\|Z-\E[Z]\|_{\HS}^2\bigr].
\]
Independence makes the centered summands in \eqref{eq: numerical Taylor estimators} orthogonal in $L^2$ under the Hilbert--Schmidt inner product. Hence
\begin{align*}
    \E\bigg[\bigg\| \hat a_{n,m}(\lambda) - \frac{\E[Y^{\lambda,(n)}_T]}{n!}  \bigg\|_{\HS}^2       \bigg] &= \E\bigg[\bigg\| \frac{1}{m} \sum_{j=1}^m \frac{Y^{\lambda,(n)}_j - \E[Y^{\lambda,(n)}_T]}{n!}  \bigg\|_{\HS}^2       \bigg]\\
    &=\frac{1}{m} \mathrm{Var}\bigg(\frac{Y^{\lambda,(n)}_T}{n!}\bigg) \le \frac{1}{m} \E\bigg[\bigg\| \frac{Y^{\lambda,(n)}_T}{n!}  \bigg\|_{\HS}^2       \bigg],
\end{align*}
which gives \eqref{eq: numerical Monte Carlo bound} upon applying \eqref{eq: dev bounds in 4.1}.

For part (2), fix $\lambda\in I$ and $h\in\R$ with $|h|\le r$. For each sample path, the terminal development is an entire function of the scale. Moreover,
\[
    \sum_{n=0}^{\infty}|h|^n
    \left\|\frac{Y_T^{\lambda,(n)}}{n!}\right\|_{L^2}
    \le \sum_{n=0}^{\infty}B(n)r^n<\infty.
\]
Hence its Taylor series converges in $L^2$, and taking expectations gives
\[
    \phi_{\bX_{[0,T]}}((\lambda+h)M)
    =\sum_{n=0}^{\infty}h^n a_n(\lambda).
\]
For the surrogate in \eqref{eq: numerical Taylor estimators}, adding and subtracting the Taylor polynomial with coefficients \eqref{eq: numerical Taylor coefficient} gives
\begin{align*}
    \|\hat \phi_{N,m}(\lambda + h) - \phi((\lambda + h)M)\|_{L^2} &\le \underbrace{\bigg\| \sum_{n=0}^N \hat a_{n,m}(\lambda)h^n - \sum_{n=0}^N  h^n a_n(\lambda)   \bigg\|_{L^2}}_{\text{Monte-Carlo approximation error}} \\
    &+ \underbrace{\bigg\|  \sum_{n=0}^N  h^n a_n(\lambda)   - \phi_{\bX_{[0,T]}}((\lambda + h)M) \bigg\|_{L^2}}_{\text{truncation error}}\\
    &\le \sum_{n=0}^N  \bigg\| \hat a_{n,m}(\lambda) - \frac{\E[Y^{\lambda,(n)}_T]}{n!} \bigg\|_{L^2}r^n + \sum_{n>N} \frac{\|Y^{\lambda,(n)}_T\|_{L^2}}{n!}r^n.
\end{align*}
Applying \eqref{eq: numerical Monte Carlo bound} to the first sum and \eqref{eq: dev bounds in 4.1} to the second gives \eqref{eq: numerical Taylor bound}.
\end{proof}

\textbf{Acknowledgements.} HN is supported in part by the EPSRC Program Grant [Grant No. UKRI1010] entitled ``High order mathematical and computational infrastructure for streamed data that enhance contemporary generative and large language models'' and by the SURE-AI Centre grant 357482, Research Council of Norway. HN is grateful for useful discussions with Fabian Andsem Harang and Kurusch Ebrahimi-Fard. CL is partially supported by the National Key R\&D Program of China (No. 2023YFA1010900).

\textbf{Declaration of generative AI use}. The authors used OpenAI's ChatGPT to assist in exploring proof strategies and producing preliminary drafts of parts of several proofs. All AI-assisted arguments were carefully checked, revised where necessary, and independently verified by the authors, who take full responsibility for the correctness and integrity of the manuscript.

\bibliography{quellen}
\bibliographystyle{amsalpha}

\end{document}